\documentclass[runningheads,envcountsame,letterpaper]{llncs}
\newif\iffull
\fulltrue
\usepackage{etex}
\usepackage[utf8]{inputenc}
\usepackage{mathtools}

\usepackage{amsmath,amssymb,tabularx,hhline,rotating,enumerate,xspace}

\usepackage{tikz}
\usepackage{multirow}

\usepackage{todonotes}

\usepackage[hidelinks]{hyperref}

\usepackage[capitalise]{cleveref}

\usepackage{adjustbox}
\usepackage{xparse}

\makeatletter
\NewDocumentCommand{\raisedminus}{m}{%
	\raisebox{0.1em}{$\m@th#1{-}$}%
}
\NewDocumentCommand{\unaryminus}{}{%
	\mathbin{%
		\mathchoice{%
			\raisedminus\scriptstyle
		}{%
			\raisedminus\scriptstyle
		}{%
			\raisedminus\scriptscriptstyle
		}{%
			\raisedminus\scriptscriptstyle
		}%
	}%
}
\makeatother

\makeatletter
\newcommand\footnoteref[1]{\protected@xdef\@thefnmark{\ref{#1}}\@footnotemark}
\makeatother

\newcommand{\OLDcompproblem}[3][]{%
	\par\vspace{0.125cm plus 0.1cm minus 0.05cm}\begin{tabularx}{\linewidth-2\parindent}{lX}%
		\if\relax\detokenize{#1}\relax%
		\else%
		\textnormal{\textsf{Constant:}}&#1\\%
		\fi%
		\textnormal{\textsf{Input:}}&#2\\%
		\textnormal{$\mathrlap{\textsf{Output:}}\hphantom{\textsf{Question:}}$}&#3\\%
	\end{tabularx}\vspace{0.125cm plus 0.1cm minus 0.05cm}\par%
}

\newcommand{\OLDproblem}[3][]{%
	\par\vspace{0.125cm plus 0.1cm minus 0.05cm}\begin{tabularx}{\linewidth-2\parindent}{lX}%
		\if\relax\detokenize{#1}\relax%
		\else%
		\textnormal{\textsf{Constant:}}&#1\\%
		\fi%
		\textnormal{\textsf{Input:}}&#2\\%
		\textnormal{\textsf{Question:}}&#3\\%
	\end{tabularx}\vspace{0.125cm plus 0.1cm minus 0.05cm}\par%
}

\iffull
\newcommand{\compproblem}[3][]{%
	\par\vspace{0.125cm plus 0.1cm minus 0.05cm}\adjustbox{valign=t}{\begin{tabularx}{\linewidth-2\parindent}{@{}lX}%
			\if\relax\detokenize{#1}\relax%
			\else%
			\textnormal{\textsf{Constant:}}&#1\\%
			\fi%
			\textnormal{\textsf{Input:}}&#2\\%
			\textnormal{$\mathrlap{\textsf{Output:}}\hphantom{\textsf{Question:}}$\!\!}&#3\\%
	\end{tabularx}}\vspace{0.125cm plus 0.1cm minus 0.05cm}\par%
}

\newcommand{\ynproblem}[3][]{%
	\par\vspace{0.125cm plus 0.1cm minus 0.05cm}\adjustbox{valign=t}{\begin{tabularx}{\linewidth-2\parindent}{@{}lX}%
			\if\relax\detokenize{#1}\relax%
			\else%
			\textnormal{\textsf{Constant:}}&#1\\%
			\fi%
			\textnormal{\textsf{Input:}}&#2\\%
			\textnormal{\textsf{Question:}}&#3\\%
	\end{tabularx}}\vspace{0.125cm plus 0.1cm minus 0.05cm}\par%
}

\else
\newcommand{\compproblem}[3][]{%
	\par\vspace{0.1cm plus 0.1cm minus 0.05cm}\adjustbox{valign=t}{\begin{tabularx}{\linewidth-1.5\parindent}{@{}lX}%
			\if\relax\detokenize{#1}\relax%
			\else%
			\textnormal{\textsf{Constant:}}&#1\\%
			\fi%
			\textnormal{\textsf{Input:}}&#2\\%
			\textnormal{$\mathrlap{\textsf{Output:}}\hphantom{\textsf{Question:}}$\!\!\!}&#3\\%
	\end{tabularx}}\vspace{0.1cm plus 0.1cm minus 0.05cm}\par%
}

\newcommand{\multicompproblem}[9][]{%
	\par\vspace{0.1cm plus 0.1cm minus 0.05cm}\adjustbox{valign=t}{\begin{tabularx}{\linewidth-1.5\parindent}{@{}lX}%
			\if\relax\detokenize{#1}\relax%
			\else%
			\textnormal{\textsf{Constant:}}&#1\\%
			\fi%
			\textnormal{\textsf{Input:}}&#2\\%
			\textnormal{$\mathrlap{\textsf{Output A:}}\hphantom{\textsf{Question:}}$\!}&#3\\%
			\textnormal{$\mathrlap{\textsf{Output B:}}\hphantom{\textsf{Question:}}$\!}&#4\\%
			\textnormal{$\mathrlap{\textsf{Output C:}}\hphantom{\textsf{Question:}}$\!}&#5\\%
			\textnormal{$\mathrlap{\textsf{Output D:}}\hphantom{\textsf{Question:}}$\!}&#6\\%
			\textnormal{$\mathrlap{\textsf{Output E:}}\hphantom{\textsf{Question:}}$\!}&#7\\%
			\textnormal{$\mathrlap{\textsf{Output F:}}\hphantom{\textsf{Question:}}$\!}&#8\\%
			\textnormal{$\mathrlap{\textsf{Question G:}}\hphantom{\textsf{Question:  }}$\!}&#9\\%
			\end{tabularx}}\vspace{0.1cm plus 0.1cm minus 0.05cm}\par%
}

\newcommand{\ynproblem}[3][]{%
	\par\vspace{0.1cm plus 0.1cm minus 0.05cm}\adjustbox{valign=t}{\begin{tabularx}{\linewidth-1.5\parindent}{@{}lX}%
			\if\relax\detokenize{#1}\relax%
			\else%
			\textnormal{\textsf{Constant:}}&#1\\%
			\fi%
			\textnormal{\textsf{Input:}}&#2\\%
			\textnormal{\textsf{Question:}}&#3\\%
	\end{tabularx}}\vspace{0.1cm plus 0.1cm minus 0.05cm}\par%
}
\fi

\newenvironment{cm}{\noindent\color{blue}}{} 

\newcommand{\taligned}{$t$-aligned\xspace}
\newcommand{\Taligned}{$T$-aligned\xspace}
\newcommand{\wal}{\tilde{w}}
\newcommand{\wcl}[2]{w_{[#1,#2]}}
\newcommand{\wopi}[2]{w_{(#1,#2]}}

\newcommand{\wopij}[2]{w_{(#1,#2)}}
\newcommand{\kcl}[2]{k_{[#1,#2]}}

\newcommand{\height}[2]{h_{#2}(#1)}
\newcommand{\hght}[1]{h({#1})}
\newcommand{\mn}[2]{m_{#1#2}}
\newcommand{\mx}[2]{M_{#1#2}}

\newenvironment{aw}{\noindent\color{magenta} AW :  }{}

\newcommand{\Delpr}{\Delta^*}

\newcommand{\Thepr}{\Theta^*}

\newcommand{\simBG}{\sim_{\BGpr}}

\newcommand{\len}[2]{\ell_{#1}(#2)}

\newcommand{\base}{q}
\newcommand{\binM}[2]{\operatorname{digit}_{#2}(#1)}

\newcommand{\cQsdr}{compact base-$\base$ signed-digit representation\xspace}
\newcommand{\sdr}{signed-digit representation\xspace}
 \newcommand{\stepone}{\proc{UpdateNodes}}
\newcommand{\steptwo}{\proc{ExtendChains}}

\newcommand{\addition}{\proc{Addition}}
\newcommand{\powertwo}{\proc{MultByPower}}
\newcommand{\FloatPoint}{\proc{MaxQPower}}
\newcommand{\comp}{\proc{Comparison}}
\newcommand{\OpLem}{\proc{operation lemmas}}

\newcommand{\sgn}[1]{\operatorname{sgn}(#1)}
\newcommand{\cy}{e}

\newcommand{\redpc}{red-PC rep.\xspace}
\newcommand{\Redpc}{Red-PC rep.\xspace}
\newcommand{\redpcs}{redPC's reps. }
\newcommand\up[2]{\operatorname{up}_{#1}(#2)}
\newcommand\low[2]{\operatorname{low}_{#1}(#2)}
\newcommand{\set}[2]{\left\{#1\mathrel{\left|\vphantom{#1}\vphantom{#2}\right.}#2\right\}}

\newcommand{\oneset}[1]{\left\{\mathinner{#1}\right\}}

\newcommand{\msumkt}[1]{\operatorname{S}_{K,T}(#1)}
\newcommand{\msumk}[1]{\operatorname{S}_{K}(#1)}

\newcommand{\abs}[1]{\left\lvert\mathinner{#1}\right\rvert}

\newcommand{\abssmall}[1]{\lvert\mathinner{#1}\rvert}

\newcommand{\floor}[1]{\left\lfloor\mathinner{#1} \right\rfloor}
\newcommand{\ceil}[1]{\left\lceil\mathinner{#1} \right\rceil}

\newcommand{\gen}[1]{\left< \mathinner{#1} \right>}
\newcommand{\Gen}[2]{\left< \mathinner{#1} \mid \mathinner{#2}\right>}

\newcommand{\genbig}[1]{\big< \mathinner{#1} \big>}

\newcommand{\N}{\ensuremath{\mathbb{N}}}
\newcommand{\Z}{\ensuremath{\mathbb{Z}}}
\newcommand{\Q}{\ensuremath{\mathbb{Q}}}
\newcommand{\R}{\ensuremath{\mathbb{R}}}

\newcommand{\TC}{\ensuremath{\mathsf{TC}^0}\xspace}
\newcommand{\NC}{\ensuremath{\mathsf{NC}}\xspace}

\renewcommand{\phi}{\varphi}
\newcommand{\eps}{\varepsilon}
\renewcommand{\epsilon}{\varepsilon}
\newcommand{\depth}{\operatorname{depth}}

\newcommand{\val}{\operatorname{val}_q}

\newcommand{\cor}{\operatorname{CR}_\base}

\newcommand{\e}{\eps} % Used for evaluations, so we can change VD.

\newcommand{\mOne}{\xspace{\unaryminus\!1}}
\newcommand{\supp}{\sigma}

\newcommand{\del}{\delta}

\newcommand{\Sig}{\Sigma}

\newcommand{\Del}{\Delta}

\newcommand\GG{\Gamma}
\newcommand\LL{\Lambda}

\newcommand{\Oh}{\mathcal{O}}
\newcommand{\cA}{\mathcal{A}}

\newcommand{\cC}{\mathcal{C}}
\newcommand{\cU}{\mathcal{U}}

\newcommand{\cV}{\mathcal{V}}
\newcommand{\cW}{\mathcal{W}}

\newcommand{\cZ}{\mathcal{Z}}
\newcommand{\BS}[2]{\ensuremath{\mathrm{\mathbf{BS}}_{#1,#2}}}
\newcommand{\BG}{\ensuremath{\mathrm{\mathbf{G}}_{1,2}}\xspace} %{\mathop{\mathrm{\bf{BG}}}({1},{2})

\newcommand{\BGq}{\ensuremath{\mathrm{\mathbf{G}}_{1,\pcBase}}\xspace} %{\mathop{\mathrm{\bf{BG}}}({1},{2})
\newcommand{\BSq}{\ensuremath{\mathrm{\mathbf{BS}}_{1,\pcBase}}}
\newcommand{\inBSpr}{\in_{\text{\tiny $\mathbf{B}$}}}
\newcommand{\eqBSpr}{=_{\text{\tiny $\mathbf{B}$}}}
\newcommand{\BGpr}{\ensuremath{\mathrm{\mathbf{G}}_{p,r}}}
\newcommand{\BGpq}{\ensuremath{\mathrm{\mathbf{G}}_{p,pq}}}
\newcommand{\BSpq}{\ensuremath{\mathrm{\mathbf{BS}}_{p,pq}}}
\newcommand{\numba}[1]{\mathcal{A}(#1)}
\newcommand{\eqBGpr}{=_{\BGpr}}
\newcommand{\inBGpr}{\in_{\BGpr}}
\newcommand{\eqBGpq}{=_{\BGpq}}
\newcommand{\inBGpq}{\in_{\BGpq}}

\newcommand{\oi}[1]{{#1}^{-1}}

\renewcommand{\L}{\LOGSPACE}
\newcommand{\NP}{\ensuremath{\mathsf{NP}}\xspace} %
\newcommand{\coNP}{\ensuremath{\mathsf{coNP}}\xspace}
\newcommand{\LOGSPACE}{\ensuremath{\mathsf{LOGSPACE}}\xspace} %

\newcommand{\SAC}{\ensuremath{\mathsf{SAC}^1}\xspace} %
\newcommand{\Tc}[1]{\ensuremath{\mathsf{TC}^{#1}}\xspace}
\newcommand{\Ac}[1]{\ensuremath{\mathsf{AC}^{#1}}\xspace}
\newcommand{\Nc}[1]{\ensuremath{\mathsf{NC}^{#1}}\xspace}

\renewcommand{\P}{\ensuremath{\mathsf{P}}\xspace}
\newcommand{\uAC}[1]{\ensuremath{\mathsf{uAC}^{#1}}\xspace}
\newcommand{\uAc}[1]{\ensuremath{\mathsf{uAC}^{#1}}\xspace}
\newcommand{\uTC}[1]{\ensuremath{\mathsf{uTC}^{#1}}\xspace}
\newcommand{\uAcf}[1]{\ensuremath{\mathsf{uAC}^{#1}(F_2)}\xspace}
\newcommand{\PTc}[2]{\ensuremath{\mathsf{LinDepParaTC}^{#1}}\xspace}

\newcommand{\DLOGTIME}{\ensuremath{\mathsf{DLOGTIME}}\xspace} %

\newcommand\tow{\mathop \tau} %\mathrm{tow}}

\newcommand{\Breduced}{Britton-re\-du\-ced\xspace}

\newcommand{\PC}{power circuit\xspace}

\newcommand{\pcBase}{q}
\newcommand{\domQ}{\interval{-\pcBase+1}{\pcBase-1}}
\newcommand{\sdZ}{\ensuremath{\Z[1/\pcBase] \rtimes \Z}}

\newcommand\lds{,\ldots ,} 
\newcommand{\sse}{\subseteq}
\newcommand{\es}{\emptyset}

\newcommand{\interval}[2]{[ \mathinner{#1}..\mathinner{#2}] }

\newcommand\ei[1]{{\emph{#1}\xspace}\index{#1}} 

\newcommand{\proc}[1]{\ensuremath{\text{\textsc{#1}}}\xspace}

\newcommand{\gnot}{\proc{Not}}
\newcommand{\gor}{\proc{Or}}
\newcommand{\gmaj}{\proc{Majority}}
\newcommand{\gand}{\proc{And}}

\newcommand\ie{i.e., }
\newcommand\eg{e.g.\xspace}

\newcommand{\comparator}{\mathrel{\triangle}}
\newcommand{\compOpSet}{\oneset{ = , \neq, <,\leq ,>, \geq }}

\begin{document}

\title{Improved Parallel Algorithms for Baumslag Groups \thanks{This work has been partially supported by DFG Grant WE 6835/1-2.}}

\titlerunning{Improved Parallel Algorithms for Baumslag Groups}
\author{Caroline Mattes\inst{1} \and
	Armin Weiß\inst{1}\orcidID{0000-0002-7645-5867}}
\authorrunning{C. Mattes \and A. Weiß}
\institute{
	Universität Stuttgart, Institut für Formale Methoden der Informatik, Germany
%	\email{lncs@springer.com}\\
%	\url{http://www.springer.com/gp/computer-science/lncs} \and
%	ABC Institute, Rupert-Karls-University Heidelberg, Heidelberg, Germany\\
%	\email{\{abc,lncs\}@uni-heidelberg.de}
}
\maketitle              % typeset the header of the contribution
\begin{abstract}
	The Baumslag group had been a candidate for a group with an extremely difficult word problem until Myasnikov, Ushakov, and Won succeeded to show that its word problem can be solved in polynomial time. Their result used the newly developed data structure of power circuits allowing for a non-elementary compression of integers. Later this was extended in two directions: Laun showed that the same applies to the Baumslag groups $\BGq$ for $q \geq 2$ and we established that the word problem of the Baumslag group $\BG$ can be solved in \Tc2.
	
	In this work we refine the operations on reduced power circuits to further improve upon both results. We show that the word problem of the Baumslag groups $\BGpq$ with $\abs{p},\abs{q} \geq 1$ can be solved in $\uTC{1}$. 
	Moreover, we prove that the conjugacy problem in $\BGpq$ is strongly generically in $\uTC{1}$ (meaning that for ``most'' inputs it is in $\uTC{1}$). Finally, for every fixed $g \in \BGq$ (case $p=1$) conjugacy to $g$ can be decided in $\uTC{1}$ for \emph{all} inputs.
	
	We further show that the word problem of the Baumslag-Solitar groups $\BS{p}{pq}$ is in $\uAcf{0}$ if the input word is given in a quite compressed form and so give a complexity result for a special case of the power word problem for these groups.

\end{abstract} 
%

%\author{Caroline Mattes}{Universität Stuttgart, Institut für Formale Methoden der Informatik (FMI), Universitätsstraße 38,
%	70569 Stuttgart, Germany}{caroline.mattes@fmi.uni-stuttgart.de}{}{}%TODO mandatory, please use full name; only 1 author per \author macro; first two parameters are mandatory, other parameters can be empty. Please provide at least the name of the affiliation and the country. The full address is optional
%
%\author{Armin Weiß}{Universität Stuttgart, Institut für Formale Methoden der Informatik (FMI), Universitätsstraße 38,
%	70569 Stuttgart, Germany}{armin.weiss@fmi.uni-stuttgart.de}{https://orcid.org/0000-0002-7645-5867}{}
%Funded by DFG project DI 435/7-1.

\section{Introduction}

In the early 20th century, Dehn \cite{dehn11} introduced the {\em word problem} as one of the basic algorithmic problems in group theory: given a word over the generators of a group $G$, the questions is whether this word represents the identity of $G$. 
Already in the 1950s, Novikov and Boone constructed finitely presented groups with an undecidable word problem \cite{boone59,nov55}. Still, many natural classes of groups have an (efficiently) decidable word problem~-- most prominently, the class of linear groups (groups embeddable into a matrix group over some field): their word problem is in \LOGSPACE \cite{lz77,Sim79}~-- in particular, in \NC, \ie decidable by Boolean circuits of polynomial size and polylogarithmic depth (or, equivalently decidable in polylogarithmic time using polynomially many processors). 
There are several other results on word problems of groups in small complexity classes defined by circuits, for example
for solvable linear groups in \TC (constant depth with majority gates) \cite{KonigL18poly}, for Baumslag-Solitar groups in \LOGSPACE \cite{Weiss16}, and for hyperbolic groups in $\SAC\sse \Nc2$ \cite{Lo05ijfcs}.
%Moreover, Thompson's groups are co-context-free \cite{LehnertS07} and hyperbolic groups have word problem in \LOGCFL \cite{Lo05ijfcs}. All these complexity classes are contained within \NC.
%
Nevertheless, there are also finitely presented groups with decidable, yet arbitrarily hard, word problems \cite{brs02}. % but with arbitrarily high complexity \cite{brs02}. %\cite[Theorem 1.3]{brs02}

A \emph{one-relator group} is a group that can be written as a free group modulo a normal subgroup generated by a single element (\emph{relator}). A famous algorithm called the \emph{Magnus breakdown procedure} \cite{mag32} shows that one-relator groups have decidable word problems (see also \cite{LS01,mks}).
%Magnus \cite{mag32} showed that one-relator groups have a decidable word problem; his algorithm is called the Magnus breakdown procedure (see also \cite{LS01,mks}). 
Its complexity remains an open problem: while it is not even clear whether the word problems of one-relator groups are solvable in elementary time, \cite{BMS} asks for polynomial-time algorithms.

In 1969 Gilbert Baumslag defined the group
$\BG = \Gen{a,b}{bab^{-1} a = a^2bab^{-1}}$
as an example of a one-relator group enjoying certain remarkable properties. It is infinite and non-abelian, but all its finite quotients are cyclic and, thus, it is not residually finite
 \cite{baumslag69}. 

Moreover, Gersten showed that the Dehn function of $\BG$ is non-elementary \cite{gersten91} meaning that it cannot be bounded by any fixed tower of exponentials (see also \cite{plat04}). This made the Baumslag group a candidate for a group with a very difficult word problem. 
Indeed, when applying the Magnus breakdown procedure to an input word of length $n$, one obtains as intermediate results words of the form $v_1^{x_1} \cdots v_m^{x_m}$ where $v_i \in \{a,b, bab^{-1} \}$, $x_i \in \Z$, and $m \leq n$. The issue is that the $x_i$ might grow up to $\tau_2(\log n)$ (with $\tow_2(0) = 1$ and $\tow_2(i+1) = 2^{\tow_2(i)}$ for $i \geq 0$~-- the tower function). %; hence, this algorithm has non-elementary running time. 
However, Myasnikov, Ushakov and Won succeeded to show that the word problem of \BG is, indeed, decidable in polynomial time \cite{muw11bg}. Their crucial contribution were so-called \emph{power circuits} in \cite{MyasnikovUW12} for compressing the $x_i$ in the above description.
% A particular interest here is the computation of so-called \emph{compact} markings, which allow for a unique representation of integers.

Roughly speaking, a (base-$2$) \emph{power circuit} is a directed acyclic graph with edges labelled by numbers from $\{-1, 0, 1\}$. 
One defines an evaluation of a vertex $P$ as two raised to the power of the (weighted) sum of the successors of $P$. 
Hence, the value $\tow_2(n)$ can be represented by an $n + 1$-vertex power circuit~-- thus, power circuits allow for a non-elementary compression. The crucial feature for the application to the Baumslag group is that they not only efficiently support the operations $+$, $-$, and $(x,y) \mapsto x \cdot 2^y$, but also the test whether $x = y$ or $x < y$ for two integers represented by power circuits can be done in polynomial time. The main technical part of the comparison algorithm is to compute a so-called \emph{reduced} power circuit.

Based on these striking results, Diekert, Laun and Ushakov \cite{DiekertLU13ijacstacs} improved 
the running time for the word problem of the Baumslag group from $\Oh(n^7)$ down to $\Oh(n^3)$ and described a polynomial-time algorithm for the word problem of the Higman group $H_4$ \cite{higman51}.
Subsequently, more applications of power circuits to similar groups emerged: 
In \cite{Laun14} Laun gave a polynomial-time solution for the word problem of the Baumslag groups $\BGq = \Gen{a,b}{bab^{-1} a = a^\pcBase bab^{-1}}$ for $q \geq 1$ and also for generalized Higman groups. In order to do so, he generalized power circuits to arbitrary bases $\pcBase \geq 2$ and adapted the corresponding algorithms from \cite{MyasnikovUW12,DiekertLU13ijacstacs}. 
Finally, in \cite{MattesWCC23}, the authors of the present paper studied the word problem of the Baumslag group $\BG$ from the point of view of parallel complexity and showed that it can be solved in the circuit class \Tc2.

% The proof consists of two main steps: first, to show that for a power circuit of logarithmic depth a corresponding reduced power circuit can be computed in \Tc1 (in contrast to the general case where computing reduced power circuits is \P-complete \cite[Theorem C]{MattesWCC23}) and, second, to show that the Magnus breakdown procedure can be performed in a tree-shape manner leading to a logarithmic number of rounds with each individual round doable in \Tc1.

An interesting variant of the word problem is the power word problem:  the question is the same as for the word problem, but with the input given as $p_1^{x_1}\cdots p_n^{x_n}$ where $x_i \in \Z$ and the $p_i$ are words over the group generators \cite{LohreyZetsche23,LohreyWeiss19}. Note that the input is given in a quite compressed form here, so the power word problem might be harder to solve as it is the case e.g. for $S_5 \wr \Z$ under the hypothesis that $\L \neq \NP$. The word problem of this group is in \LOGSPACE, but its power word problem is \coNP-complete \cite{LohreyWeiss19}. On the other hand, there are groups such that the word problem and the power word problem are in the same complexity class:  for the Baumslag-Solitar groups $\BSq= \Gen{a,t}{tat^{-1} = a^q}$ it is shown that the power word problem \cite{LohreyZetsche23} as well as the word problem \cite{Robinson93} both are in the same (small parallel) complexity class $\TC$. 
\smallskip

Another one of Dehns fundamental problems is the conjugacy problem: Given elements $g,h \in G$ does there exist an element $z \in G$ such that $z^{-1}gz=h$? 
 In \cite{DiekertMW2016SI} the conjugacy problem of the Baumslag group is shown to be strongly generically in \P and in \cite{Baker20} the same is done for the conjugacy problem of the Higman group. Here ``generically'' roughly means that the algorithm works for most inputs~-- for a precise definition, see  \cref{sec:complexity} below. The idea is that often the ``generic-case behavior'' of an algorithm is more relevant than its average-case or worst-case behavior. We refer to \cite{KMSS1,KMSS2} where the foundations of this theory were developed  
 and to \cite{MyasnikovSU08} for applications in cryptography.

\vspace{-1mm}
\paragraph*{Contribution.}
In this work we consider the parallel complexity of the word problem of the  Baumslag group $\BGpq = \Gen{a,b}{bab^{-1} a^p = a^{pq} bab^{-1}}$ for $\abs{p},\abs{q}\geq 1$. Moreover, we extend the methods used in \cite{Weiss16} to consider a special case of the power word problem of the Baumslag-Solitar group $\BS{p}{pq}=\Gen{a,t}{ta^pt^{-1} = a^{pq}}$: if the input is given as a word over the powers of the \emph{generators}. This form of input previously has been studied in \cite{MyasnikovW22,gurevichs07} (though with binary exponents instead of exponents encoded by power circuits, as we do). Let $a,t$ be the generators of $\BS{p}{pq}$. We show the following:
\newtheorem{propositionA}{Proposition}
\renewcommand*{\thepropositionA}{\Alph{propositionA}}

\newtheorem{theoremA}[propositionA]{Theorem}
\begin{theoremA}\label{thm:WPinBS}
	For every $p,q \in \Z$ with $\abs{p},  \abs{q} \geq 1$ the word problem of $\BS{p}{pq}$ with the input word given over $a^k, t^x$  with $k,x \in \Z$ represented by markings in a reduced power circuit is in $\uAcf{0}$. 
\end{theoremA}
This theorem gives a partial answer to the question about the complexity of the power word problem in Baumslag-Solitar groups for $p>1$ raised in \cite[Section 6]{LohreyZetsche23}. In \cite{LohreyZetsche23} it is shown that the power word problem for the groups $\BSq$ is in $\Tc{0}$. 
%In this work, we show that a restricted version of it is in $\uAcf{0}$ for the groups $\BS{p}{pq}$, but we allow the exponents to be given as power circuit representations (i.e. in a quite compressed form).  \\

Our main result is on the word problem  for the Baumslag groups $\BGpq$.\footnote{Be aware that in \cite{MattesW22} and the previous version of this work we called the groups $\BGpr$ ``generalized Baumslag groups''. However, \cite{muw11bg} already uses the term ``generalized Baumslag group'' for another kind of groups. Therefore, we omit this term here  and call the groups $\BGpr$ for $p,r \in \Z$  simply ``Baumslag groups'' and changed the title of this paper accordingly. }   Here we prove the following result:

\begin{theoremA}\label{thm:WPinBG}
	For every $p,q \in \Z$ with  $\abs{p}, \abs{q} \geq 1$ the word problem of the Baumslag group $\BGpq$ is in $\uTC{1}$.
\end{theoremA}

In \cite{muw11bg} it is stated that the word problem in $\BGpq$ is in \P using the same approach as for $\BG$, but no explicit proof is given. By proving Theorem~\ref{thm:WPinBG} we give such a proof since $\uTC{1} \subseteq \P$.

A crucial ingredient to prove Theorem~\ref{thm:WPinBG} and reduce  the complexity from \Tc2 (by \cite{MattesWCC23})  to  $\uTC{1}$ is that we succeed to make the operations on power circuits more efficient.
 We  perform all operations directly on reduced base-$\abs{q}$ power circuits.  We allow operations in a SIMD (single instruction multiple data) fashion: many operations of the same type are performed on the same power circuit in parallel in $\uTC{0}$.
Furthermore, we improve the algorithm to get power circuits of quasi-linear size~-- thus, close to the optimal size as in the sequential algorithms \cite{DiekertLU13ijacstacs,Laun14}.

In the last part of our paper, we consider the conjugacy problem for $\mathbf{G}_{p,pq}$. As in \cite{MattesW22} (for $\BGq$)  we use our results for the word problem to improve the complexity of the strongly generic algorithm from \cite{DiekertMW2016SI} by showing (for a precise definition of strongly generic, see \cref{sec:complexity}):

\begin{theoremA}\label{thm:main_conjugacy}
	For every $p,q \in \Z$ with  $\abs{p}, \abs{q}  \geq 1$ the conjugacy problem of the Baumslag group $\mathbf{G}_{p,pq}$ is strongly generically in $\uTC{1}$.

	Moreover, for \emph{every} fixed $g \in \BGq$ (if $p=1$), the problem to decide whether some input word $w$ is conjugate to $g$ is in $\uTC{1}$.

\end{theoremA}

This work is  the full version of the conference publication \cite{MattesW22} where we considered the word problem of $\BGq$. 
Besides containing all proofs omitted in the conference version, the present paper extends the previously published results (also compared to the previous version of this work) in the following ways:
\begin{itemize}

	\item We added a completely new chapter ``The word problem for Baumslag-Solitar groups'' proving Theorem~\ref{thm:WPinBS} and establishing some lemmas in order to prove the extended version of Theorem~\ref{thm:WPinBG} (see the following point).
	
			\item 
	We extended our results from the group $\BGq$ to $\BGpq$ meaning that we obtain new results about the complexity of the word and conjugacy problem for the groups $\BGpq$ for $\abs{p} \geq 1$ (Theorem~\ref{thm:WPinBG} and~\ref{thm:main_conjugacy}).
	While the general outline of the proof and most of the machinery on power circuits remains the same as in \cite{MattesW22}, for this extension, several of the lemmas about the Baumslag group become considerably more complicated.

	\item Moreover, in addition to the conference version, we formulate and prove all our results for uniform circuit families and give some more details on this. 
\end{itemize}

\section{Notation and Preliminaries}

The logarithm $\log$ is with respect to base two, while $\log_{\pcBase}$ denotes the base-$\pcBase$ logarithm. Let $\pcBase \in \N$. Then the  base-$\pcBase$ \ei{tower function} $\tow_\pcBase\colon\N \to \N$
is defined by $\tow_\pcBase(0) = 1$ and $\tow_\pcBase(i+1) = \pcBase^{\tow_\pcBase(i)}$ for $i \geq 0$. 
It is primitive recursive, but already $\tow_2(6)$  written in binary cannot be stored in the memory of any conceivable real-world computer.
We denote the support of a function $f\colon X  \to \R$ by $\supp(f) = \set{x \in X}{f(x) \neq 0}$. Furthermore, the interval of integers $\oneset{i, \dots, j} \sse \Z$ is denoted by $\interval{i}{j}$.  %and we define $\oneinterval{n} = \interval{0}{n-1}$.
For $q,x \in \Z$, we write $q \nmid x$ if $q$ does not divide $x$. Moreover, $\sgn{x}$ denotes the sign of $x \in \Z$.

Let $\Sigma$ be a set. The set of words over $\Sigma$ is denoted by $\Sigma^* = \bigcup_{n\in \N}\Sig^n$.
For an alphabet $\Sigma$ we denote the length of $w \in \Sigma^*$ by $\len{\Sigma}{w}$. Throughout this work we will use different alphabets, and thus need to be more specific here. 

A dag is a directed acyclic graph. 
For a dag $\Gamma$ we write $\depth(\Gamma)$ for  the length (number of edges) of a longest path in $\Gamma$.

\subsection{Complexity}\label{sec:complexity}
We assume the reader to be familiar with the complexity classes \LOGSPACE and \P (polynomial time); see e.g.\ \cite{AroBar09} for details. 
Most of the time, however, we use circuit complexity within \NC.

Throughout, we assume that inputs to and outputs of functions $f$ are encoded over the binary alphabet $\oneset{0,1}$.
Let $k \in \N$. A function $f$ is in \Ac{k} if there is a family of polynomial-size Boolean circuits of depth $\Oh(\log^kn)$ (where $n$ is the input length) computing $f$.  
More precisely, a Boolean circuit is a dag (directed acyclic graph) where the vertices are either input gates $x_1, \ldots, x_n$, or \gnot, \gand, or \gor gates. 
Some of these gates are marked as output gates $o_1, \dots, o_m$.
All gates may have unbounded fan-in (\ie there is no bound on the number of incoming wires).
A function $f:\{0,1\}^* \rightarrow \{0,1\}^*$ belongs to $\Ac{k}$ if there exists a family $(C_n)_{n  \in \N}$ of Boolean circuits such that the $i$-th output gate $o_i$ evaluates to the $i$-th bit of $f(x)$  when assigning $x = x_1 \cdots x_n$ to the input gates.
Moreover, $C_n$ may contain at most $n^{\Oh(1)}$ gates and have depth $\Oh(\log^kn)$. Here, the depth of a circuit is the length of the longest path from an input gate to an output gate.
Likewise  a language $L$ is in $\Ac{k}$ if its characteristic function is $\Ac{k}$-computable.

The class \Tc{k} is defined analogously with the difference that also \gmaj gates are allowed (a \gmaj gate outputs $1$ if its input contains more $1$s than $0$s). Moreover, $\NC = \bigcup_{k\geq 0} \Tc{k} = \bigcup_{k\geq 0} \Ac{k}$. For more details on circuits we refer to \cite{Vollmer99}.

This is in fact the definition of non-uniform $\Ac{k}$. Here ``non-uniform'' means that the mapping $n \mapsto C_n$ is
not restricted in any way. In particular, it can be non-computable. But, the  class $\NC$ is  contained in \P if we consider uniform circuits.  The most ``uniform'' version of $\Ac{k}$ is \DLOGTIME-uniform $\Ac{k}$. For this,
one encodes the gates of each circuit $C_n$ by bit strings of length $\mathcal{O}(\log n)$. Then the circuit family $(C_n)_{n \geq 0}$
is called \emph{\DLOGTIME-uniform}  if (i) there exists a deterministic Turing machine that computes for a given gate $u \in \{0,1\}^*$
of $C_n$ ($|u| \in \mathcal{O}(\log n)$) in time $\mathcal{O}(\log n)$ the type of gate $u$, where the types are $x_1, \ldots, x_n$, \emph{not}, \emph{and}, \emph{or}, oracle gate,
and (ii) there exists a deterministic Turing machine that decides for two given gates $u,v \in \{0,1\}^*$
of $C_n$ ($|u|, |v| \in \mathcal{O}(\log n)$) and a binary encoded integer $i$ with $\mathcal{O}(\log n)$ many bits
in time $\mathcal{O}(\log n)$ whether $u$ is the $i$-th input gate for $v$.
In the following, we write $\uAC{k}$ for \DLOGTIME-uniform $\Ac{k}$ and similarly for $\Tc{k}$. 
For more details on these definitions we refer to~\cite{Vollmer99}.
If the language $L$ (or $K$) in the above definition of $\uAC{k}$ is defined over a non-binary alphabet $\Sigma$, then one first has to fix a binary
encoding of $\Sigma$ as words in $\{0,1\}^\ell$ for some large enough $\ell \in \N$.

To show that a circuit is uniform,  we will use that $\uAC{0}=\mathsf{FO}[<,\mathrm{bit}]$ \cite[Theorem 4.73]{Vollmer99}, with $\mathrm{bit}(i,j)$ holds if and only if the $j$th bit in the binary representation of $i$ is $1$. In particular, if a problem is describable by a first order formula, it is in $\uAC{0}$. 

Note that, as soon as we consider $\Tc{k}$-circuits with $k\geq 1$ we also can use $\L$-uniformity (which is defined similarly as \DLOGTIME-uniformity but is a weaker condition).
Our algorithms (or circuits) rely on two basic building blocks which can be done in $\uTC{0}$:

\begin{example}\label{ex:itAdd}
Base-$\pcBase$ iterated addition is as follows: on input of $n$ base-$\pcBase$ numbers $A_1, \dots, A_n$ each having $n$ digits, compute $\sum_{i = 1}^{n}A_i$ (in base-$\pcBase$ representation). 
For binary numbers this is well-known to be in $\uTC{0}$ (see \eg \cite{Vollmer99} for the description of a non-uniform circuit family; it is easy to transform that circuit family into a uniform circuit family). The standard proof can be translated for other bases. 
The result also can be easily derived from other existing results: as iterated multiplication and division are in $\uTC{0}$ by \cite{hesse01,HeAlBa02}, one can convert a base-$\pcBase$ integer to a binary integer, do the addition in binary, and convert the number back. 
\end{example}

\begin{example}\label{ex:sortTC}
	Let $(k_1, v_1), \dots, (k_n,v_n)$ be a list of $n$ key-value pairs $(k_i, v_i)$ equipped with a total order on the keys $k_i$ such that it can be decided in $\uTC{0}$ whether $k_i < k_j$. 
	Then the problem of sorting the list according to the keys is in $\uTC{0}$: the desired output is a list $(k_{\pi(1)}, v_{\pi(1)}), \dots, (k_{\pi(n)},v_{\pi(n)})$ for some permutation $\pi$ such that $k_{\pi(i)} \leq k_{\pi(j)}$  for all $i < j$.
	We briefly describe a circuit family to do so: The first layer compares all pairs of keys $k_i,k_j$ in parallel. For all $i$ and $j$ the next layer computes a Boolean value $P(i,j)$ which is true if and only if $\abs{\set{\ell}{k_\ell < k_i}} = j$. The latter is computed by iterated addition. 
%	%
	As a final step the $j$-th output pair is set to $(k_i, v_i)$ if and only if $P(i,j)$ is true. According to the above, these operations clearly are in $\uTC{0}$. 
\end{example}

 A function $f$ is \emph{$\Ac{k}$-reducible} to a function $g$ if there is a family of $\Ac{k}$-circuits computing $f$ which does not only use Boolean gates but may also use oracle gates for $g$. Such an oracle gate outputs $g(x)$ on input of $x$. We write 	$\uAcf{k}$ for the uniform family of problems which are $\uAc{k}$-reducible to the word problem of $F_2$. The word problem of $F_2$ is in $\LOGSPACE$ \cite{lz77}  and thus in $\uTC{1}$ \cite[Theorem 4.12]{Vollmer99}. So we obtain that $\uTC{i} \subseteq  \uAcf{i}
 \subseteq \uTC{i+1}$. Moreover, we can find every group $F_n$ as a subgroup of $F_2$. Thus,  $\uAcf{i}=\mathsf{uAC}^{i}(F_n)$.

\paragraph*{Generic case complexity.}
A set $I \sse \Sigma^*$ is called  \emph{strongly generic} if 
the probability to find a random string outside $I$ converges exponentially fast to zero~-- more precisely, if 
%\fi
${|\Sigma^{n} \setminus I|}/ {|\Sigma^{n}|} \in 2^{-\Oh(n)}.$
Let $\cC$ be some complexity class. A problem $L \sse \Sigma^*$ is called \emph{strongly generically in $\cC$} if there is a strongly generic set $I \sse \Sigma^*$  and a (partial) algorithm (or circuit family) $\cA$ running within the bounds of $\cC$ such that $\cA$ computes the correct answer for every $w \in I$; outside of $I$ it  provides either the correct answer or none (or outputs ``unknown'').

\subsection{Notations from group theory}\label{sec:GroupTheory}
\newcommand{\absbeta}[1]{\abs{#1}_\beta}

\paragraph{Group presentations.}

By $F_n$ we denote the free group generated by $n$ elements.  For a group $G$ and $S \sse G$, we write $\gen{S}$ for the subgroup of $G$ generated by $S$.

 If $\Sigma = S \cup S^{-1}$ where $S^{-1}$ is some disjoint set of formal inverses and $R\sse \Sigma^* \times \Sigma^*$ is some set of relations, we write $G=\Gen{S}{R}$ for the group generated by $S$ subject to the relations that $s=r$ for $(r,s) \in R$. 
 Alternatively, we can view $G$ as the quotient monoid of $\Sigma^*$ modulo the congruence  generated by $R$ together with the relations $aa^{-1} = a^{-1} a = 1$ for $a \in \Sigma$.
 Thus, we obtain a canonical surjective monoid homomorphism $\eta\colon  \Sigma^* \to G$. 
 Usually, we do not write the homomorphism $\eta$ and treat words over $\Sigma$ both as words and as their images under $\eta$.
  We write $v =_G w$ with the meaning that $\eta(v) = \eta(w)$ and $v \in \gen{w_1, \dots, w_k}$ if $\eta(v)$ is contained in the subgroup $\gen{\eta(w_1), \dots, \eta(w_k)}$. Similarly, for a monoid homomorphism $\pi: \Theta^* \to \Delta^*$ (with  alphabets $\Theta, \Delta$) we write $v=_{\Delta}w$ if  $\pi(v)=\pi(w)$ (the homomorphism will be always clear from the context).

\medskip
\noindent The word problem for a fixed group $G=\Gen{S}{R}$ with $\Sigma$ as above is as follows:
\ynproblem{A word $w \in \Sigma^*$.}{Is $w =_G 1$?}
\noindent  For further background on group theory, we refer to \cite{LS01}.

\paragraph*{HNN-Extensions.} Let $G$ be a group. Let $A$ and $B$ be subgroups of $G$ and $\phi: A \rightarrow B$ an isomorphism. Let $t$ be a fresh symbol and let $F(t)$ be the free group generated by $t$. The HNN-extension of $G$ w.r.t. $A,B $ and $\phi$ is  defined as 
\begin{align*}
\text{HNN}(G;A,B,\phi)=(G* F(t))/\{tat^{-1}= \phi(a)~|~a \in A\}.
\end{align*} 	 
\paragraph*{Britton reductions.}
Britton reductions are a standard way to solve the word problem for HNN extensions. Let $H=\text{HNN}(G;A,B,\phi)$ be a HNN-extension and let $\Delta$ be the alphabet $\Delta=(G\setminus\{1\}) \cup \{t,t^{\mOne}\}$. A word $w \in \Delta^*$ is called \emph{Britton-reduced} if it is of the form 
\begin{align*}
w=g_0t^{\epsilon_1}g_1 \cdots t^{\epsilon_n}g_n
\end{align*} 
with $\epsilon_i \in \{\pm 1\}$ and $g_i \in G$ (note that we allow $g= 1$ being the empty word) and if there are no factors of the form $tat^{-1}$ with $a \in A$ or $t^{-1}\phi(a)t$ with $a \in A$. If $w$ is not Britton-reduced, one can apply one of the rules
\begin{align*}
tat^{-1} &\:\to\:\phi(a)\\
t^{-1}\phi(a)t &\:\to\: a\\
g_ig_j&\:\to\: [g_ig_j]
\end{align*} 
where $[g_ig_j] $ denotes the element $ g_ig_j$ in the group $G$ and, as above, the neutral element $1$ of $G$ is identified with the empty word. 
We obtain a shorter word (over the alphabet $\Del$) representing the same group element.  The application of such a rule is called a \emph{Britton-reduction}. 

\begin{lemma}[Britton's Lemma] \label{lem:BrittonsLemma}
	Let $w=g_0t^{\epsilon_1}g_1 \cdots t^{\epsilon_n}g_n$ be Britton-reduced and $n \geq 1$. Then $w \neq 1$ in $H$. 
\end{lemma} 

\newcommand{\cQsdrAbbr}{base-$q$ csdr\xspace}
\newcommand{\sdrAbbr}{sdr\xspace}

\section{Compact representations}\label{sec:compact}

Based on the concept of compact  sums and power circuits to base 2 (as introduced in \cite{MyasnikovUW12}), Laun \cite{Laun14} described so-called power sums: A power sum to base $\base \geq 2$ is a sum $\sum_{i\geq 0}a_i q^i$ with $a_i \in \interval{-q+1}{ q-1}$ and only finitely many $a_i$ are non-zero. We are interested in \textit{compact representations} of such power sums.  For $q=2$ they have already been considered  e.g. in  \cite{REITWIESNER60,GUNTZER87,Jedwab89,Shallit93} under different names. 
 In \cite[Proposition 2.18]{Laun14} it is shown that each power sum has a unique compact representation, which is obtained using a confluent rewriting system. Using Boolean formulas for this construction we show that it is in $\uAC{0}$. This will be an important ingredient for our power circuit operations to be in $\uTC{0}$. Note that in \cite[Theorem 3.2]{MattesWCC23} we gave a proof for base $\base=2$. Here, we fix $\base \geq 2$.

\begin{definition}
	Let $A=(a_0, \ldots, a_{m-1})$ be a sequence with $a_i \in \domQ $. 
\vspace*{-1.2mm}%
	\begin{itemize}%[(i)]
		\item We define $\val(A)=\sum_{i=0}^{m-1}a_i\cdot \base^i$.
		\item We call $A$ a \emph{(base-$\base$) signed-digit representation} (short \sdrAbbr) of $\val(A)$. 
		\item We call $A$ \emph{compact} if the following conditions hold for all $i \in \interval{0}{m-2} $: 
		\begin{enumerate}[(1)]
			\item if $\abs{a_i}=\base-1$, then $\abs{a_{i+1}}<\base-1$, 
			\item  if $\abs{a_i} \neq 0$, then  $a_{i+1}=0$ or $\sgn {a_{i}}=\sgn{a_{i+1}}$. 
		\end{enumerate}
	\end{itemize}
\end{definition}
We set $a_i=0$ for $i \geq m$. Note that, if $A=(a_0, \ldots, a_{m-1})$ is an \sdrAbbr with $a_i \in \interval{0}{\base-1}$, we have a usual base-$\base$ representation of an integer. Allowing negative digits gives more flexibility when working with power circuits~-- however, with the price that representations are no longer unique. This uniqueness property can be regained by requiring the \sdrAbbr to be compact. 
Moreover, \cQsdr{}s  (for short \cQsdrAbbr) can be compared easily.

\begin{lemma}[\!\,{\cite[Proposition 2.18]{Laun14}}] \label{lem:compareCQSDR}
	For every $x \in \Z$ there is a unique compact base-$\base$ signed-digit representation $A=(a_0, \dots, a_{m-1})$ with $\val(A)=x$. 
	
	Moreover, two \cQsdrAbbr{}s $A = (a_{0}, \ldots, a_{m-1})$ and $B = (b_{0}, \ldots,b_{m-1})$ can be compared using the lexicographical order~-- more precisely, $\val(A) < \val(B)$ if and only if $a_{i_{0}} < b_{i_{0}}$ for 
	$i_{0} = \max\set{i \in \interval{0}{m-1}}{a_{i} \neq b_{i}}$.
\end{lemma}
If $A$ is a base-$\base$ \sdrAbbr, in the following we will write $\cor(A)$ for its \cQsdr. % 

\begin{lemma}\label{lem:maxCQsdr}
	If $A=(a_0, \dots, a_{m-1})$ is a \cQsdrAbbr, then $\val(A) \leq \floor{\frac{\base^{m+1}}{\base^2-1}}$.
\end{lemma}
\begin{proof}
	It is clear that $\val(A)$ becomes maximal if $a_{m-1} = \base-1$ and then (going from right to left) $\base -2$ and $\base -1$ alternate. Thus, the maximal value can be computed as
	\begin{align*}
		\sum_{i = 0}^{m-1} (\base -2) \base^i + \sum_{j = 0}^{\floor{(m-1)/2}} \base^{m-1 - 2j} = \floor{\frac{\base^{m+1}}{\base^2-1}}. %= (q-2) \frac{q^m-1}{q-1} + q^{m-1}\frac{q^{-2(\floor{(m-1)/2} + 1)}-1}{q^{-2}-1}
	\end{align*}\qed
\end{proof}

Next, we construct the \cQsdrAbbr of a given \sdrAbbr $A=(a_0, \dots, a_{m-1})$. By giving this construction, we also show that it is possible to construct $\cor(A)$ in $\Ac{0}$ for a given \sdrAbbr $A$.  We start by restricting $A$ to only non-negative digits (\ie a usual base-$\base$ representation). In a first step we need the following formula for $i \geq 0$: 
\begin{align*}
	\cy_{i} = &\bigvee_{j\in \interval{1}{i}}\Bigl[(a_{j}=q-1) \wedge (a_{j-1}=q-1)  \\&\quad \wedge\bigwedge_{k\in \interval{j+1}{i-1}}\bigl(\left((a_{k}=q-1) \vee (a_{k+1}=q-1)\right)\wedge (a_k\geq \base-2)\bigr)\Bigr]
\end{align*}

\begin{lemma}\label{lem: compactstepone} 
	For every signed-digit representation $A=(a_0, \ldots, a_{m-1})$ with $a_i \in \interval{0}{\base-1} $ the sequence $B=(b_0, \ldots, b_m)$ defined by $ b_i=a_i -\base\cdot \cy_{i+1}+\cy_i$ satisfies:
	\begin{itemize}
		\item $\val(A)=\val(B)$
		\item  $b_i \in \interval{\mOne}{\base-1} $
		\item $b_i=\mOne \implies b_{i+1}=0$  
		\item $b_i=\base-1\implies b_{i+1} < \base-1$
	\end{itemize}
\end{lemma}
Note that $b_i>0$ and $b_{i+1}=-1$ is still possible. This is the only condition that is missing for $b$ being compact. 

\begin{proof}
	First observe that because $e_{m+1}=e_0=0$, $\val(A)=\val(B)$ follows directly from the definition of the $b_i$.
	
	Next we show that $b_i \in \interval{\mOne}{\base-1}$. First assume that $\cy_{i+1}=1$. Then $a_i \geq \base-2$. If $a_i=\base-2$, then $\cy_i=1$ as otherwise $\cy_{i+1}=1$ is not possible. So if $e_{i+1}=1$, then $b_i \in \interval{\mOne}{\base-1}$. Further, if $a_i=\base-1$ and $\cy_i=1$, then $\cy_{i+1}=1$. So $b_i \in \interval{\mOne}{\base-1}$  if $e_i=1$. If $e_i=e_{i+1}=0$, then $b_i=a_i \in \interval{0}{\base-1}$.  \\

	To show that $b_i=\mOne$ implies that $b_{i+1}=0$ we first assume for a contradiction that $b_{i+1}=\mOne$ meaning that
	\begin{align}\label{minOne}
		\begin{split}
			b_i=a_i-\base\cdot \cy_{i+1} +\cy_i&=\mOne \qquad \text{ and}\\
			b_{i+1}=a_{i+1}-\base\cdot \cy_{i+2} +\cy_{i+1}&=\mOne. 
		\end{split}
	\end{align}
	
	If $e_j=e_{j+1}=0$, then $b_j=a_j$.  But $a_j \geq 0$ for all $j$, so this is not possible if $b_j=\mOne$. Moreover, if $e_j=1$ and $e_{j+1}=0$, then $a_j=-2$ for $j \in \{i,i+1\}$. This contradicts $a_j \geq 0$.   In the following table, we consider the remaining  possibilities for $e_i,e_{i+1},e_{i+2}$ and calculate $a_i$ and $a_{i+1}$ assuming (\ref{minOne}). In each row, the \textcolor{blue}{blue} entries lead to a contradiction which is described in the last column.

	\begin{center}
		\begin{tabular}[h]{|c|c|c||c|c||c|}
			\hline
			$\cy_{i}$ & $\cy_{i+1}$ & $\cy_{i+2} $ & $a_{i}$ & $a_{i+1}$ & contradicting \\
			%\hline 
			%$0$ & $1$ & $0$ & $\base-1 $ & \textcolor{blue}{$-2 $} &$a_i \geq 0$\\
			\hline 
			\textcolor{blue}{$0$} & $1$ & $1$ & \textcolor{blue}{$\base-1 $} & \textcolor{blue}{$\base-2 $} & $\cy_{i+1}=1$\\
		%	\hline 
			%$1$ & $0$ & $1$ & \textcolor{blue}{$-2 $} & $\base-1 $ &$a_i \geq 0$\\
			%\hline 
			%$1$ & $1$ & $0$ & $\base-2 $ & \textcolor{blue}{$-2 $} &$a_i \geq 0$\\
			\hline 
			$1$ & $1$ & $1$ & \textcolor{blue}{$\base-2 $} & \textcolor{blue}{$ \base-2$} & $\cy_{i+1}=1$\\
			\hline 
		\end{tabular}
	\end{center}
	Hence, we have shown that $b_i=\mOne$ implies $b_{i+1} \neq \mOne$.
	We now assume that 
	\begin{align*}
		b_i=a_i-\base\cdot \cy_{i+1} +\cy_i&=\mOne \qquad\text{ and}\\
		b_{i+1}=a_{i+1}-\base\cdot \cy_{i+2} +\cy_{i+1}&>0.
	\end{align*}
	As for (\ref{minOne})
	the cases $e_{i+1}=e_i=0$, and  $e_i=1,~ e_{i+1}=0$ lead to a contradiction. Moreover, if $e_{i}=0$ and $e_{i+1}=1$, then $a_i=a_{i+1}=q-1$. Thus, $e_{i+2}=1$. But then, $a_{i+1}>q-1$ which contradicts $a_{i} \leq q-1$. It remains the case that $e_i=e_{i+1}=1$. Then, $a_{i}=q-2$ and thus $a_{i+1}=q-1$. This contradicts $e_{i+2}=0$. 
	
%	 $e_i=0 $ and $e_{i+1}=1$, is only possible if  $a_i=a_{i+1}=\base-1$. If $e_{i+1}=1$ and $a_i=\base-2$, then $a_{i+1}=\base-1$. For the remaining cases we construct a table in the same way as we did above. Again, the \textcolor{blue}{blue} entries lead to the contradiction in the last column.  
%	
%	\begin{center}
%		\begin{tabular}[h]{|c|c|c|c|c|c|}
%			\hline
%			$\cy_{i}$ & $\cy_{i+1}$ & $\cy_{i+2} $ & $a_{i}$ & $a_{i+1}$ & contradicting \\
%			\hline 
%			$0$ & $1$ & $0$ &  \textcolor{blue}{$\base-1 $} & \textcolor{blue}{$\base-1$}  & $\cy_{i+2}=0$\\
%			\hline 
%			$0$ & $1$ & $1$ & $\base-1$  &  \textcolor{blue}{$ >\base-1$}  & $a_{i} <\base-1$  \\
%			\hline 
%			$1$ & \textcolor{blue}{$1$} & $0$ & $\base-2$  & \textcolor{blue}{$\base-1$}  & $\cy_{i+2}=0$\\
%			\hline 
%			$1$ & $1$ & $1$ & $\base-2$ & \textcolor{blue}{$>\base-1$}  & $a_{i} <\base-1$\\
%			\hline 
%		\end{tabular}
%	\end{center}
	
	So we showed that, if $b_i=\mOne$, then $b_{i+1}=0$. It remains to show that $b_{i}=\base-1$ implies that $b_{i+1}<\base-1$. We assume that 
	\begin{align*}
		b_i=a_i-\base\cdot \cy_{i+1} +\cy_i&=\base-1 \qquad\text{ and}\\
		b_{i+1}=a_{i+1}-\base\cdot \cy_{i+2} +\cy_{i+1}&=\base-1.
	\end{align*}
	First observe that if $e_{j+1}=1$, then $a_j>\base-1$ for $j=i,i+1$. If $e_i=e_{i+1}=e_{i+2}=0$, then $a_i=a_{i+1}=\base-1$, which is a contradiction to $e_{i+1}=0$. It remains the case $e_i=1$, $e_{i+1}=e_{i+2}=0$. But then $a_i=\base-2$ and $a_{i+1}=\base-1$. With $e_i=1$ we obtain a conatradiction to $e_{i+1}=0$.  
	This shows the lemma. 
	\qed\end{proof}

For the second step~-- to make a signed-digit representation $B=(b_0, \ldots, b_m)$ as in \cref{lem: compactstepone} compact~-- we need the following formula: 
\begin{align*}
	f_i=\bigvee_{j\in \interval{i}{m}} \Bigl[(b_j=-1) \wedge  \bigwedge_{\ell\in \interval{i-1}{j-1}}(b_\ell >0) \Bigr]
\end{align*}
\begin{lemma}\label{lem: compactsteptwo}
	For every signed-digit representation $B=(b_0, \ldots, b_{m})$ satisfying the conditions of the output of \cref{lem: compactstepone} (\ie $b_i \in \interval{\mOne}{\base-1}$, $b_i=\base-1$ implies $b_{i+1}<\base-1$ and $b_i=-1$ implies $b_{i+1}=0$) the sequence $C=(c_0, \ldots, c_{m})$ defined by $c_i=b_i-\base\cdot f_{i+1} + f_{i}$  satisfies:  
	\begin{itemize}
		\item $\val(B)=\val(C)$
		\item $c_i \in \interval{-\base+1}{\base-1}$ 
		\item $C$ is compact. 
	\end{itemize}
\end{lemma}

\begin{proof}
	As in the proof of the previous lemma, because of  $f_{m+1}=f_0=0$, $\val(B)=\val(C)$ follows directly from the definition of the $c_i$. Before proving the other conditions, we first make the following observations: 
	\begin{enumerate}[(I)]
		\item  \label{pone}$b_i=0 \implies f_i=f_{i+1}=0$; therefore, $b_i=0 \implies c_i=0$. 
		\item  \label{ptwo}$b_i=\mOne $ implies that $b_{i+1}=0$, and thus also $f_{i+1}=0$. 
		\item  \label{pthree} If $f_{i}=1$ and $f_{i+1}=0$, then $b_i=\mOne$.  
	\end{enumerate}

	It is clear that, if $0<b_i<\base-1$, then $c_i \in \interval{-\base+1}{\base-1}$. If $b_i=0$, then  $c_i = 0$ by \eqref{pone}. If $b_i=\mOne$, then $c_i \in \oneset{-1,0}$ by \eqref{ptwo}. Now assume that $b_i=\base-1$. If $f_i=f_{i+1}=0$, then $c_i=\base-1$. If $f_{i+1}=1$, then $c_i \in \oneset{\mOne, 0}$. Because of \eqref{pthree}, $f_i=1$ and $f_{i+1}=0$ is not possible. So we showed that $c_i \in \interval{-\base+1}{\base-1}$. \\

	Now we will prove that $C$ is compact. First, we show that $\abs{c_i} =\abs{ c_{i+1}} = \base -1$ is not possible. Observe that $|c_i|=\base-1$ is only possible if $b_i \in \oneset{0, 1, \base-2, \base-1} $. Because of \eqref{pone} only the cases $b_i, b_{i+1} \in \oneset{1,\base-2,\base-1}$ remain.

	If $f_i=f_{i+1}=f_{i+2}=0$, then $c_i=b_{i}$ and $c_{i+1}=b_{i+1}$ and we are done as already in $B$  there are no two $q-1$'s next to each other.  Moreover, if $b_j \in  \oneset{1,\base-2,\base-1}$, then $\abs{b_j-\base +1} < \base -1$. So there is nothing more to show in case $f_{i+1}=f_{i+2}=1$.   
	If $f_{i+2}=1$ and $b_i, b_{i+1} \in  \oneset{1,\base-2,\base-1}$, then $f_{i+1}=1$~-- so the previous case applies. The remaining cases ($f_{i+2}=0$ and either $f_{i+1} =1$ or $f_i=1$) are ruled out by \eqref{pthree}. 
	\\

	We still have to show that if $c_i \neq 0$, then $c_{i+1}=0$ or $\sgn{c_i}=\sgn{c_{i+1}}$. First assume that $c_i>0$. Then $f_{i+1}=0$ and by \eqref{pone} we have $b_i>0$. If  $c_{i+1}<0$ and  $f_{i+2}=0$, then $b_{i+1}=\mOne$. Because $b_i>0$ this implies that $f_{i+1}=1$, which is a contradiction.
	So if $c_{i+1}<0$, then $f_{i+2}=1$ and so $b_{i+1}>0$. Together with $b_i>0$ this implies that $f_{i+1}=1$. So, $c_i>0$ and $c_{i+1}<0$ is not possible. 
	
	Now assume that $c_i<0$. If $f_{i+1}=0 $, then $ b_i=\mOne$. So by assumption, $b_{i+1}=0$. Thus, $c_{i+1}=0$ by \eqref{pone}. 
	So assume that $f_{i+1}=1$ and $c_{i+1}>0$. It follows that $f_{i+2}=0$. 
	So \eqref{pthree} implies that $b_{i+1}=\mOne$. Thus, $c_{i+1} = 0$.  So we considered all possible cases and  showed that $C$ is compact. 
	This proves the lemma. 
	\qed\end{proof}

\begin{theorem}\label{thm:makeCompact}
	The following is in \uAC{0}:% (even if $\base$ given in binary is part of the input): 
	\compproblem{A base-$\base$ signed-digit representation $A=(a_0, \ldots, a_{m-1})$. }{ A \cQsdrAbbr $B=(b_0, \ldots, b_m)$  such that $\val(A)=\val(B)$.} 
\end{theorem}
\begin{proof}
	Observe that there exist signed-digit representations $C=(c_0, \ldots, c_{m-1})$ and $D=(d_0, \ldots, d_{m-1})$ such that $c_i, d_i \in \interval{0}{ q-1}$ and such that $\val(A)=\val(C)-\val(D)$ (we just collect the negative digits of $A$ into $D$ and the positive ones into $C$). Now, we compute $\abs{\val(C)-\val(D)}$ in the usual base $\base$  representation (which is in $\uAC{0}$~-- see \eg \cite[Theorem 1.15]{Vollmer99}  for base 2; the general case follows the same way) and make it compact by first applying
	\cref{lem: compactstepone} and then \cref{lem: compactsteptwo}. This clearly is in $\uAC{0}$ because we showed that it can be described by formulas using only $\land, \lor$ and addition. If $\val(C)-\val(D) < 0$, we invert this number digit by digit.
\qed\end{proof}

\section{Power circuits} \label{sec: PowerCircuits}

The original definition \cite{MyasnikovUW12} is for power circuits to base 2. Here, following \cite{Laun14}, we define power circuits with respect to an arbitrary base $\pcBase$~-- hence, from now on we fix $\pcBase \geq 2$.  
Consider a pair $(\GG,\del)$ where $\GG$ is a set of $n$ vertices and $\del$ is a mapping $\del\colon  \GG \times \GG\to \domQ$.
The support of $\del$ is the subset $\supp(\del) \sse \GG \times \GG$ 
consisting of those $(P,Q)$ with $ \del(P,Q) \neq 0$.  
Thus, $(\GG,\supp(\del))$ is a directed graph without multi-edges. Throughout we require that $(\GG, \supp(\del))$ is acyclic~-- \ie it is a dag. 
In particular, $\del(P,P) = 0$ for all vertices $P$. 
 A \ei{marking} is a mapping $M\colon  \GG \to \domQ$. 
Each node $P\in \GG$ is associated in a natural way with a marking $\LL_P\colon  \GG \to \domQ, \; Q \mapsto \del(P,Q)$ called its successor marking. The support of $\LL_P$ consists of the target nodes of outgoing edges from $P$. We denote the marking with the empty support by $\emptyset$. 
We define the \ei{evaluation} $\e(P)$ of a node ($\e(M)$ of a marking resp.)
bottom-up in the dag by induction:
\begin{align*}
\e(\es) & = 0, \\
\e(P) & = \pcBase^{\e(\LL_P)} &\text{for a node $P$}, \\
\e(M) & = \sum_{P}M(P)\e(P) &\text{for a marking  $M$}.
\end{align*}
We have $\eps(\LL_P) = \log_\pcBase(\eps(P))$, \ie the marking $\LL_P$ plays the role of a logarithm. Note that nodes of out-degree zero (sinks) evaluate to $1$ and every node evaluates to a positive real number. However, we are only interested in the case that all nodes evaluate to integers:

\begin{definition}\label{def:PC}
	A (base-$\pcBase$) \emph{\PC} is a pair $(\GG,\del)$ with $\del\colon  \GG \times \GG\to \domQ$ such that $(\GG, \supp(\del))$
	is a dag and all nodes evaluate to an integer in $\pcBase^{\N}$. 
\end{definition}
The size of a power circuit is the number of nodes $\abs{\Gamma}$.
By abuse of language, we also simply call $\Gamma$ a power circuit and suppress $\delta$ whenever it is clear. 
If $M$ is a marking on $\Gamma$ and $S \sse \Gamma$, we write $M|_S$ for the restriction of $M$ to $S$.
Let  $(\Gamma', \delta')$ be a power circuit, $\Gamma \sse \Gamma'$, $\delta = \delta'|_{\Gamma \times \Gamma}$, and  $\delta'|_{\Gamma \times (\Gamma'\setminus \Gamma)} = 0$. Then $(\Gamma, \delta)$ itself is a power circuit. We call it a \emph{sub-power circuit} and denote this by $(\Gamma, \delta) \leq (\Gamma', \delta')$.
If $M$ is a marking on $ S \sse \Gamma$, we extend $M$ to $\Gamma$ by setting $M(P) = 0$ for $P \in \Gamma \setminus S$. With this convention, every marking on $\Gamma$ also can be seen as a marking on $\Gamma'$ if  $(\Gamma, \delta) \leq (\Gamma', \delta')$.
If $M$ is a marking, we write $-M$ for the marking defined by $(-M)(P) = - (M(P))$, which clearly evaluates to $-\eps(M)$. 
%For a list of markings $\vec M=(M_1, \ldots, M_n)$ we define $\msum{\vec M}=\sum_{i=1}^{n}\abs{\sigma(M_i)}$ (and $\msum{M}=\abs{\sigma(M)}$ for a single marking).

\begin{example}\label{ex:powtow}
	A \PC of size $n+1$ to base $\pcBase$ can realize $\tow_q(n)$ since a directed path of $n+1$ nodes represents
	$\tow_q(n)$ as the evaluation of the last node. The following power circuit to base $2$ realizes $\tow_2(5)$ using $6$ nodes: 	
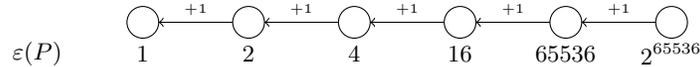
\begin{center}%\vspace{-3mm}
			\tikzstyle{pcnode} = [minimum size = 12pt,circle,draw ]
				\begin{tikzpicture}[outer sep = 0pt, inner sep = 0.7pt, node distance = 40]
					{\footnotesize
						\node[pcnode] (1)    {};
						\node[pcnode, right of = 1] (2)   {};
						\node[pcnode, right of = 2] (3)    {};
						\node[pcnode, right of = 3] (4)   {};
						\node[pcnode, right of = 4] (5)   {};
						\node[pcnode, right of = 5] (6)   {};
						\node [below of=1, yshift=28](7){\footnotesize $1$};
						\node [below of=2, yshift=28](8) {\footnotesize $2$};
						\node [below of=3, yshift=28](9){\footnotesize $4$};
						\node [below of=4, yshift=28](10){\footnotesize $16$};
						\node [below of=5, yshift=28](11){\footnotesize $65536$};
						\node [below of=6, yshift=28](12){\footnotesize $2^{65536}$};
						\node [left of=1](13){};
						\node [left of=7](17){\small $\epsilon(P) $};
						
\path[->]         		(2) edge node[above, yshift=2]{\tiny $+1$} (1)
						(3) edge node[above, yshift=2]{\tiny $+1$}  (2)
						(4) edge node[above, yshift=2]{\tiny $+1$} (3)
						(5) edge node[above, yshift=2]{\tiny $+1$} (4)
						(6) edge node[above, yshift=2]{\tiny $+1$} (5)	
						;
}
				\end{tikzpicture}
		
			\end{center}
\end{example}

\begin{example}\label{ex:binarybasis}
	We can represent every integer in the range 
	$[-\pcBase^n + 1,\pcBase^n-1]$ by some marking on a base $\base$ \PC  with nodes $\oneset{P_{0} \lds P_{n-1}}$ with $\e(P_i) = \pcBase^{i}$ for $i \in \interval{0}{n-1}$.
	Thus, we can convert the $\pcBase$-ary notation of an $n$-digit integer into a \PC
	with $n$ vertices, $\Oh( n\log_{\pcBase} n)$ edges (each successor marking requires at most $\floor{\log_{\pcBase} n} + 1$ edges) and depth at most $\log_{\pcBase}^*n$. %
	 For an example of a marking representing the integer $187$ to base $3$, see \cref{fig: exmarking}. 
\end{example}

\begin{figure}[tbh]
	\begin{center}
		\tikzstyle{pcnode}=[minimum size= 12pt,circle,draw ]
	%	\vspace*{-0.9cm}
		\begin{tikzpicture}[scale=1, outer sep=0pt, inner sep = 0.7pt, node distance=1.2cm]
			{

				\node[pcnode] (0) at (0,0) {\textcolor{blue}{\tiny $-2$}};
				\node[pcnode] (1) [right of=0] {};
				\node[pcnode] (2) [right of=1] {};
				\node[pcnode] (3) [right of=2] {\textcolor{blue}{\tiny +1}};
				\node[pcnode] (4) [right of=3] {\textcolor{blue}{\tiny +2}};
%				\node[pcnode] (5) [right of=4] {};
				
				\node (6) [below of =0, yshift=9] {\footnotesize  1};
				\node (7) [below of =1,yshift=9] {\footnotesize  3};
				\node (8) [below of =2,yshift=9] {\footnotesize  9};
				\node (9) [below of =3,yshift=9] {\footnotesize 27};        
				\node (10) [below of =4,yshift=9] {\footnotesize  81};

				\draw[->] (1) edge node[above, yshift=1] {\tiny +1} (0)
				(2) edge [bend left=25] node[below, yshift=-1 ] {\tiny+2} (0)
				(4) edge [bend right=38]node[below, xshift=-25, yshift=2] {\tiny+1} (1)
				(4) edge[bend right=39] node[below left, xshift=-40, yshift=-3] {\tiny+1} (0)
				(3) edge[bend right=35] node[above,  xshift=23, yshift=-4] {\tiny+1} (1)
				;
			}
		\end{tikzpicture}
	\end{center}
		\caption{Each integer $z \in \interval{-242}{242}$ can be represented by a marking on the following power circuit. The marking given in \textcolor{blue}{blue} is representing the number $187$.}\label{fig: exmarking}
\end{figure}
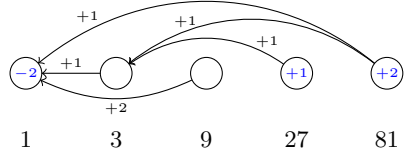\vspace{-4mm}
\begin{definition}\label{def:PCreduced}
	\begin{enumerate}[(a)]
		\item 	
		We call a marking $M$ \emph{compact} if for all $P, Q \in \supp(M)$  with  $P \neq Q$  we have $\eps(P) \neq \eps(Q)$ and, if $\abs{M(P)}=\abs{M(Q)}=q-1$ or $\sgn{M(P)}\neq \sgn{M(Q)}$, then $\abs{\epsilon(\Lambda_{P})-\epsilon(\Lambda_{Q})} \geq 2$. 
		\item A \emph{reduced power circuit} of size $n$ is a power circuit $(\Gamma,  \delta)$ with $\Gamma$ given as a sorted list $\Gamma = (P_{0}, \ldots, P_{n-1})$ such that all successor markings are compact and $\epsilon(P_{i}) < \epsilon(P_{j})$ whenever $ i < j$. In particular, all nodes have pairwise distinct evaluations. 
	\end{enumerate}
\end{definition}
Note that by \cite[Corollary 6.6]{MattesWCC23} it is crucial that the nodes in $\Gamma$ are sorted by their values. Still, sometimes it is convenient to treat $\Gamma$ as a set~-- we write $P \in \Gamma$ or $S \sse \Gamma$ with the obvious meaning. 

Also note some slight differences compared to other literature:
In \cite{DiekertLU13ijacstacs,Laun14}, the definition of a reduced power circuit also contains a bit-vector indicating which nodes have successor markings differing by one~-- we compute this information on-the-fly whenever needed. Moreover, in \cite{Laun14} the (successor) markings of a reduced power circuit do not have to be compact.  Working only with compact markings helps us to compare them in $\uAC{0}$.

\begin{remark}\label{rem:toposort}
	If $(\Gamma, \delta)$ is a reduced power circuit with $\Gamma = (P_{0}, \ldots, P_{n-1})$, we have $\delta(P_i, P_j) = 0$ for $j \geq i$. Thus, the order on $\Gamma$ by evaluations is also a topological order on the dag $(\Gamma, \supp(\delta))$.
\end{remark}

\begin{definition}\label{def:chain}
	Let $(\Gamma, \delta)$ be a reduced power circuit with $\Gamma = (P_{0}, \ldots , P_{n-1})$. 	
	\vspace*{-1.3mm}\begin{enumerate}[(i)]
		\item A \emph{chain} $C$ of length $\ell = \abs{C}$ in $\Gamma$ is a sequence $(P_{i}, \ldots, P_{i+\ell-1})$ such that $\epsilon(P_{i+j+1}) = q \cdot \epsilon(P_{i+j})$ for all $j \in \interval{0}{\ell-2} $.

		\item We call a chain $C$ \emph{maximal} if it cannot be extended in either direction. 
		We denote the set of all maximal chains by $\mathcal{C}_{\Gamma}$.
		
		\item There is a unique maximal chain $C_0$ containing the node $P_0$ of value $1$. We call $C_0$ the \emph{initial maximal chain} of $\Gamma$ and denote it by $C_0 = C_0(\Gamma)$.
		
	\item Let $M$ be a marking on $(\Gamma, \delta)$ and  $C = (P_{i}, \ldots, P_{i+\ell-1}) \in \mathcal{C}_{\Gamma}$. For $a_{j} = M(P_{i+j})$ for $j \in \interval{0}{\ell-1}$ we write $\binM{M}{C} = (a_{0}, \ldots, a_{\ell-1})$.

	\end{enumerate}
\end{definition}
Note that a marking $M$ is compact if and only if $\binM{M}{C}$ is compact (in the sense of \cref{sec:compact}) for all $C \in \mathcal{C}_{\Gamma}$.

\subsection{Operations on reduced power circuits}\label{sec:pc_operations}

\newcommand{\intervalBig}[2]{\Bigl[ \mathinner{#1}..\mathinner{#2}\Bigr] }

We continue with  fixed $\pcBase\geq 2$ and assume that all power circuits are with respect to base $\pcBase$. We can  compare compact markings on reduced base-$\pcBase$ power circuits in $\uAC{0}$. 

 \begin{lemma}[\normalfont \comp] \label{lem:compareCompactMarkings}
	Let ${\comparator} \in \compOpSet$. The following problems are in $\uAC{0}$:
	\begin{enumerate}[(a)]
		\item\label{compMarkings} \ynproblem{A reduced power circuit $(\Gamma, \delta)$ and compact markings $L,M$ on $\Gamma$.}{Is $\epsilon(L) \comparator \epsilon(M)$? }
		\item\label{compMarkingsPlusK} \ynproblem{A reduced power circuit $(\Gamma, \delta)$ with compact markings $L, M$ and $k \in \interval{0}{\bigl\lfloor\frac{q^{\abs{C_{0}}+1}}{q^2-1} \bigr\rfloor } $ given in binary.}{ Is $\epsilon(L) \comparator \epsilon(M)+ k $?} 
	\end{enumerate}
\end{lemma}
The proof of part (\ref{compMarkings}) is a straightforward application of \cref{lem:compareCQSDR} (checking these conditions can be cleary expressed by a first order formula and thus is in $\uAC{0}$) and of the fact that a marking $M$ is compact if and only if $\binM{M}{C}$ is compact for all $C \in \mathcal{C}_{\Gamma}$. 

To prove part (\ref{compMarkingsPlusK}), we need a base-$q$ analog of \cite[Lemma 4.3]{MattesWCC23}. The proof of this analog is verbatim the same, just replacing $2$ by $q$. Using this, to prove (\ref{compMarkingsPlusK}), we only need addition of $q$-ary numbers and  part (\ref{compMarkings}). 
\medskip

In our definition of reduced power circuits there is no data structure describing the maximal chains. Using the following lemma, we can compute the maximal chains  whenever needed:
\begin{lemma}[{\cite[Corollary~4.7]{MattesWCC23}}]\label{cor:findChains}
	We can decide in $\uAC{0}$, given a reduced power circuit $(\Gamma, \delta)$ and nodes $P,Q \in \Gamma$, whether $P$ and $Q$ belong to the same maximal chain of $\Gamma$.
\end{lemma}
Note that in \cite{MattesWCC23} the lemma was stated in the non-uniform case. But the proof only involves comparison of compact markings and \sdr{}s, and so it clearly is in $\uAC{0}$. 
\medskip

The next lemma turns out to be quite versatile and of interest on its own. In particular, it allows to compare a marking on a reduced power circuit in $\uTC{0}$ with some integer given in binary. Moreover, we use it to get rid of the technical condition $\mu \leq \floor{{2^{\abs{C_{0}(\Gamma)}+1}}/{3}} $ of \cite[Lemma~4.15]{MattesWCC23} leading to \cref{stz:extendChains} below. 

\begin{lemma}\label{lem: longchain}
The following problem is in $\uTC{0}$: 
\compproblem{A reduced power circuit $(\Gamma, \delta)$ and $\mu \in \mathbb{N}$ given in unary.  }{A reduced power circuit $(\Gamma', \delta')$ such that $\abs{\mathcal{C}_0(\Gamma')} \geq \mu$ and\newline $(\Gamma, \delta)\leq(\Gamma', \delta') $ (and $\abs{\Gamma'} \leq \abs{\Gamma} + \mu$). }
\end{lemma}

\begin{proof}
	Writing $\nu = \ceil{\log_q \mu}+1$, we know that every node to be added for the desired chain $\mathcal{C}_0(\Gamma')$ has a successor marking using only the first $\nu$ nodes of $\mathcal{C}_0(\Gamma')$. 
	Thus, as a first step for proving \cref{cor:findChains}, we will extend $\mathcal{C}_0(\Gamma)$ to have length at least $\nu$.
	
	We start by creating a new reduced power circuit $\Delta= (Q_0, \dots, Q_\nu)$ such that $\eps(Q_i) = \pcBase^i$: We  write every number $i \in \interval{0}{\nu}$ as its compact base-$q$ signed-digit representation.  Then,  $\Lambda_{Q_i}(Q_j)$  is exactly the $j$-th digit of the representation of $i$ and, thus, defines the successor marking of the node $Q_i$. This is clearly possible in $\uTC{0}$. 
	
	\newcommand{\numI}[1]{\lambda(#1)}
	Next, we wish to merge $\Delta$ with our input power circuit $\Gamma = (P_0, \dots, P_{n-1})$. In order to do so, we ``guess'' a subset $X \sse \interval{0}{\nu}$ (note that there are at most $2^{\nu+1}$ guesses, so they can be checked all in parallel). 
	For $i\in X$ we write $\numI{i} = j-1$ if $i$ is the $j$-th element in the sorted order of $X$ (\ie $\numI{i} = \abs{\set{j \in X}{j < i}}$). Note that sorting and counting is possible in $\uTC{0}$.
	
	For the guessed set $X$ we want to check whether $\eps(Q_i) = \eps(P_{\numI{i}})$ for all $i \in X$ holds and whether $X$ is maximal with this property, \ie the intuition behind $X$ is that it comprises all nodes of $\Delta$ which are also present in $\Gamma$. Note that, because $\Gamma$ is reduced,  if $\abs{\Gamma} \geq 2$ then $\eps(P_i)=\eps(Q_i)$ for $i=1,2$. 
	
	To verify whether $\eps(Q_i) = \eps(P_{\numI{i}})$ holds  for all $i \in X$, ideally, we would check whether $\Lambda_{Q_i} = \Lambda_{P_{\numI{i}}}$.
	However, this is not possible as the successor markings are on different power circuits. Instead, for all $i \in X$ we check whether
		\begin{itemize}
			\item $\Lambda_{Q_i}(Q_j) = \Lambda_{P_{\numI{i}}}(P_{\numI{j}})$ for all $j\in X$,
			\item $\Lambda_{Q_i}(Q_j) = 0$ for all $j \not\in X$,
			\item $\Lambda_{P_{\numI{i}}}(P_{k}) = 0$ for all $k \not\in \numI{X}$. % which are not of the form $\numI{j}$ for $j \in X$.
		\end{itemize}
	Note that this clearly can be done in $\uTC{0}$. 
	Now, if these conditions are satisfied, it follows by induction on $i$ that $\eps(Q_i) = \eps(P_{\numI{i}})$ holds for all $i \in X$. On the other hand, as compact markings are unique,  $\eps(Q_i) = \eps(P_{\numI{i}})$ implies that also these conditions are satisfied.
		
	To check whether $X$ is maximal with $\eps(Q_i) = \eps(P_{\numI{i}})$ for all $i \in X$, we proceed as follows: if $X$ is not maximal, there are $Q_i \in \Delta$  for $i \not\in X$ and $P_j \in \Gamma$ with $\eps(Q_i) = \eps(P_j)$. Taking such $Q_i$ of smallest evaluation, it follows that their successor markings are the same  (when identifying $Q_k$ with $P_{\numI{k}}$ for $k \in X$) and use only nodes from $\set{Q_k}{k \in X}$. Hence, it can be checked in  $\uTC{0}$ whether $X$ is (not) maximal. 	
	
	From now on, let us assume that we already have computed the maximal $X\sse \interval{0}{\nu}$ with $\eps(Q_i) = \eps(P_{\numI{i}})$ for all $i \in X$. 
	Now, we have to integrate all nodes $Q_k$ for $k\in \interval{0}{\nu} \setminus X$ into $\Gamma$. 
	Notice that for each node $P_j \in \Gamma$ with $\eps(P_j) \leq \pcBase^\nu$ there is a node $Q_i \in \Delta$ with $\eps(Q_i) = \eps(P_j)$ and $i \in X$; moreover, $i$ can be computed via $\numI{i} = j$. Therefore, we can compare values of nodes $P_j$ with $Q_k$ by comparing the respective successor markings in $\Delta$. 
	This allows us to insert the nodes $Q_k$ for $k\in \interval{0}{\nu} \setminus X$ at their correct position into $\Gamma$ and to define their successor markings on $\Gamma$. 
	Thus, we obtain a reduced power circuit $(\Gamma^\bullet, \delta^\bullet)$ such that for each $i\in \interval{0}{\nu} $ there is some node of evaluation $\pcBase^i$.

	Finally, to also have nodes of value $\pcBase^i$ for each $i\in \interval{\nu+1}{\mu} $, we compute the compact base-$\pcBase$ signed-digit representation of each such $i$ (note that we need only at most $\nu+1$ digits for that) and create a new node with the respective successor marking. 
	Again as the successor markings are already present in $(\Gamma^\bullet, \delta^\bullet)$, we can compare the new nodes with existing nodes in $\Gamma^\bullet$ and insert them at the right position if necessary (see \cite[Lemma 4.13]{MattesWCC23}, \proc{InsertNodes}). As above, this clearly is in $\uTC{0}$. 
\qed\end{proof}

In \cite{MattesWCC23} we proved that we can reduce a power circuit $(\Pi, \delta_{\Pi})$ using \Tc{} circuits of depth $\Oh(\depth(\Pi))$ (\PTc{0}{D} parametrized by $\depth(\Pi)$). 
 In this paper, the non-reduced  power circuits are   ``almost reduced'': $\Pi=(\Gamma\cup\Xi, \delta)$ with $\Gamma\cap \Xi=\emptyset$ and $(\Gamma, \delta_{\Gamma})$ is a reduced power circuit. If $P \in \Xi$, then $\Lambda_P$ is a compact marking on $\Gamma$. Reducing such a power circuit is possible in $\uTC{0}$. 
 We will adapt the operations on reduced power circuits like addition or multiplication by a power of $q$ such that we only obtain almost reduced power circuits as intermediate results and include their reduction in the construction of the new markings.

We borrow the following two lemmas from \cite{MattesWCC23}. While \cite{MattesWCC23} treats only power circuits to base 2, here we allow an arbitrary base $\pcBase \geq 2$. On the other hand, we state the first lemma for a special case where the power circuits are already ``almost'' reduced. Its proof is actually verbatim the same for $q\geq2$. For the second Lemma we indicate the small differences in the proof. In \cite{MattesWCC23} both Lemmas are stated in the non-uniform case. But for the proof of the first one we only need  comparison of compact markings, sorting and counting.  These are all possible in $\uTC{0}$. 
\begin{lemma}[{\normalfont\stepone{}\bfseries}, {\cite[Lemma~4.14]{MattesWCC23}}]
\label{stz:insertminU}
The following is in $\uTC{0}$: 
\compproblem{A power circuit  $(\Gamma \cup \Xi, \delta)$ as above.}
{ A reduced power circuit $(\Gamma', \delta')$  such that
\begin{itemize}
	\item for each $ Q \in \Xi$ there is a node $P \in \Gamma'$ with $\epsilon(P) = \epsilon(Q)$,
	\item $(\Gamma, \delta|_{\Gamma \times \Gamma}) \leq (\Gamma', \delta')$,
	\item $\abs{\Gamma'}\leq \abs{\Gamma}+\abs{\Xi}$, and 
	\item $\abs{\mathcal{C}_{\Gamma'}}\leq \abs{\mathcal{C}_{\Gamma}}+\abs{\Xi}$.
\end{itemize} 
}
\end{lemma}

\newcommand{\prolongate}{\mu}
\begin{lemma}[{\normalfont\steptwo{}\bfseries}, {\cite[Lemma~4.15]{MattesWCC23}}] \label{stz:extendChains}
	The following is in $\uTC{0}$:
	\compproblem{A reduced power circuit $(\Gamma, \delta)$ and $\mu \in \mathbb{N}$ given in unary.}{ A reduced power circuit $(\Gamma',\delta')$ such that 
		\begin{itemize}
			\item for each $P \in \Gamma$ and each $i \in \interval{0}{\prolongate} $ there is a node $Q \in \Gamma'$ with $\epsilon(\Lambda_{Q}) = \epsilon(\Lambda_{P})+i$,
			\item $(\Gamma, \delta) \leq (\Gamma',\delta')$, 
			\item $\abssmall{\Gamma'} \leq \abs{\Gamma}+\abs{\mathcal{C}_{\Gamma}}\cdot\prolongate$, and
			\item $\abssmall{\mathcal{C}_{\Gamma'}} \leq \abs{\mathcal{C}_{\Gamma}} $.
	\end{itemize}}
\end{lemma}
Since we changed the statement slightly, we indicate the differences in the proof.
In \cite[Lemma~4.15]{MattesWCC23} we still had the technical condition that \small$\mu \leq \floor{\frac{2^{\abs{C_{0}(\Gamma)}+1}}{3}}  $ \normalsize (compare to \cref{lem:maxCQsdr} for $\base=2$). By using \cref{lem: longchain} we do not need this condition anymore. Moreover, in \cite{MattesWCC23} it was stated in the non-uniform case. 

\begin{proof}
We first prolongate $C_0$ to a new chain $\tilde C_0$ such that the last $\mu$ nodes of  $\tilde C_0$ are not already present in $\Gamma$. This can be done \eg by applying \cref{lem: longchain} to prolongate $C_0$ to length $i$ for all $i \in \interval{1}{(\mu + 1)\cdot \abs{\Gamma}}$ and checking the number of newly introduced nodes at the end of the chain. This is possible in $\uTC{0}$ because we  construct polynomially many power circuits of polynomial size. For each chain, we either introduced $\mu$ new nodes at the end or we merged two maximal chains. We choose $i$ minimal such that there are at least $\mu$ new nodes at the end of the new $C_0$ and discard the other circuits. 
	
 This replaces Step 1 in the proof of \cite[Lemma~4.15]{MattesWCC23}. Now, we can proceed exactly as in Step 2 of this proof~-- except that we do not prolongate the chain $\tilde C_0$ any more. Note that the operations we need like addition, comparison and making a \sdr compact are all in $\uTC{0}$.

For compact markings $L,M$ on $\Gamma$, it should be clear that  $\cor(\epsilon(M|_{C_0})+\mu)$ can be represented as a compact marking on $\tilde C_0$; thus, using \comp (\cref{lem:compareCompactMarkings}) we can check whether $\epsilon(L)\leq \epsilon(M)+\mu$ in \uAC{0} and create new nodes with successor markings of value $\epsilon(M)+i$ for $i \in \interval{1}{\mu}$ if necessary. 
\qed\end{proof}

On input of a \redpc for $k,\ell$ we want to construct a \redpc for $m=k+\ell$ or $m=k \cdot q^{\ell}$. If we proceed as  in  \cite[Lemma~4.9]{MattesWCC23}, the power circuit representing $m$ will not be (almost) reduced in general, even if we start with a reduced power circuit.  We adapt these operations such that we  obtain (almost) reduced power circuits as intermediate results. Thus the construction of a \redpc for $m$ is in $\uTC{0}$.  
Moreover, we show that we can do several of these operations on the \emph{same} power circuit in parallel, which will be crucial in our applications. 
\begin{lemma}[\normalfont\addition]\label{lem: add}
The following is possible in $\uTC{0}$: 
\compproblem{A reduced power circuit $(\Gamma, \delta)$ with compact markings $L^{(i)}_j $ on $\Gamma$ for $i\in \interval{1}{\ell}$, $j\in \interval{1}{k}$.   } 
{A reduced power circuit $(\Gamma', \delta')$ and compact markings $M^{(i)}$ on $\Gamma'$  with  $\epsilon(M^{(i)})=\epsilon(L^{(i)}_1)+\cdots + \epsilon(L^{(i)}_{k})$ for $i\in \interval{1}{\ell}$ and
\begin{itemize}
 \item $(\Gamma, \delta)  \leq (\Gamma', \delta')$,
\item $\abs{\Gamma'} \leq \abs{\Gamma}+ \ceil{\log_\pcBase(k)}\cdot \abs{\mathcal{C}_{\Gamma}} $,
 \item $\abs{\mathcal{C}_{\Gamma'}}\leq\abs{\mathcal{C}_{\Gamma}}$,
 \item $\abssmall{\sigma(M^{(i)})}\leq \sum_{j=1}^{k}\abssmall{\sigma(L^{(i)}_j)}$ for each $i \in \interval{1}{\ell}$.
\end{itemize}
 }%
\end{lemma}
\begin{proof}
We first apply $\steptwo(\ceil{\log_\pcBase(k)})$. We obtain a reduced power circuit $(\Gamma', \delta')$ containing the markings $L^{(i)}_j$. Let $C'$ be a maximal chain in $(\Gamma', \delta')$. Then the following is well-defined: 
\begin{align}\label{RedAdd}
\binM{M^{(i)}}{C'}=\cor\left(\sum_{j=1}^{k}\binM{L^{(i)}_j}{C'}\right).
\end{align}
Note that the last $\ceil{\log_\pcBase(k)}$ nodes of $C'$ are not marked by any $L^{(i)}_j$. Therefore, by \cref{lem:maxCQsdr}, we have \[\sum_{j=1}^{k}\val(\binM{L^{(i)}_j}{C'}) \leq k \cdot \floor{\frac{\base^{\abssmall{C'}-\ceil{\log_\pcBase(k)}+1}}{\base^2-1}} \leq \floor{\frac{\base^{\abssmall{C'}+1}}{\base^2-1}}.\] 
Therefore, the right side of \eqref{RedAdd}, indeed, is a compact representation using only $\abs{C'}$ digits.

 Finally, if we define $M^{(i)}$ like this on all maximal chains, it is clear that $M^{(i)}$ is a compact marking and that $\epsilon(M^{(i)})=\epsilon(L^{(i)}_1)+\cdots +\epsilon(L^{(i)}_{k})$.  
It is also clear that $\abssmall{\sigma(M^{(i)})}\leq \sum_{j=1}^{k}\abssmall{\sigma(L^{(i)}_j)}$.  \cref{stz:extendChains} shows that $(\Gamma, \delta)\leq (\Gamma', \delta')$ as well as the size conditions of \cref{lem: add}. Because base-$q$ iterated addition is in $\uTC{0}$ (see \cref{ex:itAdd}), with \cref{thm:makeCompact} and \cref{stz:extendChains} it follows that the right side in (\ref{RedAdd}) can be constructed in $\uTC{0}$.  This shows the lemma. 
\qed\end{proof}

\begin{lemma}\label{lem: powtwo}\label{lem: floatpOp}%
	The following is possible in $\uTC{0}$:
	\begin{enumerate}
	\vspace*{-1.2mm}	\item\label{powtwo} \powertwo:
			\compproblem{A reduced power circuit $(\Gamma, \delta)$ with compact markings $K^{(i)}, L^{(i)} $ on $\Gamma$ for $i \in \interval{1}{\ell}$ such that $\epsilon(K^{(i)})\cdot\pcBase^{\epsilon(L^{(i)})} \in \Z$.}
			{
				A reduced power circuit $(\Gamma', \delta')$ with compact markings $M^{(i)}$ such that $\epsilon(M^{(i)})=\epsilon(K^{(i)})\cdot\pcBase^{\epsilon(L^{(i)})}$ for  $i \in \interval{1}{\ell}$.
			}

		\item \label{MaxPow} \FloatPoint:
			\compproblem{A reduced power circuit $(\Gamma, \delta)$ with  compact markings $K^{(i)}$ for $i \in \interval{1}{\ell}$.}{A reduced power circuit $(\Gamma', \delta')$ with compact markings $U^{(i)}$, $E^{(i)}$ such that $\epsilon(K^{(i)})=\epsilon(U^{(i)})\cdot \pcBase^{\epsilon(E^{(i)})}$ with $\epsilon(U^{(i)}) = 0$ or $q \nmid  \epsilon(U^{(i)})$ for  $i \in \interval{1}{\ell}$.  
		}

	\end{enumerate}
\medskip

\noindent We have that $(\Gamma, \delta) \leq (\Gamma', \delta')$ and
\vspace*{-1.2mm}\begin{itemize} 
	\item $\abs{\Gamma'} \leq  \abs{\Gamma}+ \abs{\mathcal{C}_{\Gamma}} + \sum_{i=1}^{\ell} \abs{\sigma(K^{(i)})} $
	\item $\abs{\mathcal{C}_{\Gamma'}} \leq \abs{\mathcal{C}_{\Gamma}}+\sum_{i=1}^{\ell} \abs{\sigma(K^{(i)})}$
	\item $\abs{\sigma(M^{(i)})}= \abs{\sigma(K^{(i)})}$. i 
\end{itemize}
Notice that the  size of $\Gamma'$ and $\abs{\sigma(M^{(i)})}$ does not depend on $L^{(i)}$. 
\end{lemma}

\begin{proof}
	\newcommand{\rip}{R^{(i)}_{P}}
	\newcommand{\riq}{R^{(i)}_{Q}}
	\powertwo:	We start by applying $\steptwo(1)$. We denote the resulting reduced power circuit by $(\Gamma_1, \delta_1)$. Observe that $(\Gamma, \delta)\leq(\Gamma_1, \delta_1)$. Next, we apply a construction similar to the one in \cite[Lemma~4.9]{MattesWCC23}. For each $ i \in \interval{1}{\ell}$ and each node $P \in \sigma(K^{(i)})$ we construct a node $\rip$ as follows: Let $C$ be a maximal chain in $\Gamma_1$. As the last node of $C$ is neither marked by $\Lambda_{P}$  nor by $L^{(i)}$, by \cref{lem:maxCQsdr}, $\cor\left(\binM{\Lambda_{P}}{C}+\binM{L^{(i)}}{C}\right)$ uses at most $\abs{C}$ digits.
Thus,
	\[\binM{\Lambda_{\rip}}{C}=
	\cor\left(\binM{\Lambda_{P}}{C}+\binM{L^{(i)}}{C}\right)\] 
	is well-defined and can be computed in $\uTC{0}$ by \cref{thm:makeCompact}.
	We take the marking $\Lambda_{\rip}$ defined like that as the successor marking of $\rip$. Then 
	\[\epsilon(\rip) = \pcBase^{\epsilon(\Lambda_{\smash{\rip}})} = \pcBase^{\epsilon(\Lambda_P)+\epsilon(L^{(i)})}=\epsilon(P) \cdot \pcBase^{\epsilon(L^{(i)})}.\]

	We obtain a (not necessarily reduced) power circuit $(\Gamma_1 \cup \Xi, \delta)$, containing all the newly constructed nodes, with 
		\[\Xi = \set{\rip}{i \in \interval{1}{\ell}, P \in \sigma(K^{(i)})}.\]
	 All markings $\Lambda_{\rip}$ are on  $\Gamma_1$ and compact by construction. 
	Observe that $\abs{\Xi}=\sum_{i=1}^{\ell}\abssmall{\sigma(K^{(i)})}$.  
	We apply \stepone with input of $(\Gamma_1 \cup \Xi, \delta)$. By \cref{stz:extendChains} and \cref{stz:insertminU} the construction of the resulting reduced power circuit $(\Gamma', \delta')$ is possible in $\uTC{0}$ such that $(\Gamma, \delta)\leq (\Gamma', \delta')$ and such that the size conditions on $\abssmall{\Gamma'}$ and $\abssmall{\mathcal{C}_{\Gamma'}}$ in the lemma are satisfied. While applying \stepone, for each node $\rip$ we remember which node in $\Gamma'$ has the same evaluation. 
	\medskip
	
	Now we need to define the markings $M^{(i)}$ on $\Gamma'$. Let $Q \in \Gamma'$. If there exists a node $P \in \sigma(K^{(i)})$ such that $\epsilon(Q)=\epsilon(\rip)$ (\ie $\Lambda_Q=\Lambda_{\rip}$), then we set $M^{(i)}(Q)=K^{(i)}(P)$. Otherwise, we set $M^{(i)}(Q)=0$. 
	
	As $\epsilon(\rip) =  \epsilon(P) \cdot \pcBase^{\epsilon(L^{(i)})}$ for all $P \in \supp(K^{(i)})$, the marking $M^{(i)}$ is just a ``shift'' of $K^{(i)}$; thus, it is well-defined and compact.
	 Observe that $\abssmall{\sigma(M^{(i)})}=\abssmall{\sigma(K^{(i)})}$. We obtain that 
	\[\epsilon(M^{(i)})=\sum_{Q \in \Gamma'}M^{(i)}(Q)\cdot \epsilon(Q)=\sum_{P \in \sigma(K^{(i)})}K^{(i)}(P)\cdot \epsilon(P)\cdot\pcBase^{\epsilon(L^{(i)})}= \epsilon(K^{(i)})\cdot \pcBase^{\epsilon(L^{(i)})}. \]
	This proves part \ref{powtwo}.

	\paragraph{\FloatPoint:}
	Let $\sigma(K^{(i)})=\oneset{Q_1, \ldots, Q_{k} }$.  If $k=0$, then $\epsilon(K^{(i)})=0$, so we set $\epsilon(U^{(i)})=\epsilon(E^{(i)})=0$.   Now let $k \geq 1$. Because $(\Gamma, \delta)$ is reduced, we can assume that $\epsilon(Q_1)< \epsilon(Q_j)$ for all $j \in \interval{2}{k}$. Therefore, $u = \epsilon(K^{(i)}) \cdot \pcBase^{-\epsilon(\Lambda_{Q_1})}$ is integral but not divisible by $q$. We set  $E^{(i)}=\Lambda_{Q_1}$ and use \powertwo with input $K^{(i)}$ and $-E^{(i)}$ to  compute a marking $U^{(i)}$ with  $\eps(U^{(i)}) = u$. By the first part of the lemma, we can do this for all $i \in \interval{1}{\ell}$ in parallel.
\qed\end{proof}

\begin{lemma}\label{lem:AddMult}\
\begin{enumerate}[(a)]
	\item\label{SumNotInZ}  The following is in $\uTC{0}$:
	\ynproblem{A reduced power circuit $(\Gamma, \delta)$ with compact markings $K_j, L_j $ for $j \in \interval{1}{k}$.}{
		Is $\sum_{j=1}^{k}\epsilon(K_j)\cdot\pcBase^{\epsilon(L_j)} \in \Z$ ? 
	}
	
	\item \label{SumInZ} The following is in $\uTC{0}$:
	\compproblem{A reduced power circuit $(\Gamma, \delta)$ with compact markings $K_j^{(i)}, L_j^{(i)} $ for $i \in \interval{1}{\ell}$, $j \in \interval{1}{k}$ such that $\sum_{j=1}^{k}\epsilon(K_j^{(i)})\cdot\pcBase^{\epsilon(L_j^{(i)})} \in \Z$.}{	A reduced power circuit $(\Gamma', \delta')$ with compact markings $M^{(i)}$ such that $\epsilon(M^{(i)})=\sum_{j=1}^{k}\epsilon(K_j^{(i)})\cdot\pcBase^{\epsilon(L_j^{(i)})}$ for  $i \in \interval{1}{\ell}$. There is a constant $c$ such that $(\Gamma, \delta) \leq (\Gamma' \delta')$ and 
		\begin{itemize} 
			\item $\abs{\Gamma'} \leq  \abs{\Gamma}+ c \cdot \log_{q}(k) \cdot (\abs{\mathcal{C}_{\Gamma}} + \sum_{j=1}^{k}\sum_{i=1}^{\ell} \abssmall{\sigma(K_j^{(i)})}) $
			\item $\abs{\mathcal{C}_{\Gamma'}} \leq \abs{\mathcal{C}_{\Gamma}}+c \cdot \sum_{j=1}^{k}\sum_{i=1}^{\ell} \abssmall{\sigma(K_j^{(i)})}$
			\item $\abssmall{\sigma(M^{(i)})}\leq \sum_{j=1}^{k}\abssmall{\sigma(K_j^{(i)})}$. \end{itemize}	
	}

\end{enumerate}	
\end{lemma}	
\begin{proof}
For both parts, we essentially use the same computation: The following can be done for each $i \in \interval{1}{\ell}$ in parallel: We define $j_0$ such that $\eps(L^{(i)}_{j_0})=\min\{\eps(L^{(i)}_j)~|~1\leq j \leq k\}	$. Then, $\eps(L_j^{(i)})-\eps(L_{j_0}^{(i)})\geq 0$ for all $j \in \interval{1}{k}$.
 In particular, $\epsilon(K_j^{(i)})\cdot\pcBase^{\epsilon(L_j^{(i)})-\eps(L_{j_0}^{(i)})} \in \Z$ for all $j$.
 So, using \cref{lem: add} and \cref{lem: powtwo}, we can construct markings $N_j^{(i)}$ such that $\eps(N_j^{(i)})=\epsilon(K_j^{(i)})\cdot\pcBase^{\epsilon(L_j^{(i)})-\eps(L_{j_0}^{(i)})}$ for $j \in \interval{1}{k}$. Then, using \cref{lem: add} we can construct a compact marking evaluating to $\sum_{j=1}^{k}\eps(N_j^{(i)})$.

 For Part (\ref{SumNotInZ}) we now proceed as follows: If $\eps(L_{j_0})\geq 0$ we output ``yes''. Otherwise we can check if $q^{-\eps(L_{j_0})} \mid \sum_{j=1}^{k}\eps(N_j)$ using \cref{lem: powtwo}. If yes, we output ``yes'', otherwise we output ``no''.

Part (\ref{SumInZ}):  Since $q^{\eps(L_{j_0}^{(i)})}\cdot \sum_{j=1}^{k}\eps(N_j^{(i)}) \in \Z$ by  assumption, we can construct the marking $M^{(i)}$ with one more application of \cref{lem: powtwo}.
Note that each sum has length at most $k$ and each marking $K$ is  multiplied at most once by a power of $q$. Thus, the size constraints follow by \cref{lem: add} and \cref{lem: powtwo}. Since we use constantly many layers of \addition and \powertwo, it follows that the above algorithm is in $\uTC{0}$.  This shows Part (\ref{SumInZ}).\qed
\end{proof}

\begin{lemma}\label{lem: log}
	The following is in $\uTC{0}$:
	\compproblem{Markings $K,L$  in a reduced power circuit $(\Gamma, \delta)$ such that $\frac{\eps(K)}{\eps(L)}$ is a power of $q$.}{A marking $M$ in a reduced power circuit $(\Gamma', \delta')$ such that $\epsilon(M) = \log_q(\frac{\eps(K)}{\eps(L)})$, $\abs{\Gamma'} \leq \abs{\Gamma}+  \abs{\mathcal{C}_{\Gamma}}$,  $\abs{\mathcal{C}_{\Gamma'}}\leq \abs{\mathcal{C}_{\Gamma} }$. }
	\end{lemma}	
\begin{proof}
		Let $\sigma(K)=\oneset{Q_1, \ldots, Q_{k} }$ and $\sigma(L)=\oneset{R_1, \ldots, R_{\ell} }$. Because $(\Gamma, \delta)$ is reduced, we can assume that $\epsilon(Q_1)< \epsilon(Q_j)$ and $\epsilon(R_1)< \epsilon(R_j)$ for all $j \neq 1$. Then there are markings $U,V$ such that $\eps(K)=\eps(U)\cdot q^{\eps(\Lambda_{Q_1})}$, $\eps(L)=\eps(V)\cdot q^{\eps(\Lambda_{R_1})}$ and $q$ does not divide $\eps(U)$ and $\eps(V)$. Because $\frac{\eps(K)}{\eps(L)}$ is a power of $q$, $\eps(U)=\eps(V)$. Thus, $\eps(M)=\eps(\Lambda_{Q_1})-\eps(\Lambda_{R_1})$.  Now  \cref{lem: add} finishes the proof of the lemma.
	\qed\end{proof}

\begin{lemma}\label{lem:moduloPowerOfTwoConstant}
	Let $r \in \Z$ be a constant. The following problem is in $\uTC{0}$:
	\compproblem{
		A reduced power circuit $(\Gamma, \delta)$ with compact markings $K, L $.
	}
	{
		A reduced power circuit $(\Gamma', \delta')$ with a compact marking $M$ such that $\epsilon(M)=\epsilon(L) \bmod  r \cdot \pcBase^{\epsilon(K)}$.
	}
\end{lemma}

\begin{proof}
	We write $r=s\cdot x$ with $\gcd(s,q)=1$ and such that $s$ is maximal with this property. Then, $x ~|~q^a$ for some $a \geq 0$. Let $a\geq0$ be minimal such that $x ~|~ q^a$. Because $r$ is a constant, $s$, $x$, and $a$ are constants. 
Now, $s$ and $x \cdot q^{\eps(K)}$ are coprime; thus, by the Chinese Remainder Theorem, we may compute $\epsilon(L) \bmod  x \cdot\pcBase^{\eps(K)}$ and  $\epsilon(L) \bmod  s$ separately and then compose them.	

We start by computing $\epsilon(L) \bmod\pcBase^{a+\epsilon(K)}$ and then use the result to obtain $\epsilon(L) \bmod x \cdot\pcBase^{\epsilon(K)}$. 
To do so, for all $P \in \supp(L)$ we check  whether $\eps(\Lambda_{P}) < a+\eps(K)$ and create a new marking $L_a$ by setting $L_a(P) = L(P)$ if  $\eps(\Lambda_{P}) < a+\eps(K)$ and $L_a(P) = 0$ otherwise. By \cref{lem: longchain} and \cref{lem: add} we can construct a marking representing $a+\eps(K)$ in $\uTC{0}$. Also the comparison is in $\uTC{0}$ by \cref{lem:compareCompactMarkings}. Because $L$ is a compact marking in a reduced power circuit, by \cref{lem:maxCQsdr}, we obtain that  $\abs{\eps(L_a)}<  q^{\eps(K)+a}$. If $\eps(L_a) <0$, we set $\ell_a= q^{\eps(K)+a}+\eps(L_a)$, otherwise we set $\ell_a=\eps(L_a)$.
 Now, we subtract $x\cdot q^{\eps(K)}$ from $\ell_a$ until the result is smaller than $x\cdot q^{\eps(K)}$. This happens after a constant number $c$ of steps~-- hence, it can be done sequentially in a $\TC$ circuit.  Let $L_k$ be a marking with $\eps(L_k) = \ell_K = \ell_a- c\cdot x\cdot q^{\eps(K)}$. Then,  
 \[\ell_K \equiv \eps(L_a) \equiv \eps(L) \mod x \cdot q^{\eps(K)}.\]

For computing $\eps(P) \bmod s$, we proceed by induction:
As $\gcd(s, \pcBase) = 1$, we know that $\pcBase^{\phi(s)} \equiv 1 \mod s$, where $\phi$ is Euler's totient function. Therefore, 
\[\eps(P) = \pcBase^{\eps(\Lambda_{P})}  \equiv \pcBase^{\eps(\Lambda_{P}) \bmod \phi(s)} \mod s.\]
By induction, we know that $\eps(\Lambda_{P}) \bmod \phi(s)$ can be computed in $\uTC{0}$ (note that $\phi(s)$ might be non-coprime to $q$ anymore, but $\phi(s)<r$ is a constant, so we only compose a constant number of $\uTC{0}$ computations, which results again in a $\uTC{0}$ computation). Hence, we can compute $\ell'_s=\sum \left( L(P)\cdot \eps(P) \bmod s \right) $ in $\uTC{0}$ and observe that $\ell'_s < s \cdot q \cdot  \abs{\Pi}$. Therefore, to compute $\ell_s=\eps(L) \bmod s=(\sum L(P)\cdot \eps(P) \bmod s) \bmod s$, we just can subtract $k\cdot s $ from $\ell'_s$ for all $k \in \interval{1}{q \cdot \abs{\Pi}}$ and check if $0 \leq \ell'_s-k\cdot s < s$.

Now for $\epsilon(M)$ there remain the possibilities $m_j =  \ell_K + j\cdot x \cdot \pcBase^{\epsilon(K)}$ for $j\in \interval{0}{s-1}$. For each of these $j$ we can compute $m_j \bmod  s$ in $\uTC{0}$ as above and choose $j_0$ such that $m_{j_0} \bmod  s = \ell_s $. By \cref{lem: add}, we now can compute the marking $M$ with $\eps(M) = \ell_K + j_0\cdot  x\cdot\pcBase^{\epsilon(K)}$ in $\uTC{0}$. 
\qed\end{proof}

\begin{remark}
	From the proof of \cref{lem:moduloPowerOfTwoConstant} it follows that actually  when allowing $k$ in unary as part of the input, the problem
	\compproblem{A reduced power circuit $(\Gamma, \delta)$, a marking $M$ on $\Gamma$ and $k \in \N$ given in unary.}{$\eps(M) \bmod k$.}
	\noindent	is in $\uTC{1}$. In order to see this, we apply the algorithm from \cref{lem:moduloPowerOfTwoConstant} with $k=r$ and $a=0$. In order to compute $\phi(s)$, we factorize $s$ (by trying all possible factors) and then compute it via the formula $\phi(mn) = \phi(m) \phi(n)$ for coprime $m$ and $n$. Remember that integer division and multiplication is in $\uTC{0}$  \cite{hesse01,HeAlBa02}. Notice that after a logarithmic number of applications of $\phi$, we arrive at $\phi(s) = 1$. This is because at least in every second step $s$ will be even~-- and if $\ell$ is even we have $\phi(\ell) \leq \ell/2$. As each round in the proof of \cref{lem:moduloPowerOfTwoConstant}  is in $\uTC{0}$, the complete algorithm is in $\uTC{1}$.
\end{remark}

\section{The word problem for Baumslag-Solitar groups} \label{sec:BSgroups}

\newcommand{\DelBS}{\Delta_ \mathbf{B}}

Let $p,r \in \Z$. 
The Baumslag-Solitar group $\BS{p}{r}$ is defined by 
\begin{align*}
	\BS{p}{r}& = \Gen{a,t}{ta^pt^{-1} = a^r}.
\end{align*}
It is an HNN-extension of $\mathbb{Z}=\gen{a}$ with stable letter $t$ and $\phi: \gen{a^p} \rightarrow \gen{a^{r}}$.

Write $\DelBS$ for the alphabet $\Delta$ associated to the HNN extension as defined in \cref{sec:GroupTheory} meaning that $\DelBS = \{t,t^{-1}\} \cup \{a^k\mid k\in \Z \setminus \{0\}\}$. We write $\mathbf{B}=\BS{p}{r}$ as abbreviation.

Because we want to consider a more compressed form of the word problem (see Theorem~\ref{thm:WPinBS}) in this work we represent the elements of  
 $\BS{p}{r}$ as words over the alphabet
\[\Theta=\set{a^k, t^x}{k, x \in \mathbb{Z} \setminus\{0\}}.\] 
We use $a^0$ and $t^0$ as aliases for the empty word.
Thus, words $w \in \Thepr$ are of the form 
\begin{align}\label{wBRBS}
w =a^{k_0}t^{x_1}a^{k_1}t^{x_2}\cdots  t^{x_n}a^{k_n}
\end{align}
with $k_i, x_i \in \mathbb{Z}$.  Recall that $\len{\Theta}{w}$ denotes the length of $w$ as word over $\Theta$~-- this means that we count the $a^k,t^x$ for $k,x \neq 0$.

Note that the alphabet $\Theta$ is larger than the alphabet $\DelBS$. Still any word in $\Theta^*$ can be interpreted over the alphabet $\DelBS$  by the projection
\begin{align*}
\pi_\Theta: \Theta^* \to \DelBS^*, \quad a^k \mapsto a^k,\: t^x \mapsto \underbrace{t\cdots t}_{x \text{ times}}
\end{align*}
meaning that we write a letter $t^x$ as $x$ single letters $t$ (resp.\ $-x$ single letters $t^{-1}$ if $x<0$). Recall that the notation $v =_{\DelBS} w$ means $\pi_\Theta(v) = \pi_\Theta(w)$ for $v,w \in \Theta^*$.

Thus, the notion of Britton reductions and Britton's Lemma immediately can be applied to words over $\Theta$: we call $w \in \Theta^*$ Britton-reduced if and only if $\pi_\Theta(w)$ is.
Recall that a word $w \in \DelBS^*$ is called \emph{Britton-reduced} if none of the Britton reduction rules
\begin{align}\label{BRinBS}
\nonumber	ta^kt^{-1} &\:\to\: a^{k\cdot (\frac{r}{p})} \qquad \text{ if } k \in p\Z \\ 
	t^{-1} a^kt &\:\to\:  a^{k\cdot (\frac{p}{r})} \qquad \text{ if } k \in r\Z\\ \nonumber
	a^ka^\ell &\:\to\: a^{k+\ell}
\end{align}
can be applied in order to obtain a shorter word (over the alphabet $\DelBS$) representing the same group element (recall that $a^0$ is identified with the empty word).
.
For Baumslag-Solitar groups, Britton's Lemma can be rephrased as:

\begin{lemma}\label{rmk:Britton}
	If $w \in \DelBS^*$ with  $w  \inBSpr \gen{a}$, $w = ut^\eps v$  for $\eps \in \{\pm 1\}$ and $u,v \in \DelBS^*$, then 
	\begin{itemize}\vspace{-3mm}
		\item either $u = u_1t^{-\eps}u_2$ for $u_1, u_2 \in \DelBS^*$ and $t^{-\eps}u_2t^\eps \inBSpr \gen{a}$ and $u_2 \inBSpr \gen{a}$ or
		\item $v = v_1t^{-\eps}v_2$ for $v_1, v_2 \in \DelBS^*$ and $t^\eps v_1t^{-\eps} \inBSpr \gen{a}$ and $v_1 \inBSpr \gen{a}$.
	\end{itemize}
\end{lemma}

From now on, we mainly consider words over $\Theta$ of the form 
\begin{align}\label{wBRBS}
w = a^{k_0}t^{x_1}a^{k_1}t^{x_2}\cdots  t^{x_n}a^{k_n}
\end{align}
with $x_i \in \mathbb{Z}\setminus\{0\}$ (we still allow $a^0$ as abbreviation for the empty word). 
 Further, w.l.o.g. we can assume that $p\geq 1$ since $ta^pt^{-1}\eqBSpr a^r$ is equivalent to $ta^{-p}t^{-1}\eqBSpr a^{-r}$. 
 
Observe that writing words over $\Theta$ is not unique in the sense that we have $a^k$ as abbreviation for $a \cdots a$ ($k$-times) and as the letter $a^k \in \Theta$. In the following, we will only use the notation $a^k$ for the letter $a^k$ (and the same for $t$).  

\paragraph{Solvable Baumslag-Solitar groups.} In \cref{sec: ConjFixedEl} below we will use another way to represent the group  $\BSq$ (in the special case that $p=1$):  $\BSq \cong \Z[1/\pcBase] \rtimes \Z$ via the isomorphism $a\mapsto (1,0)$ and $t\mapsto (0,1)$. Here, $\rtimes$ denotes the \emph{semi-direct product} and $\Z[1/\pcBase] = \set{m/\pcBase^k\in \Q}{m,k \in \Z}$  is the set of fractions with powers of $\pcBase$ as denominators. It forms a group with addition as group operation. 
The multiplication in $\Z[1/\pcBase] \rtimes \Z$ is
defined by  $(r,m) \cdot (s,n) = (r + \pcBase^m \cdot s, m + n)$. 
Inverses can be computed by the formula $(r,m)^{-1} = (-r \cdot \pcBase^{-m}, -m)$.
Because $ta=_{\BSq} a^qt$, we can write $g \in \BSq$ as $g=t^ka^rt^m$ with $k,r,m \in \Z$ and $q \nmid r$. The element $t^ka^rt^m$ corresponds to $(q^kr,m+k)$ in the semi-direct product; it is Britton-reduced whenever $k\leq 0$.  Given an element $(s,m) \in \sdZ$ (represented by a power circuit) we can find $k$ using \cref{lem: floatpOp}. 
\medskip

From now on, we consider groups $\BS{p}{pq}$, meaning that $r=pq$ for some $q \in \Z$. We keep the abbreviation  $\mathbf{B}=\BS{p}{pq}$ and the other notations of this section.

\subsection{\Taligned words}

As in \cite{Weiss16} the idea to show that the word problem of $\BS{p}{pq}$ is in $\uAcf{0}$ is to define an equivalence relation on the $t$-blocks of $w$ to reduce the subgroup membership problem for $\gen{a}$ in $\BS{p}{pq}$ to the word problem of the free group.
Because our input words are over $\Theta$, we would like to define Britton reductions directly for words over $\Theta$ meaning that each letter $t^x$ should cancel with some letter $t^{-x}$ (see \cref{lem: CancelCompl} and \cref{rem:BRonTheta} and  below).
 This is a property of the \emph{\taligned} words we introduce in this section. 

\begin{definition}\label{def: VersionsW}
    Let $w=a^{k_0}t^{x_1}a^{k_1}t^{x_2}\cdots  t^{x_n}a^{k_n}\in \Theta^* $ and $i\leq j$. We define 
    \begin{enumerate}[(a)]
        \item $\wcl{i}{j} =a^{k_i}t^{x_{i+1}}a^{k_{i+1}}\cdots t^{x_j}a^{k_j} $
        \item $\wopi{i}{j} =t^{x_{i+1}}a^{k_{i+1}}\cdots t^{x_j}a^{k_j}$
%        \item $\wopj{i}{j} =a^{k_i}t^{x_{i+1}}a^{k_{i+1}}\cdots t^{x_j}$
        \item $\wopij{i}{j}= t^{x_{i+1}}a^{k_{i+1}}\cdots t^{x_j}$
    \end{enumerate}
     \end{definition}
In particular, we have $w=\wcl{0}{n}$. %, so we can write each word $w \in \Thepr$ as $w_I$ for an interval $I$ as in \cref{def: VersionsW}.
     \begin{example}
The following picture illustrates the subwords of $w=a^{k_0}t^{x_1}a^{k_1}t^{x_2}\cdots  t^{x_n}a^{k_n}\in \Theta^* $ defined by the different types of intervals:  $w_{\textcolor{green}{(n-1,n)}}=t^{x_n}$, $w_{\textcolor{blue}{[1,2]}}=a^{k_1}t^{x_2}a^{k_2}$ and $w_{\textcolor{red}{(0,n]}}=t^{x_1}a^{k_2}\cdots t^{x_n}a^{k_n}$. 
         
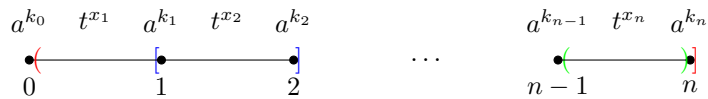
\begin{center}

\tikzstyle{pcnode} = [minimum size = 3pt,fill=black, circle, draw ]
\begin{tikzpicture}[outer sep = 0pt, inner sep = 0.7pt, node distance = 50]

\node[pcnode] (1) {};
\node[above of =1, yshift=-35] (7) {$a^{k_0}$};
\node[below of =1, yshift=40] (8) {$0$};
\node[pcnode, right of=1] (2) {};
\node[above of =2, yshift=-35] (9) {$a^{k_1}$};
\node[below of =2, yshift=40] (10) {$1$};
\node[pcnode, right of=2] (3) {};
\node[above of =3, yshift=-35] (11) {$a^{k_2}$};
\node[below of =3, yshift=40] (12) {$2$};
\node[right of=3] (4) {$\ldots$};
\node[pcnode, right of=4] (5) {};
\node[above of =5, yshift=-35] (13) {$a^{k_{n-1}}$};
\node[below of =5, yshift=40] (14) {$n-1$};
\node[pcnode, right of=5] (6) {};
\node[above of =6, yshift=-35] (15) {$a^{k_{n}}$};
\node[below of =6, yshift=40] (16) {$n$};
%\node[right of =7, xshift=-20] (17) {$t^{x_0}$};
%\node[right of =9, xshift=-20] (18) {$t^{x_1}$};
%\node[right of =13, xshift=-20] (19) {$t^{x_n}$};
\node[right of =4, xshift=3] (20) {\textcolor{green}{$($}};
\node[left of =6, xshift=48] (21) {\textcolor{green}{$)$}};
%\node[below of =17, yshift=35] (21) {\textcolor{green}{$)$}};
\node[left of =2, xshift=48] (22) {\textcolor{blue}{$[$}};
\node[right of=3, xshift=-48] (23) {\textcolor{blue}{$]$}};
\node[right of =6, xshift=-48] (24) {\textcolor{red}{$]$}};
%\node[below of=19, yshift=35, xshift=-8] (25) {\textcolor{red}{$($}};
\draw 

(1) edge  node[midway, xshift=-22]{\textcolor{red}{$($}} (2)
(2) edge  (3)
(5) edge  (6)
(7) edge[white]  node[midway, black]{$t^{x_1}$} (9)
(9) edge[white]  node[midway, black]{$t^{x_2}$} (11)
(13) edge[white]  node[midway, black]{$t^{x_n}$} (15)
;

\end{tikzpicture}
\end{center} 
     \end{example}

\begin{definition}\label{def: height}
Let $w=a^{k_0}t^{x_1}a^{k_1}\cdots t^{x_n}a^{k_n}$ and let $1\leq i,j\leq n$.
\begin{enumerate}[(a)]
\item We write $\mn{i}{j}=\min\{i,j\}$ and $\mx{i}{j}=\max\{i,j\}$.
\item We define $\height{w}{i,j}=\sum_{\mu=\mn{i}{j}+1}^{\mx{i}{j}}x_\mu$ and write $\height{w}{i}=\height{w}{0,i}$ and $\height{w}{0,n}= \hght{w}$ (meaning that $\height{w}{0}=0$).
\smallskip

\item We define $\up{w}{i}=\max\{\height{w}{i}, \height{w}{i-1}\}$ and $\low{w}{i}=\min\{\height{w}{i}, \height{w}{i-1}\}$. Then, $x_i=\sgn{x_i}\cdot \left(\up{w}{i}-\low{w}{i}\right)$.
\end{enumerate}
\end{definition}
In particular, $\height{w}{i,j}=\hght{\wcl{i}{j}}$.

\begin{definition}\label{def: Taligned} 
We call a word 
$w=a^{k_0}t^{x_1}a^{k_1}\cdots t^{x_n}a^{k_n} \in \Thepr$  
\taligned if for all $1\leq  i,j \leq n$ the following holds:
\begin{align*}%\label{deftal}
    \height{w}{j}\in \interval{\low{w}{i}}{\up{w}{i}}  \Longrightarrow \height{w}{j}\in \{\low{w}{i}, \up{w}{i} \}.
\end{align*}
\end{definition}

\begin{example}
Let $w=a^{k_0}t^{2}a^{k_1}t^{-1}a^{k_3}t^{1}a^{k_3}t^{-2}a^{k_4}$. The following picture shows the topography of $w$ if we only consider the $t$. The black dots represent the $\height{i}{w}$, $1\leq i \leq 4$ and the lines connecting them represent the $x_i$. 
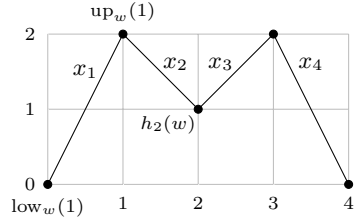
\begin{center}
 
\bigskip
\tikzstyle{pcnode} = [minimum size = 3pt,fill=black, circle, draw ]
\begin{tikzpicture}[outer sep = 0pt, inner sep = 0.7pt, node distance = 50]
  \draw[step=1cm, style=lightgray]
    (0,0) grid (4,2);
\foreach \y in {0,1,2}
	\draw[white] (0,\y) -- (0,\y) node[left=4pt, black] {$\scriptstyle\y$};
 \foreach\x in {1,2,3,4}
\draw[white] (\x,0) -- (\x,0) node[below=4pt, black] {$\scriptstyle\x$};
\node (1) at (0.5,1.5){\small $ x_1$};
\node (2) at (1,2.3){ $\scriptstyle\up{w}{1}$};
 \node (3) at (0,-0.3){ $\scriptstyle\low{w}{1}$};
 \node (4) at (1.6,0.8){ $\scriptstyle\height{w}{2}$};
 \node (5) at (1.7,1.6){\small $ x_2$};
  \node (6) at (2.3,1.6){\small $ x_3$};
   \node (7) at (3.5,1.6){\small $ x_4$};
 \draw
 (0,0) node[pcnode](0,0) {} --(1,2) node[pcnode] (1,2) {}-- (2,1)node[pcnode](2,1){}--(3,2)node[pcnode](3,2){}--(4,0)node[pcnode](4,0){};
\end{tikzpicture}	
\end{center}

We see that $\low{w}{1}< \height{w}{2}<\up{w}{1}$. So $w$ is an example of a word that is not \taligned.   
\end{example}

The following Lemma shows that for each word $w \in  \Thepr$ there indeed exists a \taligned word $\wal$ such that $\wal\eqBSpr w$. The proof of the lemma gives a construction of such a \taligned word $\wal$.

\begin{lemma}\label{lem: exTaligned} 
    For each word $w \in \Thepr$ there exists a \taligned word $\wal \in \Thepr$ such that $\wal=_{\DelBS} w$ and $\len{\Theta}{\wal} \leq(\len{\Theta}{w})^2$. 
\end{lemma}
\begin{proof}
Let $w=a^{k_0}t^{x_1}a^{k_1}\cdots t^{x_n}a^{k_n} \in \Thepr$. Let $\eta_1, \ldots, \eta_r \in \Z$ be such that $\{\eta_1, \ldots, \eta_r\}=\{\height{w}{i}, i \in \interval{0}{n} \}$ with $\eta_i < \eta_j$ if $i<j$. For each $i \in \interval{1}{n}$, $\low{w}{i},\up{w}{i} \in \{\eta_1, \ldots, \eta_r\}$. Because $\height{w}{0}=0$, also $ 0 \in \{\eta_1, \ldots, \eta_r\}$. 
For each $i \in \interval{1}{n}$, let $i_\ell$ and $i_u$ be such that $\eta_{i_\ell}=\low{w}{i}$ and $\eta_{i_u}=\up{w}{i}$. Note that $i_\mu > i_\ell$.  Then, $\eta_{i_u}-\eta_{i_\ell}=\abs{x_i}$. We define:
\begin{align}\label{defLambda}
\lambda^{(i)}_{\mu}&=\sgn{x_i}\cdot \left( \eta_{i_\ell+\mu}- \eta_{i_\ell+(\mu-1)} \right) ~ \text{ for } \mu \in \interval{1}{i_u-i_\ell}. 
% \text{ if }  x_i >0~~ (\lambda^{(i)}_{\mu}>0) \\ \nonumber
 %\lambda^{(i)}_{\mu}&=h_{i_u-\mu}- h_{i_u-(\mu-1)}  ~ \text{ for } 1\leq \mu \leq (i_u-i_\ell) \text{ if } x_i <0~~ (\lambda^{(i)}_{\mu}<0)
\end{align}  
Then $x_{i}=\sum_{\mu=1}^{i_u-i_\ell}\lambda^{(i)}_\mu$. We write 
\[a^{k_{i-1}}t^{x_i}a^{k_i}=_{\DelBS}a^{k_{i-1}}t^{\lambda^{(i)}_1}t^{\lambda^{(i)}_2}\cdots t^{\lambda^{(i)}_{i_u-i_\ell}}a^{k_i}\]
with $\lambda^{(i)}_\mu$ defined as above. So we can consider the word
\begin{align*}
    \wal&=a^{k_{0}}\underbrace{t^{\lambda^{(1)}_1}t^{\lambda^{(1)}_2}\cdots t^{\lambda^{(1)}_{1_u-1_\ell}}}_{t^{x_1}}a^{k_1}\cdots a^{k_{n-1}}\underbrace{t^{\lambda^{(n)}_1}t^{\lambda^{(n)}_2}\cdots t^{\lambda^{(n)}_{n_u-n_\ell}}}_{t^{x_n}}a^{k_n}.
\end{align*}
By construction, for each $ i$ there is $ j \in \interval{1}{r}$ such that  $\height{i}{\wal}=\eta_j$. Thus, $\height{i}{\wal} \in  \interval{\eta_j}{\eta_{j+1}}$ implies that $\height{i}{\wal} \in \{\eta_j,\eta_{j+1}\}$. Hence, $\wal$ is \taligned. For each $x_i$ we constructed at most $n$ of the $\lambda^{(i)}_\mu$, so $\len{\Theta}{\wal}\leq(\len{\Theta}{w}^2)$. By construction, $\wal=_{\DelBS} w$. This shows the lemma.
\qed\end{proof}
Note that the word $\tilde{w}$ we constructed above is not unique with the property that it is \taligned and $\tilde{w}=_{\DelBS}  w$. Because $\tilde{w}=_{\DelBS}  w$ the definition of being  \taligned  only makes sense for words over $\Theta$. Moreover, if $i<j$ and $w$ is \taligned, then $w_I$ is \taligned for each $I \in \{(i,j), (i,j], [i,j] \}$. Recall our convention that $\mn{i}{j}=\min\{i,j\}$ and $\mx{i}{j}=\max\{i,j\}$.

\begin{lemma}\label{lem: CancelCompl}
Let $w=a^{k_0}t^{x_1}a^{k_1}\cdots t^{x_n}a^{k_n} \in \Theta^*$ be \taligned. If $w\inBSpr \gen{a}$ then for each $i \in \interval{1}{n}$ there exists $j\in \interval{1}{n}$ such that
\begin{enumerate}[(1)]
\item $x_i=-x_j$,
\item $\wcl{\mn{i}{j}}{\mx{i}{j}-1}\inBSpr \gen{a^p}$ if $x_{\mn{i}{j}}>0$,
\item $\wcl{\mn{i}{j}}{\mx{i}{j}-1}\inBSpr \genbig{a^{p\cdot q^{-x_{\mn{i}{j}}}}} $ if $x_{\mn{i}{j}}<0$.
\end{enumerate}
If all conditions are satisfied, we say that $t^{x_i}$ \emph{cancels} with $t^{x_j}$.

\end{lemma}

\begin{remark}\label{rem:BRonTheta}
 \cref{lem: CancelCompl} shows that for \taligned words $w$ we have $w \inBSpr \gen{a}$ if and only if there is $s\in \Z$ such that $w$ can be reduced to $a^s$ using only the following rules (on $\Theta^*$) for $x>0$: 
\begin{align}\label{BRinBS}
t^xa^kt^{-x} &\:\to\: a^{q^x \cdot k}  &&\text{ if } k \in p\Z, \nonumber\\ 
t^{-x} a^kt^x &\:\to\:  a^{ q^{-x} \cdot k}  &&\text{ if } k \in (p\cdot q^{-x}) \Z,\\ \nonumber
a^ka^\ell &\:\to\: a^{k+\ell}.
\end{align}
%¸
\end{remark}

\begin{proof}[of \cref{lem: CancelCompl}] 
Assume that $w \inBSpr \gen{a}$.
 We show that conditions (1) and (2) or (1) and (3) are satisfied. 

Let $ i \in \interval{1}{n}$ and assume that $x_i<0$. We write $t^{x_i}=_{\DelBS}t^{-1}t^{x_i+1}$. By Britton's Lemma (see \cref{rmk:Britton}) there exists $j$ with $x_j>0$ and $0\leq y<x_j$ such that
\begin{align*}
tt^{y}\wcl{j}{i-1}t^{-1}
\;&\inBSpr  \gen{a} \qquad \text{and } &t^{y}\wcl{j}{i-1}
\;&\inBSpr  \gen{a}  \qquad \text{or} \\
t^{-1}t^{x_i+1}\wcl{i}{j-1}t^{y}t
%\eqBSpr t^{-1}a^{k\cdot p\cdot q^\ell}t\;
 &\inBSpr \gen{a} \qquad \text{and }  &t^{x_i+1}\wcl{i}{j-1}t^{y} &\inBSpr  \gen{a}.
\end{align*}
% for some  $k,\ell \in \Z$, $\ell>0$. 
  Because $w$ is \taligned, $\up{w}{i}=\up{w}{j}$ and $\low{w}{i}=\low{w}{j}$ and thus $x_i=-x_j$ (this implies that $y=0$ or $y=x_i+1$). Hence, we are in one of the situations
\begin{align}
  \wopij{j-1}{i} &=_{\DelBS}  t^{-x_i-1}t\wcl{j}{i-1}t^{-1}t^{x_i+1}\qquad \text{or} \label{inner}\\ 
  \wopij{i-1}{j} &=_{\DelBS} t^{-1}t^{x_i+1}\wcl{i}{j-1}t^{-x_i-1}t. \label{outer}
\end{align} 
If we are in case (\ref{inner}),  because $\wcl{j}{i-1} \inBSpr \gen{a}$, we obtain $ t\wcl{j}{i-1}t^{-1} \eqBSpr  t a^k t^{-1}$ for some $k \in \Z$. As $t a^k t^{-1} \inBSpr \gen{a}$, it follows by Britton's Lemma that  $\wcl{j}{i-1}\inBSpr \gen{a^p}$. Thus, (1) and (2) hold. 

Assume that we are in case (\ref{outer}). Then, by a similar argument as before, we obtain  $\wopij{i-1}{j}\inBSpr \gen{a^p} $. 
Therefore, $\wcl{i}{j-1} \eqBSpr t^{-x_i} \wopij{i-1}{j} t^{x_i} \inBSpr \genbig{a^{p \cdot q^{-x_i}}}$.
%We write $z = t^{x_i+1}\wcl{i}{j-1}t^{-x_i-1} $. 
%As $\wopij{i-1}{j}\in \gen{a} $, we have $z\inBSpr \gen{a^{pq}}$.
%%If  $\wcl{i}{j-1}\inBSpr \gen{a^{p \cdot q^{-x_i}}}$  then, $z \in \gen{a^{pq}}$.
%Thus, we obtain $\wcl{i}{j-1}\eqBSpr t^{-x_i-1}z t^{x_i+1}\inBSpr \genbig{a^{p \cdot q^{-x_i}}}$.
 Hence, (1) and (3) hold. 
 
For the case $x_i>0$ we write $t^{x_i}=_{\DelBS} t^{x_i-1}t$ and use the same arguments as above. \qed
\end{proof}

\newcommand{\BStwth}{\ensuremath{\mathrm{\mathbf{BS}}_{2,3}}}
\begin{remark}\label{rem:BSTwoThree}
In \cref{lem: CancelCompl}, for $\BS{p}{r}$ we really need that $p$ divides  $r$.  Consider the group  $\mathbf{BS}_{2,3}$ and let \[w=t^{-2}a^3t^2a^2t^{-2}a^6t^2.\] This word is \taligned and 
\begin{align*}
t^{-2}a^3t^2a^2t^{-2}a^6t^2&=_{\BStwth} t^{-2}a^3ta^3t^{-1}a^6t^2\\
&=_{\BStwth} t^{-1}a^2a^3a^4t\\
&=_{\BStwth} a^6. 
\end{align*}
On the other hand,  $t^{-2}a^3t^2=_{\BStwth} t^{-1}a^2t$, $t^{2}a^2t^{-2}=_{\BStwth} ta^3t^{-1}$ and $t^{-2}a^6t^2=_{\BStwth} t^{-1}a^4t$. All the right sides are Britton-reduced, so no $t^{x_i}$ cancels completely with another $t^{x_j}$. This is a contradiction to the statement of \cref{lem: CancelCompl}. 

\end{remark}

\subsection{A formula for $w\inBSpr \gen{a}$}

\begin{lemma}[{{\cite[Lemma 9]{Weiss16}}}]\label{lem: wInGenA} 
Consider the word $w=a^{k_0}t^{x_1}a^{k_1}t^{x_2}\cdots  t^{x_n}a^{k_n}$. Then $w\inBSpr \gen{a}$, if and only if $w\eqBSpr a^{k}$  with $ k=\sum_{\nu=0}^{n}k_{\nu}\cdot q^{\height{w}{\nu}}$. 
\end{lemma}

Notice that this lemma also holds if $w$ is not \taligned.  
\begin{definition}\label{def: defOfkij} 
Let $w=a^{k_0}t^{x_1}a^{k_1}t^{x_2}\cdots  t^{x_n}a^{k_n}$ and $1\leq i<j \leq n$. We introduce the following notations: 
\begin{enumerate}[(i)]
	\item $\numba{w}=\sum_{\nu=0}^{n}k_{\nu}\cdot q^{\height{w}{\nu}}$
	\item $\kcl{i}{j}=\sum_{\nu=i}^{j}k_\nu \cdot q^{\height{w}{i,\nu}}$.
\end{enumerate}	
In particular, $\numba{\wcl{i}{j}}=\kcl{i}{j}$ and $\numba{w}=\kcl{0}{n}$. 
\end{definition}

Note that if $w \not\inBSpr \gen{a}$, then $\numba{w}$  is possibly not in $\Z$; but $w \inBSpr \gen{a}$ implies that $\numba{w} \in \Z$.
\smallskip

Next, we prove a formula how to calculate $\numba{vw}$ from $\numba{v}$ and $\numba{w}$. We also show that $\numba{v}$ does not change under Britton-reductions in $\BS{p}{pq}$. 
\begin{lemma}\label{lem: ConcOfwv}
Let $v=a^{k_0}t^{x_1}a^{k_1}\cdots  t^{x_n}a^{k_n}$ and $w=a^{\ell_0}t^{y_1}a^{\ell_1}\cdots  t^{y_m}a^{\ell_m}$. Then, $\numba{vw}=\numba{v}+q^{\hght{v}}\cdot \numba{w}$. Moreover, if $w \eqBSpr w'$, then $\numba{w}=\numba{w'}$. 
\end{lemma}
\begin{proof}
Observe that $ vw=a^{k_0}t^{x_1}\cdots t^{x_n}a^{k_n+\ell_0}t^{y_1}a^{\ell_1}\cdots t^{y_{m}}a^{\ell_{m}}$.
 By \cref{lem: wInGenA},
\begin{align*}
    \numba{vw}&=\left(\sum_{\nu=0}^{n}k_\nu \cdot  q^{\height{v}{\nu}}\right)+\ell_0\cdot q^{\hght{v}} +\sum_{\nu=1}^{m} \ell_\nu \cdot q^{\hght{v}+\height{w}{\nu}}
    \\&=\sum_{\nu=0}^{n}k_\nu \cdot  q^{\height{v}{\nu}} +q^{\hght{v}}\cdot \sum_{\nu=0}^{m} \ell_\nu \cdot q^{\height{w}{\nu}}
    \\&=\numba{v}+q^{\hght{v}}\cdot\numba{w}.
\end{align*}
So we have that
\begin{align*}
\numba{w_1ta^p t^{-1}w_2}&=\numba{w_1}+q^{\hght{w_1}}(\numba{ta^p t^{-1}}+q^{\hght{ta^p t^{-1}}}\numba{w_2})\\
&=\numba{w_1}+q^{\hght{w_1}}qp+q^{\hght{w_1}}\numba{w_2}\\
&=\numba{w_1a^{pq}w_2}.
\end{align*}
The case $t^{-1}a^{pq}t$ follows in a similar way.
This shows that $\numba{w}$ does not change under Britton-reductions in $\BS{p}{pq}$. Thus, the lemma follows. \qed
\end{proof}

\begin{remark}\label{rem:proj}
Using $\numba{ }$ we obtain a projection from $\BS{p}{pq}$ onto $\BSq$ if we use the presentation of $\BSq$ as the semi-direct product $\BSq\cong \sdZ$ via
\begin{align}
\rho: \BS{p}{pq} \rightarrow \BSq, \:\: \rho(u)=(\numba{u}, \hght{u}). 
\end{align}	
By \cref{lem: ConcOfwv} and the definition of the multiplication in $\sdZ$ we see that $\rho$ is indeed a homomorphism:
\begin{align*}
(\numba{u}, \hght{u}) \cdot(\numba{v},  \hght{v})=(\numba{u}+q^{\hght{u}}\cdot\numba{v}, \:\hght{u}+\hght{v})=\rho(uv).
\end{align*}
Moreover, if $(q^k\cdot r,m) \in \sdZ$ we have that $\rho(t^ka^rt^{m-k})=(q^k\cdot r,m)$. So, $\rho$ is, indeed, surjective.

\end{remark}

\subsection{Two helpful relations} 

In this chapter we define two relations $\approx$ and $\sim_c$ on the index set $\{1, \ldots, n\}$ if $n$ is the number of  $t$-blocks of $w \in \BS{p}{pq}$. The definition of $\sim_c$ will be quite similar to the one in \cite{Weiss16}, but the one of $\approx$ will be different. This will lead to a much shorter proof of the property shown in \cref{NecCond} and that $\approx$ is indeed an equivalence relation. 
For the definition of the relations and the proof of their properties we do not require \taligned words. However, in the following, we will use these relations only on \taligned words. So, one might safely assume that $w$ is \taligned.

First, we show how to add different $\kcl{i}{j}$: 

\begin{lemma} \label{lem: Addofk}
Let $w$ be as in (\ref{wBRBS}) and $i<j<\ell$. 
Then the following holds:
\begin{enumerate}[(i)]
\item \label{AddSameHeight}  $\kcl{i}{\ell-1}=\kcl{i}{j-1}+\kcl{j}{\ell-1}$ if $\height{w}{i}=\height{w}{j}$,
\smallskip
 \item  $\kcl{i}{\ell-1}=\kcl{i}{j-1}+q^{x_j}\cdot \kcl{j}{\ell-1}$ if $\height{w}{j}=\height{w}{i}+x_j$
\end{enumerate}

\end{lemma}
\begin{proof}
First observe that $\wcl{i}{\ell-1}=\wcl{i}{j-1}\wopi{j-1}{\ell-1}$. Moreover, note that $\hght{\wcl{i}{j-1}}=\height{w}{i,j-1}.$ Thus by \cref{lem: ConcOfwv}, 
\begin{align*}
   \numba{\wcl{i}{\ell-1}}=\numba{\wcl{i}{j-1}}+q^{\height{w}{i,j-1}}\cdot \numba{\wopi{j-1}{\ell-1}}.
\end{align*}
If $\height{w}{i}=\height{w}{j}$ then $\height{w}{i,j-1}=-x_j$. If $\height{w}{j}=\height{w}{i}+x_j$ then $\height{w}{i,j-1}=0$. This shows (i) and (ii). 
\qed\end{proof}

\begin{definition}\label{def: EquivRel}
Let $w$ be as in (\ref{wBRBS}). We say that $i \approx j$ if
\begin{enumerate}[(1)]
\item \label{EquivxiEqual} $x_i=x_j$,
\item \label{heightEqual} $\height{w}{i}=\height{w}{j}$, 
\item \label{EquivCond}
\setlength{\abovedisplayskip}{-10pt}
\begin{flalign*}
\kcl{\mn{i}{j}}{\mx{i}{j}-1} &\in p\Z  &\text{ if } x_{\mn{i}{j}}>0&&&&&&\\
\kcl{\mn{i}{j}}{\mx{i}{j}-1} &\in (p \cdot q^{-x_{\mn{i}{j}}})\Z & \text{ if } x_{\mn{i}{j}}<0.&&&&&&  
\end{flalign*} 
\end{enumerate}
By definition we say that $i \approx i$. 
%If $j<i$ we check the conditions for $\kopj{j}{i}$ or  $\kopi{j-1}{i-1}$ respectively. 
\end{definition}

We claim that $\approx$ is an equivalence relation. It is \emph{reflexive} and \emph{symmetric} by definition.  It remains to show that it is \emph{transitive}. 
\medskip

 Assume that $i\approx j$ and $j \approx \ell$. Then  $x_i=x_j=x_\ell$ and $\height{w}{i}=\height{w}{j}=\height{w}{\ell}$. Let $\{\alpha, \beta, \gamma\}=\{i,j,\ell\}$ with $\alpha < \beta < \gamma$.
  By \cref{lem: Addofk},
  \begin{align}
\label{EqRelOne}\kcl{\alpha}{\gamma-1}&=\kcl{\alpha}{\beta-1} +  \kcl{\beta}{\gamma-1}. %\text{ and }
%\nonumber \\ 
%\kopi{i-1}{\ell-1}&=\kopi{i-1}{j-1}+\kopi{j-1}{\ell-1}.
\end{align}
%With \cref{lem: kclAndkop} it follows that %\begin{align}\label{EqRelTwo}
%\kopj{i}{\ell}&=\kopj{i}{j} +  \kopj{j}{\ell} 
%\end{align}
For two summands in (\ref{EqRelOne}) we know that they are in $s\Z$ for the appropriate $s \in  \{p, p \cdot q^{\abs{x_i}}\}$ and so is the third one. This shows that condition (\ref{EquivCond}) of \cref{def: EquivRel} is satisfied. 

We now define a second relation $i \sim_c j$ that indicates if some necessary conditions for $t^{x_i}$ and $t^{x_j}$ to cancel are satisfied.  
\begin{definition}\label{def: DefCancRel}
Let $w$ be as in (\ref{wBRBS}). We say that $i \sim_c j$ if
\begin{enumerate}[(1)]
\item $x_i=-x_j$, 
\item \label{relSumCond} $\height{w}{j}=\height{w}{i}+x_j$,
\item \label{relDivCond}
\setlength{\abovedisplayskip}{-10pt} \begin{flalign*}
\kcl{\mn{i}{j}}{\mx{i}{j}-1} &\in p\Z && \text{ if } x_{\mn{i}{j}}>0&&&&&&\\
\kcl{\mn{i}{j}}{\mx{i}{j}-1} &\in (p \cdot q^{-x_{\mn{i}{j}}})\Z &&   \text{ if } x_{\mn{i}{j}}<0.&&&&&&
\end{flalign*} 
\end{enumerate}
\end{definition}
This relation is symmetric by definition, but it is not an equivalence relation since it is not reflexive.

\begin{lemma}\label{NecCond}
If $i \approx i'$ and $i\sim_c j$, then also $ i' \sim_c j$. 
\end{lemma}

\begin{proof}
First observe that  $i \approx i'$ implies that $\height{w}{i}=\height{w}{i'}$ and $x_i=x_{i'}$. Moreover, $i\sim_c j$ implies that $x_{i'}=x_i=-x_j$. Let $\alpha < \beta<  \gamma$ such that $\{\alpha,\beta,\gamma\}=\{i,i',j\}$. We consider the following sums:
\begin{align*}
    \underbrace{x_\alpha+\sum_{\mu=\alpha+1}^{\gamma-1}x_\mu}_{(1)}= \underbrace{x_\alpha+\sum_{\mu=\alpha+1}^{\beta-1}x_\mu}_{(2)}  + ~ \underbrace{x_{\beta}+\sum_{\mu=\beta+1}^{\gamma-1}x_\mu}_{(3)}
\end{align*}
Because $i \approx i'$, condition (\ref{heightEqual}) in \cref{def: EquivRel} is satisfied and $x_i=x'_i$. Thus, one of the sums (1)--(3)  evaluates to $0$ and the remaining $x_\alpha$ and $x_\beta$ cancel.  If (3) is $0$, this is clear. If (1) is $0$, then $\{i,i'\}=\{\alpha, \gamma\}$ and so $x_\alpha=-x_\beta$. If (2) is $0$ then $\{i,i'\}=\{\alpha, \beta\}$ and thus $x_\alpha=x_\beta$. Because $i \sim_c j$ one of the remaining sums is $0$, and so is the last one. This shows that $\height{w}{\mx{i'}{j}}=\height{w}{\mn{i'}{j}}+x_{\mx{i'}{j}}$.  Thus,  
\newcommand{\exq}{e}
\begin{align*}
\kcl{\alpha}{\gamma-1}=\kcl{\alpha}{\beta-1}+q^{\exq \cdot x_\beta}\cdot \kcl{\beta}{\gamma-1} 
\end{align*}
with $e=1$ if $j \in \{\alpha, \beta\}$ and $e=0 $ for $\gamma=j$. In the latter case it is clear that all three summands are in $r\Z$ for $r \in \{p, p\cdot q^{\abs{x_i}}\}$.   Next observe that $\kcl{i}{i'} \in r\Z$  if and only if $\kcl{j}{i} \in (q^{x_i}\cdot r)\Z$ for $r \in \{p,p\cdot q^{-x_i}\}$ (if $j<i, i<i'$).  Moreover, $\beta=j$ if and only if $\alpha \in \{i,i'\}$. Thus, considering the remaining cases, it is easy to see that $\kcl{\mn{i'}{j}}{\mx{i'}{j}} \in p\Z$  if $x_{\mn{i'}{j}}>0$ and $\kcl{\mn{i'}{j}}{\mx{i'}{j}} \in (pq^{-x_{\mn{i'}{j}}})\Z$  otherwise. 
This shows point (\ref{relDivCond}) of \cref{def: DefCancRel} and thus the lemma. 
\qed\end{proof}

\subsection{The word problem of $\BS{p}{pq}$}

In this chapter we show that the word problem of $\BS{p}{pq}$ with the input given as a \redpc of a word $w \in \Thepr$ (as in Theorem~\ref{thm:WPinBS}) is in $\uAcf{0}$. We first introduce the following notations (see also \cite{Weiss16}):
Let $w=a^{k_0}t^{x_1}a^{k_1}t^{x_2}\cdots  t^{x_n}a^{k_n} \in \Thepr$ and let $\Sigma_w=\{[i]~|~ i \in \{1, \ldots, n\}\}$ be the set of equivalence classes of $\approx$ w.r.t. $w$. We set $\overline{[i]}=[j]$ if $i \sim_c j$. By \cref{NecCond}, this is well-defined. Moreover, $\overline{\overline{[i]}}=[i]$ because $\sim_c$ is symmetric and $[i]\neq \overline{[i]}$ because $i \not\sim_c i$. We assume that $\Lambda_w$ is such that $\Sigma_w$ is the disjoint union $\Lambda_w \cup \bar{\Lambda_w}$. So: $\Sigma^*_w/\{[i] \overline{[i]}= \overline{[i]}[i]=1~|~i \in \{1, \ldots, n\}\}=F_{\Lambda_w}$, the free group generated by $\Lambda_w$. We define 
\[C(w)=[1]\cdots [n], ~ \text{ and } ~ C(\wcl{i}{j})=[i+1]\cdots [j].\]
The following lemma follows closely its analog in  \cite[Lemma 14]{Weiss16}.

\begin{lemma}\label{lem: CwEqOne}
Let $w \in \Thepr$ be \taligned. Then we have $\wcl{i}{j}\inBSpr \gen{a}$ if and only if $C(\wcl{i}{j})=_{F_{\Lambda_w}}1$. 
\end{lemma}
Note that \cref{lem: CwEqOne} reduces the subgroup membership problem of $\gen{a}$ in $\BS{p}{pq}$ to the word problem in the free group $F_2$.
\begin{proof}
``$\Rightarrow$'' We assume that $\wcl{i}{j}\inBSpr \gen{a}$. By  \cref{lem: CancelCompl}  there exists $i+1 < \ell \leq j$ such that $t^{x_{i+1}}$ and $t^{x_\ell}$ cancel. In particular, $i+1 \sim_c \ell$ and $[i+1]=\overline{[\ell]}$.  Because $t^{x_{i+1}}\wcl{i+1}{\ell-1} t^{x_\ell}\inBSpr  \gen{a}$, also  
$\wcl{\ell}{j}\inBSpr  \gen{a}$.  By induction on the length of $C()$, $C(\wcl{i+1}{\ell-1})=_{F_{\Lambda_w}}C(\wcl{\ell}{j})=_{F_{\Lambda_w}}1$. So 
\begin{align}\label{CwEqOne}
C(\wcl{i}{j})=[i+1]C(\wcl{i+1}{\ell-1})\overline{[i+1]}C(\wcl{\ell}{j})=_{F_{\Lambda_w}}1.
\end{align}

``$\Leftarrow$'' Assume that $C(\wcl{i}{j})=_{F_{\Lambda_w}}1$. Then $[i+1]$ cancels with some $[\ell]$ in $F_{\Lambda_w}$. So we are in the same situation as in (\ref{CwEqOne}). 
This implies $C(\wcl{i+1}{\ell-1})=_{F_{\Lambda_w}}C(\wcl{\ell}{j})=_{F_{\Lambda_w}}1$. By induction on the number of $t$-blocks, $\wcl{i+1}{\ell-1}\inBSpr \gen{a}$ and $\wcl{\ell}{j}\inBSpr \gen{a}$. 
Because $[i+1]=\overline{[\ell]}$, $i+1 \sim_c \ell$. Thus, $\wcl{i+1}{\ell-1}\inBSpr \gen{a}$ implies that  $t^{x_{i+1}}\wcl{i+1}{ \ell-1}t^{x_\ell} \inBSpr  \gen{a}$. Hence, \[\wcl{i}{j}=a^{k_i}t^{x_{i+1}}\wcl{i+1}{\ell-1}t^{x_\ell}\wcl{\ell}{j}\inBSpr \gen{a}.\] 
This shows the lemma.
\qed\end{proof}
Observe that in \cref{lem: CwEqOne} it is necessary that $w$ is \taligned. Otherwise it can happen that $w \inBSpr \gen{a}$ but $i \not\sim_c j$ for all $i,j$ and thus, $C(w) \neq 1$. An example for such a word is $w=t^2a^pt^{-1}a^pt^{-1}$. 
\medskip 

The following corollary is similar to \cite[Corollary 17]{Weiss16} but because we consider \taligned words over $\Theta$ we need to check some additional conditions here. 

\begin{corollary}\label{cor: BrittonRed} 
Let $w=a^{k_0}t^{x_1}\cdots  t^{x_n}a^{k_n} \in  \Thepr$ be \taligned and $ \alpha=[i_1]\cdots [i_m] \in \Sigma^*_w$ be freely-reduced with $i_j <i_h$ for $j<h$ and $C(w)=_{F_{\Lambda_w}}\alpha$. Write
\begin{align*}
    w'=a^{\kcl{0}{i_1-1}}t^{x_{i_1}}a^{\kcl{i_1}{i_2-1}}\cdots t^{x_{i_m}}a^{\kcl{i_m}{n}},
\end{align*}
$\ell_j=\kcl{i_j}{i_{j+1}-1}$ and $x_{i_j}=y_j$. Then,  $  w'\eqBSpr w$ and the following hold: 
\begin{enumerate}[(a)]
%\item \label{pqyDoesnotDiv}  $p\cdot q^{-y_j }$ does not divide $\ell_j$ for all $j$,
\item  \label{ljZero} If  $\ell_j=0 $ then $\sgn{y_j}=\sgn{y_{j+1}}$. 
\item \label{CancelXiNeg}
$w'$ is Britton-reduced if and only if there are no factors of the form $t^{y_j}a^{\ell_j}t^{-y_j}$ with $y_j<0$ and $\ell_j \in (pq)\Z$.
\item \label{BRw} We obtain a Britton-reduced word $w_{red}$ with $w_{red}\eqBSpr w$ by the following construction: In $w$, we replace all $t^{y_j}a^{\ell_j}t^{-y_{j}}$ with $y_j<0$ and $\ell_j \in (pq)\Z$ by $t^{y_j+s_j}a^{q^{-s_j}\cdot\ell_j}t^{-y_j-s_j}$ with  $s_j>0$ maximal such that $q^{s_j}$ divides $\ell_j$.
\end{enumerate}

\end{corollary} 
\begin{proof} 
Because $C(\wcl{i_j}{i_{j+1}-1})=_{F_{\Lambda_w}}1$ and $\numba{\wcl{i_j}{i_{j+1}-1}}=\kcl{i_j}{i_{j+1}-1}$, by \cref{lem: wInGenA} and \cref{lem: CwEqOne} it follows that $  w'\eqBSpr w$.  Moreover, each $t^{x_i}$ cancels with some $t^{x_j}$ with $x_i=-x_j$ or is equal to some $x_{i_\ell}$.   So, 
 $w'$ is still \taligned.  

If $\ell_j=0$ and $\sgn{y_j}= -\sgn{y_{j+1}}$, then, because $w'$ is \taligned, $\abs{y_j}=\abs{y_{j+1}}$ and thus $[i_j]=\overline{[i_{j+1}]}$.   We obtain a contradiction to $\alpha$ being freely reduced, so we have shown (\ref{ljZero}). 

If there are factors as described in (\ref{CancelXiNeg}) it is clear that $w'$ is not Britton-reduced. 
So assume that  there are no such factors. Hence, for a Britton-reduction to be possible, there has to exist a factor of the form $t^{y_i}a^{\ell_i}t^{y_{i+1}}$ with  $\ell_i  \in p\Z$, $y_i>0$ and $y_{i+1}<0$ (see (\ref{BRinBS}), we can assume that all $y_i \neq 0$). Because $w'$ is \taligned, $\abs{y_i}=\abs{y_{i+1}}$. This implies that $[i]=\overline{[i+1]}$. Thus we  obtain a contradiction to $\alpha$ being freely reduced. 
Thus, no Britton-reductions are possible. This also shows the last point of the lemma: After replacing factors as described in (\ref{BRw}), no factors as in (\ref{CancelXiNeg})  are left and thus we obtain a Britton-reduced word (note that $s_j<\abs{y_j}$, otherwise we obtain again a contradiction to $\alpha$ being freely reduced). 
\qed\end{proof}

\subsubsection{Power circuit representations.} 

As before we fix a Baumslag-Solitar group $\BS{p}{pq}$ with $p \geq 1$ (and $q\neq 0$).
To be able to proof Theorem~\ref{thm:WPinBS} we need to explain how we represent words $w \in \Thepr$ using reduced power circuits.

\begin{definition}\label{def: PCrepBS}
	Let $(\Gamma, \delta)$ be a reduced power circuit. 
	\begin{enumerate}[(a)]
		\item Let $w=w_1\cdots w_n$ with $w_i \in \Theta$. We call $\mathcal{W}=(X_i, K_i)_{i \in \interval{1}{\len{\Theta}{w}}}$ a \redpc (reduced power circuit representation) of $w$ over $(\Gamma, \delta)$ if  the following holds for all $i \in \interval{1}{\len{\Theta}{w}}$: 
		\begin{itemize}
			\item $X_i$ and $K_i$ are compact markings on $\Gamma$,
			\item If $w_i=a^k$, then $\eps(K_i)=k$ and $X_i=\emptyset$,
			\item If $w_i=t^x$, then $\eps(X_i)=x$ and $K_i=\emptyset$.
		\end{itemize}
	
		\item For a \redpc $\mathcal{W}=(X_i, K_i)_{i \in \interval{1}{\len{\Theta}{w}}}$ of a word $w \in \Thepr$ we define $\msumk{\mathcal{W}}=\sum_{i=1}^{\len{\Theta}{w}} \abs{\sigma(K_i)}$. We further set $\len{\Theta}{\mathcal{W}}=\len{\Theta}{w}$.
\end{enumerate}
\end{definition}
Note that the term ``reduced'' in  \redpc refers to the power circuit (and does not mean that the word $w$ has to be Britton-reduced).

\begin{remark}\label{rem:MultByMinQ}
	If $q$ is negative we still can use usual power circuits over a positive basis $\abs{q}$ for the \redpc of a word $w \in \Thepr$:  We show how we can use \cref{lem:moduloPowerOfTwoConstant} to  multiply by powers of $-q$. To calculate $\eps(K) \cdot (-q)^{\eps(X)}$ for two markings $K$ and $X$ we first calculate $\eps(X) \bmod 2$ using \cref{lem:moduloPowerOfTwoConstant}. Then we calculate $\eps(K)\cdot\abs{q}^{\eps(X)}$ using  \cref{lem: powtwo} and calculate the sign according to the result of $\eps(X) \bmod 2$.  
	\smallskip
\end{remark}

\begin{remark}\label{lem:computeAh}
Observe that using \cref{lem: add} and \cref{lem:AddMult}, on input of a \redpc of a word $w \in \Thepr$ we can construct compact markings in a reduced power circuit evaluating to $\numba{w}$ or $\hght{w}$ respectively in $\uTC{0}$. This is due to the fact that we need at most three layers of addition and multiplications by powers of $q$. 
\end{remark}

The two crucial steps to solve the word problem in $\BSpq$ are making a word \taligned and then use \cref{cor: BrittonRed} to calculate the Britton-reduction. In the following, we perform these steps using \redpcs.

\begin{lemma}\label{lem: MakeTalinged} 
	The following is in $\uTC{0}$:
	\compproblem{A \redpc $\mathcal{W}=(X_i, K_i)_{i \in \interval{1}{\len{\Theta}{w}}}$ of a word  $w \in  \Thepr$ over $(\Gamma, \delta)$.}{A \redpc $\tilde{\mathcal{W}}=(Y_i, L_i)_{i \in \interval{1}{\len{\Theta}{\tilde{w}}}}$ of a word  $\tilde{w} \in  \Thepr$ such that $\wal$ is \taligned, $\wal\eqBSpr w$, $\len{\Theta}{\wal} \leq (\len{\Theta}{w})^2$, and
		\begin{itemize}
			\item  $\abs{\Gamma'}\leq \abs{\Gamma}+(\ceil{\log(\len{\Theta}{w})}+1)\cdot \abs{\mathcal{C}_{\Gamma}}$,
			\smallskip
			
			 \item$\abs{\mathcal{C}_{\Gamma'}}\leq \abs{\mathcal{C}_{\Gamma}}$, 
			\item
			\smallskip
						
			$\msumk{\tilde{\mathcal{W}}}=\msumk{\mathcal{W}}$.
	\end{itemize}}
\end{lemma}
\begin{proof}  
	We use the same construction as in the proof of  \cref{lem: exTaligned}. To calculate the $\height{w}{j}$, to compare them and to finally calculate the $\eps(Y_i)$ (the $\lambda_\mu$ in \cref{lem: exTaligned}) we only need a constant number of layers of \addition and \comp (see \cref{lem: add} and \cref{lem:compareCompactMarkings}).  For the markings $L_i$ we have that $L_i=K_j$ for some $j$ or $L_i=\emptyset$. When calculating the $\height{w}{j}$  we have sums of length at most $\len{\Theta}{w}$. To obtain the $Y_i$ we use the $\height{w}{j}$ to calculate sums of length two. Thus, according to \cref{lem: add} we can construct the markings $Y_i$  independently in parallel in $\uTC{0}$. Also the size constraints follow by that lemma.
	\qed\end{proof}

The following theorem is a main ingredient for the proof of Theorem ~\ref{thm:WPinBS}.

\begin{theorem} \label{thm: BRinBS}
	The following is possible in $\uAcf{0}$ : 
	\compproblem{A \redpc $\mathcal{W}=(X_i, K_i)_{i \in \interval{1}{\len{\Theta}{w}}}$ of a word  $w \in  \Thepr$ over $(\Gamma, \delta)$.}{A \redpc $\mathcal{W}^{(r)}=(X^{(r)}_i, K^{(r)}_i)_{i \in \interval{1}{\len{\Theta}{w_{red}}}}$ of a Britton-reduced word $w_{red}\in  \Thepr$ with $w_{red}\eqBSpr w$ over $(\Gamma^{(r)}, \delta^{(r)})$ such that $\len{\Theta}{w_{red}} \leq \len{\Theta}{w}$,  $\msumk{\mathcal{W}^{(r)}} \leq \msumk{\mathcal{W}} $ and there is a constant $c$ such that
		\begin{itemize}
				\item $\abssmall{\mathcal{C}^{(r)}_{\Gamma}} \leq \abs{\mathcal{C}_{\Gamma}}+ c \cdot \msumk{\mathcal{W}}$. 
				\smallskip
				
			\item $\abssmall{\Gamma^{(r)}}  \leq \abssmall{\Gamma}+c\cdot(\msumk{\mathcal{W}}+\abs{\mathcal{C}_{\Gamma}})\cdot (\ceil{\log(\len{\Theta}{w})})$
				\end{itemize}
	}
\end{theorem}

\begin{proof}
	The proof consists of the following steps: 
First, we construct the \redpc of a \taligned word $\tilde{w} \in \Thepr$ with $\tilde{w}\eqBSpr w$ and $\len{\Theta}{\tilde{w}} \leq (\len{\Theta}{w})^2$ (\cref{lem: MakeTalinged}). Then we compute $C(\tilde{w}) \in F_{\Lambda_{\tilde{w}}}$ by checking if $i \approx j$ or $i \sim_c j$ for all $i,j$. Let  $ [i_1]\cdots [i_m]$ be freely reduced such that $ C(\tilde{w})=_{F_{\Lambda_{\tilde{w}}}} [i_1]\cdots [i_m]$. We construct a \redpc of a (still \taligned)  word \begin{align}\label{wAlmostBR}
w'=a^{\kcl{0}{i_1-1}}t^{x_{i_1}}a^{\kcl{i_1}{i_2-1}}\cdots t^{x_{i_m}}a^{\kcl{i_m}{n}}
\end{align} with $w'\eqBSpr w$. Since $\tilde{w}$ is \taligned, this can be seen by \cref{lem: CwEqOne}.  Now we can Britton-reduce  $w'$ using \cref{cor: BrittonRed}.  In the following, we give some more details on these steps. 
	\smallskip

	Assume that we have constructed a \redpc $\mathcal{W}^{(\ell)}=(X^{(\ell)}_i, K^{(\ell)}_i)_{i \in \interval{1}{\len{\Theta}{\wal} }}$ of a \taligned word $\wal$  with $\tilde{w}\eqBSpr w$ and $\len{\Theta}{\tilde{w}} \leq (\len{\Theta}{w})^2$. This is possible in $\uTC{0}$ due to \cref{lem: MakeTalinged}.
	 Let $n$ be the number of letters $t^x \in \Theta$ occurring in $\tilde{w}$. To compute  $C(\wal)=[1]\cdots [n]$ as a word in  $\Sigma_{\wal}$ (see beginning of this section), 
	 for each $i,j \in \interval{1}{n}$ we need to check if $i \approx j$ or $i \sim_c j$ by checking the conditions in  \cref{def: EquivRel} and \cref{def: DefCancRel}. For this, we calculate $\numba{\tilde{w}_{[i,j]}}$ (\cref{def: defOfkij}   and \cref{lem:AddMult}), $\height{\tilde{w}}{i}$ (\addition) and $\bmod~p \cdot q^{x}$ (\cref{lem:moduloPowerOfTwoConstant},  $p,q$ are constants). So, this is in $\uTC{0}$. 
    For each pair $i,j$  we use a separate checking power circuit we discard after the checking. 
	Now, we freely reduce $C(\wal)$ as a word over $\Sigma_{\wal}$.
	\smallskip
	
	 Let $ [i_1]\cdots [i_m]$ be this freely-reduced word. We choose the representatives such that $i_j<i_k$ for $j<k$ (so $\kcl{i_j}{i_{j+1}-1}$ is well-defined). Since $C(\tilde{w}_{[{i_j},{i_{j+1}-1}]} )=_{F_{\Lambda_{\tilde{w}}}}1$ and $\tilde{w}$ is \taligned,
	\cref{lem: CwEqOne} implies that $\tilde{w}_{[{i_j},{i_{j+1}-1}]} \inBSpr \gen{a}$ and, therefore,  $\numba{\tilde{w}_{[{i_j},{i_{j+1}-1}]}}=\kcl{i_j}{i_{j+1}-1} \in \Z$  for $j \in \interval{0}{m}$ (with $i_0=0$ and $i_{m+1}-1=n$). In particular, $w'\eqBSpr \tilde{w}$ for $w'$ as in (\ref{wAlmostBR}).
	
 We construct markings $K'_j$ with $\eps(K'_j)=\kcl{i_j}{i_{j+1}-1}$ using the formula in \cref{def: defOfkij}   on input of the markings $K_i^{(\ell)}$. By \cref{lem:AddMult}
	this is possible in \uTC{0}. We set $X_j'=X^{(\ell)}_{i_j}$ for $j \in \interval{1}{m}$. Markings $K'_j,X'_j$, $j \in \interval{1}{\len{\Theta}{w'}}$ not defined during this process are set to the empty marking.  Then, $(X'_j,K'_j)_{j \in \interval{1}{\len{\Theta}{w'}}}$ is a \redpc of $w'$ as in (\ref{wAlmostBR}). Note that $w'$ is still \taligned, since if $[i]$  cancels with  $[j]$ then $\height{w}{i,j-1}=0$ and $x_i=-x_j$.
	\smallskip
	
	According to \cref{cor: BrittonRed}, the word $w'$  is Britton-reduced if and only if it has no factors of the form $t^{x_{i_j}}a^{\kcl{i_j}{i_{j+1}-1}}t^{x_{i_{j+1}}}$ such that $x_{i_j}<0, x_{i_{j+1}}>0$ and $pq ~|~ \kcl{i_j}{i_{j+1}-1}$. If they exist, we can find them using \comp and \cref{lem:moduloPowerOfTwoConstant} (for calculating $\bmod~pq$). We construct a marking $S_j$ such that $\eps(S_j)>0$ is maximal with $p \cdot q^{\eps(S_j)} \mid\kcl{i_j}{i_{j+1}-1}$  using \cref{lem: powtwo}(\ref{MaxPow}). 
	Using \addition and \powertwo we construct compact markings $K^{(r)}_j, X''_{j}$ and $X''_{j+1}$ with  $\epsilon(K^{(r)}_j)= q^{-\epsilon(S_j)}\cdot \epsilon(K'_j)$, $\eps(X''_{j})=\eps(X'_{j})+\eps(S_j)$ and  $\eps(X''_{j+1})=\eps(X'_{j+1})-\epsilon(S_j)$. For all markings $K'_i,X'_i$ not changed during this step we set $X_j''=X'_j$ and $K^{(r)}_j=K_j'$.  According to \cref{cor: BrittonRed} no more Britton-reductions are possible (observe that $\epsilon(S_j) < \abssmall{\eps(X'_j)} $ because otherwise, $[i_1]\cdots [i_m]$ is not freely reduced). 
	\smallskip
	
	Among the above steps, only making a word \taligned can increase its length over $\Theta$. So to ensure that $\len{\Theta}{w_{red}} \leq \len{\Theta}{w}$, we reverse this as follows: for each $j \in \interval{1}{m}$ with $\eps(K^{(r)}_{j-1})\neq 0$ we search for $s\geq j$ maximal such that $\eps(K^{(r)}_{\ell})=0$ for all $\ell \in \interval{j}{s}$.  We construct a marking  $X^{(r)}_{j}$ with $\eps(X^{(r)}_{j})=\sum_{\ell=j}^{s+1}\eps(X_{\ell}'')$ using \addition. If no such $s$ exists we set $X_j^{(r)}=X''_j$. By \cref{cor: BrittonRed}(\ref{ljZero}) all $\eps(X_{\ell}'')$ have the same sign, so  no $t$'s can cancel and the resulting word is still Britton-reduced (but not \taligned in general).  All markings $K^{(r)}_i,X^{(r)}_i$ not defined yet we set to be the empty marking. 
	Thus, $\mathcal{W}^{(r)}=(X^{(r)}_j, K^{(r)}_j)_{j \in \interval{1}{\len{\Theta}{w_{red}}}}$ is a \redpc over a reduced power circuit  $(\Gamma^{(r)}, \delta^{(r)})$ of a Britton-reduced word $w_{red}$ with $w_{red}\eqBSpr w$ and $\len{\Theta}{w_{red}} \leq \len{\Theta}{w}$.
	\medskip

	During the whole construction, we only use a constant number of layers of \addition and \powertwo, calculations of $\bmod~p\cdot q^x$ (\cref{lem:moduloPowerOfTwoConstant}) and one reduction in the free group. For calculating modulo we use separate checking power circuits we discard after the checking.  So, since $\uTC{0} \subseteq\uAcf{0}$ the above algorithm is in $\uAcf{0}$.

	For the size constraints,  the crucial observation is that each marking $K'_{j}$ appears at most once outside of an exponent (\ie as ``mantissa'') multiplied by a suitable power of $q$ in the formulas in \cref{def: defOfkij} or when calculating $q^{\eps(S)} \cdot \eps(K'_j)$ (if we have both operations for one $j$, we can do them at the same time). This includes multiplications by powers of $-q$ due to \cref{rem:MultByMinQ}. Moreover, by \cref{lem:AddMult}, $\abssmall{\sigma(K'_j)}\leq \msumk{\tilde{w}_{[{i_j},{i_{j+1}-1}]}}$. So, by \cref{lem: powtwo} it follows that   $\msumk{\mathcal{W}^{(r)}}\leq \msumk{\mathcal{W}}$.
 Moreover, all sums have length at most $\len{\Theta}{\tilde{w}}\leq\len{\Theta}{w}^2$.
 	Together with the \OpLem and \cref{lem: MakeTalinged} we obtain that there is a constant $c$ such that
	\begin{align*}
	\abssmall{\mathcal{C}_{\Gamma^{(r)}}} &\leq \abs{\mathcal{C}_{\Gamma}}+c\cdot\msumk{\mathcal{W}}\\[0.5mm]
	\abssmall{\Gamma^{(r)}} &\leq \abs{\Gamma} + c \cdot (\abs{\mathcal{C}_{\Gamma}}+\msumk{\mathcal{W}})\cdot (\ceil{\log(\len{\Theta}{w})})
	\end{align*}

	The theorem follows. \qed
	\end{proof}
%

%Now we are ready to prove Theorem~\ref{thm:WPinBS}: 

\begin{corollary}[Theorem~\ref{thm:WPinBS}] 	For every $p,q \in \Z$ with $\abs{p},  \abs{q} \geq 1$ the word problem of $\BS{p}{pq}$ with the input word given over $a^k, t^x$  with $k,x \in \Z$ represented by markings in a reduced power circuit is in $\uAcf{0}$. 
\end{corollary}

\begin{proof}%[Theorem~\ref{thm:WPinBS}]\textbf{}
By 	\cref{lem:BrittonsLemma}, if $w \in \Thepr$   is Britton-reduced, to check if $w\eqBSpr 1$ we just need to check if $w$ is the empty word. So, on input of a \redpc of a word $w \in \Thepr$ we  construct the \redpc of a Britton-reduced word $w_{red}\in \Thepr$ such that $w\eqBSpr w_{red}$. By \cref{thm: BRinBS} this is possible in $\uAcf{0}$. 
	
Observe that in the case $\abs{q}=1$, we have $ta^{p}t^{-1}=a^{\sgn{q} \cdot p}$. Thus, during the algorithm above, no multiplication by a power of $r$ with $r \geq 2$ is performed. So, for the \redpcs of  words over $\Thepr$ we can use power circuits to base $2$ and perform \addition and $\bmod~p\cdot q^x$ (to calculate the sign) on these power circuits.
\qed\end{proof}

To solve the word problem in $\BGpq $  (\cref{sec:WPinBG}), for the Britton-reductions in $\BSpq$ we can restrict to a special case: On input of two already Britton-reduced words $u,v \in \Thepr$ we can Britton-reduce the word $uv$ in $\uTC{0}$. This will be an important ingredient to show that the word problem of $\BGpq$ is in $\uTC{1}$.

\begin{lemma}\label{lem:BRinBSinTC}
The following is in $\uTC{0}$: 
\compproblem{\Redpc's $\mathcal{U}^{(i)}, \mathcal{V}^{(i)}$ of Britton-reduced words $u_i,v_i \in \Thepr$ over $(\Gamma, \delta)$ for $i \in \interval{1}{\nu}$.}{\Redpc's $\mathcal{W}^{(i)}$  of Britton-reduced words $w_i \in \Thepr$  over $(\Gamma', \delta')$ for $i \in \interval{1}{\nu}$  such that  $w_i\eqBSpr u_iv_i$, $(\Gamma, \delta)\leq (\Gamma', \delta')$,  and there is a constant $c$ such that
\begin{itemize}
	\item $\abs{\Gamma'} \leq \abs{\Gamma}+c\cdot\left( \mathcal{S}+\abs{\mathcal{C}_{\Gamma}}\right)\cdot \ceil{\log(\mathcal{M})}$
	\item $\abs{\mathcal{C}_{\Gamma'}}\leq \abs{\mathcal{C}_{\Gamma}}+c \cdot \mathcal{S} $
	\item $\msumk{\mathcal{W}^{(i)}} \leq \msumk{\mathcal{U}^{(i)}}+\msumk{\mathcal{V}^{(i)}}$ for all $i \in \interval{1}{\nu}$
	\item $ \len{\Theta}{w_i} \leq \len{\Theta}{u_i}+\len{\Theta}{v_i}$ for all $i \in \interval{1}{\nu}$
\end{itemize}
with $\mathcal{S}=\sum_{i=1}^{\nu}\left(\msumk{\mathcal{U}^{(i)}}+\msumk{\mathcal{V}^{(i)}}\right)$\newline and $ \mathcal{M}=\max_{i \in \interval{1}{\nu}}(\len{\Theta}{u_i}+\len{\Theta}{v_i})$.
} 	
\end{lemma}	
Note that, in particular,  $\sum_{i=1}^{\nu}\msumk{\mathcal{W}^{(i)}} \leq \mathcal{S}$ and
$ \max_{i \in \interval{1}{\nu} }(\len{\Theta}{w_i}) \leq \mathcal{M}$. 

\begin{proof}
We describe the proof for the case that $\nu=1$ writing $u= u_1$, $v= v_1$ etc. The general case follows easily by the observation that the \OpLem allow the manipulation of several (bunches of) markings on the same reduced power circuit in parallel~-- we give some details at the end of the proof.

We obtain a \redpc of $uv$ by concatenating the \redpcs of $u$ and $v$. Because \taligned words have nice cancellation properties,  we start by constructing a \redpc of a \taligned word $\tilde{w}$ such that $\tilde{w}\eqBSpr uv$ and $\len{\Theta}{\tilde{w}} \leq (\len{\Theta}{uv})^2$ (\cref{lem: MakeTalinged}). Let
\[\tilde{w}=a^{k_m}t^{x_{m}}a^{k_{m-1}}\cdots t^{x_1}a^{k_0+k'_0}t^{x'_1}a^{k'_1}\cdots t^{x'_n}a^{k'_n}. \] 
with $u=_{\DelBS} a^{k_m}t^{x_{m}}a^{k_{m-1}}\cdots t^{x_1}a^{k_0}$ and $v=_{\DelBS} a^{k'_0}t^{x'_1}a^{k'_1}\cdots t^{x'_n}a^{k'_n}$ (note that, to simplify the notation for this proof we changed the indices for $u$ compared to (\ref{wBRBS})).
For a \taligned word $w$ we have that
\begin{align}\label{FormTildeW}
\height{w}{i,j-1}=0 \text{ and } \sgn{x_i}=-\sgn{x_j}  \implies \abs{x_i}=\abs{x_j}.
\end{align} Thus, because $\tilde{w}$ is \taligned and $u$ and $v$ are both Britton-reduced, $\tilde{w}$  is Britton-reduced if and only if either $t^{x_1}$ does not cancel with $t^{x'_1}$ or $pq \nmid (k_0+k'_0)$.
 Even more, if $\tilde{w}$ is not Britton-reduced, $t^{x_i}$ can only cancel with $t^{x'_\ell}$  (for arbitrary $i$ and $\ell$) if $i=\ell$, $x'_\ell=-x_i$, and  $t^{x_j}$ cancels with $t^{x'_j}$ for all $j \leq i$ (note that, in general, this is not true for non-\taligned words). Because $u,v$ are Britton-reduced, these are the only possible Britton-reductions. 

We proceed as follows: For each $i\leq\min\{m,n\}$ we compute a bit $\lambda_i$ indicating whether  $t^{x_i}$ cancels with $t^{x'_i}$ under the assumption that  $t^{x_{i-1}}$ cancels with $t^{x'_{i-1}}$. Then we search for $i_0$ maximal such that $\lambda_i=1$ for all $i \leq i_0$. 

If $x_j \neq -x'_j$ for some $j \leq i$ then $t^{x_i}$ cannot cancel with $t^{x'_i}$ and we set $\lambda_i=0$. Otherwise, we set
\begin{align*}
z_i=a^{k_{i-1}}t^{x_{i-1}}\cdots t^{x'_{i-1}}a^{k'_{i-1}}.
\end{align*}
Now, using \cref{lem:AddMult}  we check whether $\numba{z_i} \in \Z$. If not, then $t^{x_{i-1}}$ cannot cancel with $t^{x'_{i-1}}$ and we set $\lambda_i=0$. If yes, we use the formula in \cref{def: defOfkij} and \cref{lem:AddMult} to calculate a marking $Z_i$ with  $\eps(Z_i)=\numba{z_i} $. We can check in $\uTC{0}$ if $\numba{z_i} \in p\Z$ and $x_i>0$ or $\numba{z_i} \in (pq^{-x_i})\Z$ and $x_i<0$. If we are in one of these two cases, then $t^{x_i}$ cancels with $t^{x'_i}$ under the assumption that  $t^{x_{i-1}}$ cancels with $t^{x'_{i-1}}$ and we set $\lambda_i=1$.
\smallskip

For each $i$, we use a separate checking power circuit we delete after the checking, only remembering the bits $\lambda_i$. We search for  the maximal $i_0$ such that for all $i \leq i_0$ we have $\lambda_i=1$. If $\lambda_i=0$ for all $i$, we set $i_0=0$ and $\ell_{0}=k_{0}+k'_{0}$. By the choice of $i_0$, we know that  $\ell_{i_0}=k_{i_0}+q^{x_{i_0}}\cdot\numba{z_{i_0}}+k'_{i_0} \in \Z$. We can calculate a marking $L_{i_0}$ with  $\eps(L_{i_0})=\ell_{i_0}$ using  one more  layer of \powertwo and \addition each. Then, 
\begin{align}\label{BRwithoutq}
\tilde{w} \eqBSpr a^{k_m}t^{x_{m}} \cdots t^{x_{i_0+1}} a^{\ell_{i_0}}t^{x'_{i_0+1}}\cdots t^{x'_n}a^{k'_n}. 
\end{align} 
If $t^{x_i}$ and $t^{x'_i}$ cancel, then $\height{w}{i,i'-1}=0 \text{ and } \sgn{x_i}=-\sgn{x'_i}$. Thus, by (\ref{FormTildeW})    the word in (\ref{BRwithoutq}) is still \taligned. So, as above and because of the choice of $i_0$, it is not Britton-reduced if and only if $x_{i_0+1}=-x'_{i_0+1}$, $x_{i_0+1}<0$ and $\ell_{i_0} \in (pq)\Z$. In this case, we   calculate a marking $S$ such that  $\eps(S)>0$ is maximal with $p \cdot q^{\eps(S)} \mid \ell_{i_0} $. By \cref{lem: powtwo}(\ref{MaxPow}) this is possible in $\uTC{0}$. Furthermore,  $\eps(S)< \abs{x_{i_0+1}}$, otherwise $i_0$ is not maximal.  Thus,
\begin{align*}
\tilde{w}\eqBSpr a^{k_m}t^{x_{m}}\cdots t^{x_{i_0+1}+\eps(S)}a^{ q^{-\eps(S)}\cdot\ell_{i_0}}t^{x'_{{i_0}+1}-\eps(S)} \cdots t^{x'_{n}}a^{k'_n}
\end{align*}
and the right side is Britton-reduced (if (\ref{BRwithoutq}) is already Britton-reduced,  set $S=\emptyset$). 
\smallskip

Making a word \taligned might increase its length over $\Theta$.
 To ensure that $\len{\Theta}{w} \leq \len{\Theta}{uv}$, we reverse this step as follows: For each $x_i$ (resp. $x'_i$) with $k_i \neq 0$ (resp. $k'_{i-1} \neq 0$) we search for $s$ (resp. $s'$) maximal such that $k_\mu =0$ for all $\mu \in \interval{i-1}{i-s}$ (resp. $\mu \in \interval{i}{i+s'-1}$).
  If $s$ (resp. $s'$)  does not exist, set $s=0$ (resp. $s'=0$). We replace $t^{x_i}\cdots t^{x_{i-s}}$ (resp. $t^{x'_i}\cdots t^{x'_{i+s'}}$) by $t^{\sum_{\mu=i}^{i-s}x_\mu}$ (resp. $t^{\sum_{\mu=i}^{i+s'}x'_\mu}$).  The \redpcs we obtain by \addition.
   When making a word \taligned, we split $t^x$ in parts $t^{y_i}$ with $\sgn{x}=\sgn{y_i}$ (\cref{lem: MakeTalinged}).
   Hence, because $u,v$ are Britton-reduced, all these  $x_i$ (or $x'_i$) have the same sign and, thus, do not cancel. 
   Britton-reductions do not increase the length over $\Theta$,  so the resulting word is Britton-reduced and  $\len{\Theta}{w}\leq \len{\Theta}{uv}$.  
\smallskip

During the whole construction, we only use a constant number of layers of \addition and \powertwo (including multiplications by powers of $-q$ due to \cref{rem:MultByMinQ}) and calculations of $\bmod~ p\cdot q^x$ (\cref{lem:moduloPowerOfTwoConstant}). For calculating modulo  we use separate checking power circuits we discard after the checking.  So the above algorithm is in $\uTC{0}$.

	For the size constraints,  the crucial observation is that, when constructing the marking $Z_{i_0}$ with $\eps(Z_{i_0})=\numba{z_{i_0}}$,   each marking representing $k_j, k'_j$, $j \in \interval{0}{i_0-1}$ appears at most once outside of an exponent (\ie as ``mantissa'') multiplied by a suitable power of $q$ in the formula in \cref{def: defOfkij} or when calculating $q^{s} \cdot \ell_{i_0}$ (if  we have both operations, we can do them at the same time). So, by \cref{lem:AddMult} it follows that  $\msumk{\mathcal{W}}\leq \msumk{\mathcal{U}}+\msumk{\mathcal{U}}$.
Moreover, all sums have length at most $\len{\Theta}{\tilde{w}}\leq\len{\Theta}{uv}^2$.
Now the size constraints follow by  the \OpLem and \cref{lem: MakeTalinged}.
\smallskip

  Concerning parallelism: By \cref{lem: add},  \cref{lem: floatpOp} and \cref{lem:AddMult}, \addition and \powertwo can be applied independently in parallel to several (bunches of) markings. So we can apply the above process to the words $u_iv_i$ for $i \in \interval{1}{\nu}$ independently in parallel, while when applying an operation to several (bunches of) markings at the same time we make sure that it is the same operation for all markings. This shows the Lemma. 
\qed\end{proof}

Note that the words $w_{red}$ (resp. $w$) we obtain as a result in \cref{lem: MakeTalinged} (resp. \cref{thm: BRinBS}) are not \taligned in general. 

\section{The word problem of \BGpq}\label{sec:WPinBG}
 The Baumslag group $\BGpr$ can be understood as an HNN extension of the Baumslag-Solitar group $\BS{p}{r}$ with stable letter $b$ and isomorphim $\phi: \gen{a} \rightarrow \gen{t}$, $\phi(a)=t$ (with $a,t$ being the generators of $\BS{p}{r}$):  

\begin{align*}
	\begin{split}
		\BGpr  & = \Gen{\BS{p}{r},b}{bab^{-1} = t}%\\ 
		 = \Gen{a,t,b}{ta^pt^{-1} = a^r, bab^{-1} = t}.
	\end{split}
\end{align*}
Due to $bab^{-1} = t$, we obtain  a presentation $\Gen{a,b}{bab^{-1} a^p = a^r bab^{-1}}$ as one-relator group. 
Moreover, $\BS{p}{r}$ is a subgroup of $\BGpr$ via the canonical embedding.
%

%\vspace{-1mm}
\paragraph*{Britton reductions.}
 Let us choose
\[\Delta = (\BS{p}{r}\setminus \{1\}) \cup \oneset{b, b^{-1}}\]
as (infinite) alphabet. 
Recall from \cref{sec:GroupTheory} that  a word $z \in \Delpr$ is called \emph{Britton-reduced} if there is no factor of the form $ba^kb^{-1}$ or $b^{-1}t^kb$ with $k \in \Z$ and if $z$ does not have two successive letters from $\BS{p}{r}$. If $z$ is not Britton-reduced, one can apply at least one of the rules
\begin{align} \label{BRinGpr}
	ba^kb^{-1} &\:\to\: t^k\\ \nonumber
	b^{-1} t^kb &\:\to\: a^k\\ \nonumber
	w_iw_{i+1}&\:\to\: [w_iw_{i+1}]
\end{align}
(see also \cref{sec:GroupTheory})
 in order to obtain a shorter word (over the alphabet $\Delta$) representing the same group element. Again, we write $a^0$, $t^0$ as abbreviations for the empty word. If we can apply  a Britton-reduction as in (\ref{BRinGpr}), we say that $\beta$ \emph{cancels with} $\beta^{-1}$ for $\beta \in \{b,b^{-1}\}$. 
Now, Britton's Lemma can be rephrased as:

\begin{lemma}[Britton's Lemma for $\BGpr$] \label{lem: BRinGpr}
	Let $z \in \Delta^*$ be Britton-reduced. Then $z \inBGpr \BS{p}{r}$ if and only if $z$ does not contain any letter $b, b^{-1}$. In particular, $z\eqBGpr 1$ if and only if $z\eqBSpr 1$. 
\end{lemma} 

 In the following we mainly consider words $z \in \Delta^*$ of the form
\begin{align}\label{WordsInGpr}
z=w_0\beta_1w_1\beta_2w_2 \cdots \beta_n w_n
\end{align}
with $\beta_i \in \{b, b^{-1}\}$ and  $w_i\in\BS{p}{r}$. Note that here we allow that $w_i\eqBSpr 1$.
Moreover, we will use the following notations: 
For $u, v \in \Delpr$, 
\begin{align}
u = w_h\beta_{h}  \cdots w_1  \beta_{1}w_0,  \quad v = w'_0  \beta'_{1} w'_1 \cdots  \beta'_{\ell}w'_\ell\label{betaFact}
\end{align}
with $w_j, w'_j\in \BS{p}{r}$ and $\beta_{j},\beta'_{j} \in \oneset{b, b^{-1}}$

we define 
\begin{align} 
uv[i,j] =  \beta_{i+1}w_i \cdots \beta_{1}[w_0w'_0]  \beta'_{1}  \cdots  w'_j \beta'_{j+1}.\label{betaFactIJ}
\end{align} 
For $w \in \BS{p}{r}$, $\hght{w}$ is defined as in \cref{def: height} (note that $\hght{w}$ does not change under Britton-reductions, so this is well-defined). We write $\abs{u}_\beta$ for the number of $\beta\in \{b,b^{-1}\}$ occurring in $u$. 
\smallskip

In the following, we also use the representation as words over the alphabet
$\Sigma=\{a, a^{-1},b,b^{-1},t,t^{-1}\}$ (with $a,b,t$ the generators of $\BGpq $), e.g. when considering the word problem of $\BGpr$ in \cref{cor:TC1}. We also use this representation in the following example: 

\begin{example}\label{ex:blowup}
	Let $\pcBase \geq 2$. We define words 
	$w_0 = t$ and $w_{n+1} = b\,  w_{n}\,  a^p\, w_{n}^{-1}\, b^{-1}$
	for $n \geq 0$ with $w_n \in \BS{p}{pq}$ for all $n \geq 0$. 
	Then we can prove by induction on $n$ that $ \len{\Sigma}{w_n} \leq p( 2^{n+2} -3)$ but
	$w_{n} \eqBGpq t^{k}$ with $k \geq \tow_\pcBase(n)$.
	 For $p=1$, indeed, $w_{n} =_{\BGq} t^{\tow_\pcBase(n)}$. While the length of the word $w_n$ is only exponential in $n$, the length of its \Breduced 
	form is $\tow_\pcBase(n)$ (both  over $\Sigma$). 
\end{example}
	
\newcommand{\pg}{g}

\subsection{Conditions for Britton-reductions in $\BGpq$}

As in \cite{MattesWCC23} and \cite{MattesW22}, the idea to obtain a parallel algorithm for the word problem is to compute a Britton reduction of $uv$ given that both $u$ and $v$ are Britton-reduced. For this, we have to find a maximal suffix of $u$ that cancels with a prefix of $v$, \ie we search for the maximal $i_0$ such that $uv[i_0,i_0] \inBGpq \BSpq$. The formulas we prove in Lemma \ref{lem: TableCond}  and \cref{lem: FormulaAltW} below  are our main tools for finding such an $i_0$  and to obtain $uv[i_0,i_0]$ as a word in  $\BSpq$.

\begin{lemma}\label{lem: TableCond}
Let \[w=\beta_{i+1} w_i \beta_{i} x\beta_{i}^{-1}w_i'\beta_{i+1}^{-1}\] with $\beta_{i+1}, \beta_{i} \in \{b, b^{-1}\}$  and $w_i, w_i',x \in \BS{p}{pq} $ such that $\beta_{i+1} w_i \beta_{i} $ and $\beta_{i}^{-1}w_i'\beta_{i+1}^{-1}$ are Britton-reduced and $\beta_{i} x\beta_{i}^{-1}\inBGpq \BS{p}{pq}$. Then $w \in_{\BGpq} \BS{p}{pq}$ if and only if the respective condition in the following table is satisfied. Moreover, if $w\inBGpq \BS{p}{pq}$ then $w=_{\BGpq}\hat{w}$ according to the last column of the table.

	\begin{center}
		\renewcommand{\arraystretch}{1.2}
		\vspace{-1mm}
		\begin{tabular}[h]{cc|c|l}	
			$\beta_{i+1}$ & $\beta_{i}$ 
			& Condition  & \multicolumn{1}{c}{$\hat w$} \\
				\hline 				
		
			$b$ & $b$ & \hspace*{0.2mm}   $w_it^yw_i'\in_{\BGpq} \gen{a}$  & \hspace*{0.2mm}  $t^z$, $z=\numba{w_it^{-\hght{w_iw'_i}}w_i'}$ \\[0.8mm]
			
			$b^{-1}$ & $b^{-1}~$ & \hspace*{0.2mm} $w_ia^{y}w_i'\in_{\BGpq} \gen{t}$& \hspace*{0.2mm}  $a^z$, $z=\hght{w_iw'_i} $ \\[0.8mm]

			$b$ & $b^{-1}$  &\hspace*{0.2mm} $w_ia^{y}w_i' \in_{\BGpq} \gen{a}$  & \hspace*{0.2mm} $t^z$, $z=\numba{w_ia^{y}w_i'}$\\[0.8mm]
			
			$b^{-1}$ & $b$  &\hspace*{0.2mm} $w_it^{y}w_i' \in_{\BGpq} \gen{t}$& \hspace*{0.2mm} $a^z$, $z=\hght{w_iw'_i}+y$\\[0.8mm]
			
		\end{tabular}
	\vspace{-.5mm}
	\end{center}

\end{lemma}
\begin{proof}	We distinguish the two cases $\beta_{i} = b$ and $\beta_{i} = \oi{b}$. Each case consists of two sub-cases depending on $\beta_{i+1}$.
\smallskip

\textbf{Case $\beta_{i}=b$:} 
In this case, $\beta_{i} x\beta_{i}^{-1}=_{\BGpq}  t^y$ and $w=_{\BGpq}\beta_{i+1}w_it^{y}w_i'\beta_{i+1}^{-1}$ for some $y \in \Z$. 
If $\beta_{i+1}=b$, then $\beta_{i+1}$ cancels with $\beta_{i+1}^{-1}$ if and only if  
$w_it^{y}w_i' \in_{\BGpq} \gen{a}$. So $w \in_{\BGpq} \BS{p}{pq}$ implies that  $\hght{w_it^{y}w_i'}=\hght{w_iw'_i}+y=0$. So, $w\eqBSpr t^z$ with $z=\numba{w_it^{-\hght{w_iw'_i}}w_i'}$.  

If $\beta_{i+1}=b^{-1}$ then $\beta_{i+1}$ cancels with $\beta_{i+1}^{-1}$ if and only if $w_it^{y}w_i' \inBGpq \gen{t}$. So if $w \in_{\BGpq} \BS{p}{pq} $ then $w_it^{y}w_i'\eqBSpr t^{\hght{w_iw'_i}+y}$ ($\hght{w}$ does not change under Britton-reductions in $\Thepr$). Thus, $w =_{\BGpq}a^{\hght{w_iw'_i}+y}$. 
\medskip

\textbf{Case $\beta_{i}=b^{-1}$:} There is $y \in \Z$ such that $\beta_{i} x\beta_{i}^{-1}=_{\BGpq}  a^y$  and $w=_{\BGpq}\beta_{i+1}w_ia^{y}w_i'\beta_{i+1}^{-1}$. If $\beta_{i+1}=b$ then $\beta_{i+1}$ cancels with  $\beta_{i+1}^{-1}$ if and only if $w_ia^{y}w_i' \in_{\BGpq} \gen{a}$. So if $w \in_{\BGpq}  \BS{p}{pq} $ then $\beta_{i+1}w_ia^{y}w_i'\beta_{i+1}^{-1} \eqBGpq t^z$ with $z={\numba{w_ia^{y}w_i'}}$. 

In the case $\beta_{i+1}=b^{-1}$
we have that $\beta_{i+1}$ cancels with  $\beta_{i+1}^{-1}$ if and only if  $w_ia^{y}w_i' \in_{\BGpq} \gen{t}$. Thus if $w \in_{\BGpq} \BS{p}{pq} $ then $w_ia^{y}w_i'\eqBSpr t^{\hght{w_iw'_i}}$ because $\hght{w}$ does not change under Britton-reductions in $\Thepr$. Hence, $w=_{\BGpq}a^{\hght{w_iw'_i}}$.
\qed\end{proof}

Let $u,v \in \BGpq$. For the algorithm solving the word problem in \cref{thm:TC1PC} we need to check if $\beta_{i+1}$ cancels with $\beta^{-1}_{i+1}$  under the assumption that  $\beta_{i}$ cancels with $\beta^{-1}_{i}$. This is possible by checking the conditions in \cref{lem: TableCond}: If $\beta_{i}=\beta_{i+1}$, \cref{lem: TableCond} shows how to calculate $uv[i,i]$ if $uv[i,i] \inBGpq \BS{p}{pq}$. Otherwise, we  need to calculate $y$ in the table in \cref{lem: TableCond}. The formula in \cref{lem: FormulaAltW} gives us a tool for that:   If $i_0$ is maximal such that $\beta_{i_0}=\beta_{i_0+1}$, then for $i > i_0$ we have an alternating sequence of $b$ and $b^{-1}$. So in the following we consider words
\begin{align}\label{symmW}
    uv[i,i]=\underbrace{\beta w_i\beta^{-1}w_{i-1}\beta\cdots \beta^{\pm 1}[w_0}_{u}\underbrace{w'_0]\beta^{\mp 1} \cdots \beta^{-1} w_{i-1}'\beta w_i' \beta^{-1}}_{v}
\end{align}
with $\beta \in \{b,b^{-1}\}$, $w_i,w'_i \in \BSpq$ and $u,v$ Britton-reduced. 
If $uv[i,i] \inBGpq \BSpq$ then, because $u,v$ are Britton-reduced,  $uv[\ell,\ell] \inBGpq \BS{p}{pq}$ for every $0\leq \ell  \leq i$.  
In particular,  
\begin{align}\label{FormYi}
    w_\ell \beta^{- 1}w_{\ell-1}\cdots w_{\ell-1}' \beta w_\ell'=_{\BGpq}\alpha^{y_{\ell}}
\end{align}
with $y_{\ell} \in \mathbb{Z}$ and $\alpha \in \{a,t\}$. First, consider the case  $\beta =b$. Since $\beta^{-1}$ cancels with $\beta$,  $w_{\ell-1}uv[\ell-2,\ell-2]w'_{\ell-1} \inBGpq \gen{t}$  and thus
\[a^{y_{\ell}}=_{\BGpq}w_\ell b^{-1}t^{y_{\ell-1}} b w_\ell'=_{\BGpq}w_\ell a^{y_{\ell-1}}w_\ell'.\]
By \cref{lem: ConcOfwv}  we obtain that \begin{align*}
    y_{\ell}     &=\numba{w_\ell}+q^{\hght{w_\ell}}\left(y_{\ell-1}+ \numba{w'_\ell}\right).
\end{align*}
As an abbreviation, we write 
%$\chi_i=\sum_{\mu=1}^{n_i}x_\mu$ and
$\kappa^{(\ell)}=\numba{w_{\ell}}+q^{\hght{w_\ell}}\cdot \numba{w'_{\ell}} $.
\medskip

If $\beta =b^{-1}$ then,  since $\beta^{-1}$ cancels with $\beta$,  $w_{\ell-1}uv[\ell-2,\ell-2]w'_{\ell-1} \inBGpq´ \gen{a}$  and thus
\[t^{y_{\ell}}=_{\BGpq}w_\ell ba^{y_{\ell-1}} b^{-1} w_\ell'=_{\BGpq}w_\ell t^{y_{\ell-1}}w_\ell'.\]
Observe that $\hght{}$ does not change under Britton-reductions in \BS{p}{pq}. Thus, 
\[y_{\ell}=\hght{w_\ell}+\hght{w'_\ell}+y_{\ell-1} =\hght{w_\ell w'_\ell}+y_{\ell-1}, \text{ and}\]
\begin{align} \label{YiOddorEven} 
  y_{i-\ell} &= \hght{w_{i-\ell} w'_{i-\ell}}+y_{i-(\ell+1)}  &\text{ if } \beta=b \text{ and } \ell \text{  odd};    \text{ or } \beta=b^{-1} \text{ and } \ell \text{  even}  \\[2mm]
\nonumber    y_{i-\ell} &= y_{i-(\ell+1)}\cdot q^{\hght{w_{i-\ell}}}+\kappa^{(i-\ell)} &\text{ if }  \beta=b \text{ and } \ell \text{  even};  \text{ or  } \beta=b^{-1} \text{ and } \ell \text{ odd}.
\end{align}

\newcommand{\innerIndex}{\theta}
\newcommand{\innerInnerIndex}{\zeta}
\newcommand{\namesum}{\kappa}
 \newcommand{\XiMu}{\text{ \tiny $\ceil{\frac{\mu+1}{2}}$}}
 \newcommand{\FXiMu}{\text{ \tiny $\floor{\frac{\mu+1}{2}}$}}
\begin{lemma}\label{lem: FormulaAltW}
Let $u,v \in \Delpr$ be Britton-reduced and
\[uv[i,i]=bw_ib^{-1}w_{i-1}b\cdots b^{\pm 1}w_0w'_0b^{\mp 1} \cdots b^{-1} w_{i-1}'b w_i'b^{-1}\]
be as in (\ref{symmW}). If $uv[\ell-1,\ell-1] \inBGpq \BSpq$ we define $y_\ell$ as in (\ref{FormYi}). The following holds for all $0\leq \mu \leq i$: If $uv[i,i] \inBGpq \BS{p}{pq}$, then $uv[i,i]\eqBGpq t^z$ with
%\begin{enumerate}[(a)]
\begin{align} \label{FormulaUV}
z=\sum_{\innerIndex=1}^{\FXiMu}\hght{w_{\ell_{\innerIndex}}w'_{\ell_{\innerIndex}}} \cdot q^{\xi_{\innerIndex}} +\sum_{\innerIndex=1}^{\XiMu}\kappa^{(\ell_{\innerIndex}+1)}\cdot q^{\xi_{\innerIndex-1}}+y_{i-(\mu+1)} \cdot q^{\xi_{\XiMu}}
\end{align}
with $\ell_{\innerIndex}=i-(2\innerIndex-1)$ and  $\xi_{\innerIndex}=\sum_{j=1}^{\theta} \hght{w_{\ell_{j}+1}}$ and $y_{\mOne}=0$.

%and $\nu=\ceil{\frac{\mu}{2}}$
%\item If $\beta=b^{-1}$ then $w=a^z$ with 
%\begin{align*}
    %z=\sum_{\innerIndex=1}^{\ceil{\frac{\mu}{2}}}\left(\hght{w_{\ell_{\innerIndex}+1}}+\hght{w'_{\ell_{\innerIndex}+1}} \right) \cdot q^{\zeta_{\innerIndex-1}} +\sum_{\innerIndex=1}^{\floor{\frac{\mu}{2}}}\kappa^{(\ell_{\innerIndex})}\cdot q^{\zeta_{\innerIndex}-1}+y_\mu \cdot q^{\zeta_{\floor{\frac{\mu}{2}}}}
%\end{align*}
%with $\zeta_{\innerIndex}=\sum_{\ell=1}^{\theta}\hght{w_{\ell_{\innerIndex}}}$.
%\end{enumerate}
\end{lemma}
%So, if $uv[i,i] \inBGpq \BS{p}{pq}$ and $i_0$ is maximal such that $\beta_{i_0}=\beta_{i_0+1}$ we can calculate $y_{i_0+1}$ using \cref{lem: TableCond} and then use this result to calculate $uv[i,i]$ with the formula in \cref{lem: FormulaAltW}. Observe that if $uv[i,i]$ starts with $b^{-1}$ then we  can apply \cref{lem: FormulaAltW} to $uv[i-1,i-1]$ and then apply \cref{lem: TableCond}.

\begin{proof} 
We proof this by induction over $\mu$. By assumption,  $uv[i,i] \in_{\BGpq} \BS{p}{pq}$ and $u$ and $v$ are Britton-reduced. 
Thus,
$ t^{y_i}\eqBGpq bw_ib^{-1}t^{y_{i-1}}bw_i'b^{-1}$. By (\ref{YiOddorEven}), $y_i=\kappa^{(i)}+q^{\hght{w_i}}\cdot y_{i-1}.$
For $\mu=0$ this is the right sum in (\ref{FormulaUV}) and the left one is empty. This shows the formula for $\mu=0$. Now we assume that (\ref{FormulaUV}) is correct for $\mu$. We consider the cases $\mu$ odd and $\mu$ even separately.\\

If $\mu$ is odd, then $\ceil{\frac{\mu+2}{2}}=\frac{\mu+3}{2}$. So, $i-(2 \cdot \ceil{\frac{\mu+2}{2}}-1)+1=i-(\mu+1)$. Moreover, $\floor{\frac{\mu+1}{2}}=\floor{\frac{\mu+2}{2}}$. By induction and (\ref{YiOddorEven}), $uv[i,i]\eqBGpq t^z$ with  
\begin{align*}
   z&= \sum_{\innerIndex=1}^{\FXiMu}\hght{w_{\ell_{\innerIndex}}w'_{\ell_{\innerIndex}}} \cdot q^{\xi_{\innerIndex}} +\sum_{\innerIndex=1}^{\XiMu}\kappa^{(\ell_{\innerIndex}+1)}\cdot q^{\xi_{\innerIndex-1}}+y_{i-(\mu+1)} \cdot q^{\xi_{\XiMu}}\\
   &=\sum_{\innerIndex=1}^{\FXiMu}\hght{w_{\ell_{\innerIndex}}w'_{\ell_{\innerIndex}}} \cdot q^{\xi_{\innerIndex}} +\sum_{\innerIndex=1}^{\XiMu}\kappa^{(\ell_{\innerIndex}+1)}\cdot q^{\xi_{\innerIndex-1}} \\
   &\qquad +(y_{i-(\mu+2)}\cdot q^{\hght{w_{i-(\mu+1)}}} +k^{(i-(\mu+1))}) \cdot q^{\xi_{\XiMu}}\\[1mm]
   &=\sum_{\innerIndex=1}^{\text{ \tiny $\floor{\frac{\mu+2}{2}}$}}\hght{w_{\ell_{\innerIndex}}w'_{\ell_{\innerIndex}}} \cdot q^{\xi_{\innerIndex}} +\sum_{\innerIndex=1}^{\text{ \tiny $\ceil{\frac{\mu+2}{2}}$}}\kappa^{(\ell_{\innerIndex}+1)}\cdot q^{\xi_{\innerIndex-1}}+
   y_{i-(\mu+2)}\cdot q^{\xi_{\tiny \ceil{\frac{\mu+2}{2}}}} 
\end{align*}

%\begin{align*}
%z&=\sum_{i=1}^{\floor{\frac{\mu}{2}}}\left(\abs{w_{2i}}_t+\abs{w_{2i}'}_t \right) \cdot q^{\xi(i)} +\sum_{i=1}^{\ceil{\frac{\mu}{2}}}\kappa^{(2i-1)}\cdot q^{\xi(i-1)}+(\abs{w_{\mu+1}}_t+\abs{w_{\mu+1}'}_t+y_{\mu+1}) \cdot q^{\xi(\ceil{\frac{\mu+1}{2}})}\\
%&=\sum_{i=1}^{\floor{\frac{\mu}{2}}}\left(\abs{w_{2i}}_t+\abs{w_{2i}'}_t \right) \cdot q^{\xi(i)} +\sum_{i=1}^{\ceil{\frac{\mu+1}{2}}}\kappa^{(2i-1)}\cdot q^{\xi(i-1)}+(\abs{w_{\mu+1}}_t+\abs{w_{\mu+1}'}_t+y_{\mu+1}) \cdot q^{\xi(\floor{\frac{\mu+1}{2}})}\\
%&=\sum_{i=1}^{\floor{\frac{\mu+1}{2}}}\left(\abs{w_{2i}}_t+\abs{w_{2i}'}_t \right) \cdot q^{\xi(i)} +\sum_{i=1}^{\ceil{\frac{\mu+1}{2}}}\kappa^{(2i-1)}\cdot q^{\xi(i-1)}+y_{\mu+1} \cdot q^{\xi(\ceil{\frac{\mu+1}{2}})}
%end{align*}
If $\mu$ is even, then $\ceil{\frac{\mu+1}{2}}=\ceil{\frac{\mu+2}{2}}=\floor{{\frac{\mu+2}{2}}}$ and $2\cdot \floor{\frac{\mu+2}{2}}-1=\mu+1$. Again  using  induction and (\ref{YiOddorEven}) we obtain that $uv[i,i]\eqBGpq t^z$ with 
\begin{align*}
     z&= \sum_{\innerIndex=1}^{\floor{\frac{\mu+1}{2}}}\hght{w_{\ell_{\innerIndex}}w'_{\ell_{\innerIndex}}}  \cdot q^{\xi_{\innerIndex}} +\sum_{\innerIndex=1}^{\ceil{\frac{\mu+1}{2}}}\kappa^{(\ell_{\innerIndex}+1)}\cdot q^{\xi_{\innerIndex-1}}+y_{i-(\mu+1)} \cdot q^{\xi_{\ceil{\frac{\mu+1}{2}}}}\\
     &= \sum_{\innerIndex=1}^{\floor{\frac{\mu+1}{2}}}\hght{w_{\ell_{\innerIndex}}w'_{\ell_{\innerIndex}}} \cdot q^{\xi_{\innerIndex}} +\sum_{\innerIndex=1}^{\ceil{\frac{\mu+1}{2}}}\kappa^{(\ell_{\innerIndex}+1)}\cdot q^{\xi_{\innerIndex-1}}\\
     &\qquad+(\hght{w_{i-(\mu+1)}w'_{i-(\mu+1)}}+y_{i-(\mu+2)}) \cdot q^{\xi_{\ceil{\frac{\mu+1}{2}}}}\\[1mm]
     &= \sum_{\innerIndex=1}^{\floor{\frac{\mu+2}{2}}}\hght{w_{\ell_{\innerIndex}}w'_{\ell_{\innerIndex}}}  \cdot q^{\xi_{\innerIndex}} +\sum_{\innerIndex=1}^{\ceil{\frac{\mu+2}{2}}}\kappa^{(\ell_{\innerIndex}+1)}\cdot q^{\xi_{\innerIndex-1}}+y_{i-(\mu+2)} \cdot q^{\xi_{\ceil{\frac{\mu+2}{2}}}}\\
\end{align*}
    This shows the lemma. 
\qed\end{proof}

\subsection{Power circuit representations}
In this and the following sections we deal with Britton reductions both in the Baumslag group and the Baumslag-Solitar group. We write \BSpq-Britton reduction (resp.\ \BSpq-Britton-reduced) for Britton reductions in the Baumslag-Solitar group $\BSpq$ over the alphabet $\Theta$, while simply Britton reduction (resp.\ Britton-reduced) refers to Britton reductions in the Baumslag group over the alphabet $\Delta$.

We also want to represent elements $w \in \Delpr$ using markings in reduced power circuits. 
 So we need the following definition: 
\begin{definition}\label{def: redPCBG}
Let $(\Gamma, \delta)$ be a reduced power circuit. 
\begin{enumerate}[(a)]
        \item \label{PCreprDel} Let $u=u_1 \cdots u_k \in \Delpr$. We call $\mathcal{U}=(\mathcal{W}_i, B_i ,T_i)_{i \in \interval{1}{k}}$ a \redpc (reduced power circuit representation) of $u$ over $(\Gamma, \delta)$ if the following holds for all $i \in \interval{1}{k}$: 
    \begin{itemize}
       \item If $u_i \in \BS{p}{pq}\setminus \{1\}$, then $\mathcal{W}_i=(X^{(i)}_\ell,K^{(i)}_\ell)_{\ell}$ is a \redpc of  a \BSpq-Britton-reduced word over $\Theta$ representing $u_i$ over $(\Gamma, \delta)$  and $\mathcal{W}_i$ is the empty sequence otherwise, 
        \item $B_i=u_i$ if $u_i \in \{b, b^{-1}\}$ and $B_i=\$$ otherwise, 
        \item $T_i$ is a compact marking on $\Gamma$ with $\eps(T_i)=\hght{u_i}$ if $u_i \in \BS{p}{pq}$ and $T_i=\emptyset$ otherwise. 
    \end{itemize} 
       \item Let $\mathcal{U}$ be a \redpc  of $u=u_1 \cdots u_k \in \Delpr$. We define  $\msumkt{\mathcal{U}}=\sum_{i=1}^k \left(\abs{\sigma(T_i)}+\msumk{\mathcal{W}_i} \right)$.
\item For a \redpc $\mathcal{U}$ of a word $u \in \Delpr$ as in (\ref{PCreprDel}) we write $\len{\Theta}{\mathcal{U}}=\sum_{i=1}^{k} \len{\Theta}{\mathcal{W}_i}$.  

 \end{enumerate}
\end{definition}

Note that, as in \cref{def: PCrepBS}, the term ``reduced'' in  \redpc refers to the power circuit (and does not mean that the word $z \in \Delpr$ has to be Britton-reduced). However, be aware that each single $u_i \in \BS{p}{pq}\setminus \{1\}$ is, indeed, required to be represented by a \emph{$\BSpq$-Britton-reduced} \redpc.

\begin{remark}\label{rem:Tnecessary}
To \BSpq-Britton-reduce a word $uv \in \Thepr$ with $u,v \in \Thepr$ already \BSpq-Britton-reduced, we first make $uv$ \taligned, then \BSpq-Britton-reduce the resulting word and then ``reverse'' making it \taligned (\cref{lem:BRinBSinTC}). We might add up to $(\len{\Theta}{uv})^2$ markings, so this might increase $\sum \abs{\sigma(X_i)}$.    In the formulas in \cref{lem: FormulaAltW} and \cref{lem: TableCond}  we only need $\hght{w_i}$ and not $\eps(X_j)$. Moreover, $\hght{w_i}$ does not change under \BSpq-Britton-reductions, so we can represent $\hght{w_i}$ by the marking $T_i$ instead of using $\sum X_j$. This helps us to keep the power circuits small in our Britton-reduction algorithm in \cref{lem:Bredstep}  below.  
\end{remark}

\subsection{The algorithm for the word problem of $\BGpq$}

\begin{lemma} \label{lem:Bredstep}
The following is in $\uTC{0}$: \compproblem{\Redpc's $\mathcal{U}^{(i)}=(\mathcal{W}^{(u,i)}_\ell, B_\ell^{(u,i)},T_\ell^{(u,i)})_{\ell \in \interval{1}{\len{\Delta}{u_i}}}$ and $\mathcal{V}^{(i)}=(\mathcal{W}^{(v,i)}_\ell,  B_\ell^{(v,i)},T_\ell^{(v,i)})_{\ell \in \interval{1}{\len{\Delta}{v_i}}}$ of Britton-reduced words $u_i,v_i \in \Delpr$ over $(\Gamma, \delta)$ for  $i \in \interval{1}{\nu}$. }{\Redpc's $\mathcal{Z}^{(i)}=( \mathcal{W}^{(z,i)}_{\ell}, B_\ell^{(z,i)},T_\ell^{(z,i)})_{\ell \in \interval{1}{\len{\Delta}{z_i}}}$ for  Britton-reduced words $z_i \in \Delpr$ with $z_i=_{\BGpq}u_iv_i$ over $(\Gamma', \delta')$  for  $i \in \interval{1}{\nu}$ such that 
\begin{itemize}
    \item $\abs{\Gamma'} \leq \abs{\Gamma}+c\cdot\left( \mathcal{S}+\abs{\mathcal{C}_{\Gamma}}\right)\cdot \ceil{\log(\mathcal{M})}$
    \item $\abs{\mathcal{C}_{\Gamma'}}\leq \abs{\mathcal{C}_{\Gamma}}+c \cdot \mathcal{S} $
    \item $\msumkt{\mathcal{Z}^{(i)}} \leq \msumkt{\mathcal{U}^{(i)}}+\msumkt{\mathcal{V}^{(i)}}$ for each $i \in \interval{1}{\nu}$
    \item $ \len{\Theta}{\mathcal{Z}^{(i)}} \leq \len{\Theta}{\mathcal{U}^{(i)}}+\len{\Theta}{\mathcal{V}^{(i)}}$ for each $i \in \interval{1}{\nu}$.
\end{itemize}
for some constant $c$ and $\mathcal{S}=\sum_{\ell=1}^{\nu}\left(\msumkt{\mathcal{U}^{(i)}}+\msumkt{\mathcal{V}^{(i)}}\right)$ and $ \mathcal{M}=\max_{i \in \interval{1}{\nu}}(\len{\Theta}{\mathcal{U}^{(i)}}+\len{\Theta}{\mathcal{V}^{(i)}})$.
} 
\end{lemma}

Note that the last two points mean, in particular, that $\sum_{\ell=1}^{\nu}\msumkt{\mathcal{Z}^{(i)}} \leq \mathcal{S}$
and $ \max_{i \in \interval{1}{\nu} }(\len{\Theta}{\mathcal{Z}^{(i)}}) \leq \mathcal{M}$.

\begin{proof}
	\newcommand{\bit}{\lambda}
As we did for \cref{lem:BRinBSinTC}, we describe the proof for the case that $\nu=1$ writing $u= u^{(1)}$, $v= v^{(1)}$ etc.\ with $u,v$ written as in (\ref{betaFact}) and corresponding \redpcs $\cU$, $\cV$.
 The general case follows then by the observation that, by \cref{lem: add} and \cref{lem: floatpOp}, the operations \addition and \powertwo allow the manipulation of several (bunches of) markings on the same reduced power circuit independently parallel (see also the remark at the end of the proof of \cref{lem:BRinBSinTC}). 

For $u,v,w_i, uv[i,i]$ etc.\ we will use the notation we introduced in  (\ref{betaFact}) and (\ref{betaFactIJ}). Be aware that  (\ref{betaFact}) and (\ref{betaFactIJ}) use different indices as the statement of the lemma.
Moreover, by $T^{(u)}_j$ (resp $T^{(v)}_j$) we denote the $T$ in  the \redpc corresponding to $w_j$ (resp. $w'_j$)~-- again be aware of the change in the indices. 
In order to find the Britton reduction for $uv$, we need to find the maximal $i$ such that $uv[i,i] \inBGpq \BS{p}{pq}$. 

This approach is similar to \cite[Lemma 5.9]{MattesWCC23}.

The algorithm consists of three steps.
\begin{enumerate}
	\item For each $i$ we compute a compact marking $Y_i$ such that, if $uv[i,i] \inBGpq \BS{p}{pq}$, then $uv[i,i]\eqBGpq\alpha^{\eps(Y_i)}$ for $\alpha \in \{a,t\}$. Note that, if $uv[i,i] \not\inBGpq \BS{p}{pq}$, then $Y_i$ is still a marking but we do not care about its evaluation. If $uv[i,i] \inBGpq \BS{p}{pq}$, then, since $u,v$ are Britton-reduced, $uv[i,i]$ is indeed in $\gen{a}\cup \gen{t}$.   This is the step where most of the work of the algorithms happens.
	\item For each $i$ we compute a boolean value $\bit_i$ such that $\bit_i$ is true if and only if $uv[i-1,i-1] \inBGpq \BS{p}{pq}$ implies $uv[i,i] \inBGpq \BS{p}{pq}$ (thus, again, if $uv[i,i] \not\inBGpq \BS{p}{pq}$, then we do not care about the value of $\bit_i$).   
	\item We search for the maximal $i_0$ such that $\bit_j$ is true for all $j \leq i_0$ and return a \redpc for the word $z=\beta_n w_{n}\cdots \beta_{i_0+2} z_{i_0 + 1}\beta'_{i_0+2} \cdots w'_{n} \beta'_n$ where $z_{i_0 + 1}\eqBSpr w_{i_0+1}\alpha^{\eps(Y_{i_0})}w'_{i_0+1}$.
\end{enumerate}

We will perform these computations by using 
a constant number of (sequential) applications of the \OpLem, which allow addition and multiplication by the power circuit base independently in parallel for several (bunches of) markings.

\paragraph{Step 1.}
Our goal is to compute a compact marking $Y_i$ for each $i $ such that $uv[i,i] \inBGpq \BS{p}{pq}$ implies $uv[i,i]\eqBGpq\alpha^{\eps(Y_i)}$ for $\alpha \in \{a,t\}$  (depending on $\beta_{i+1}$).
We compute $Y_i$ for each $i$ in a different copy of our power circuit. In our computations, we use the markings $T_\ell^{(\mu)}$  to represent $\hght{w^{(\mu)}_\ell}$, $\mu \in \{u,v\}$ instead of sums of the $X^{(\mu, \ell)}_{r}$ (which is important as outlined in  \cref{rem:Tnecessary}).

Let us write $y_i = \eps(Y_i)$.
If $\beta'_{j} \neq \beta_j ^{-1}$ for some $j \leq i+1$, as both $u$ and $v$ are Britton-reduced, we know that $uv[i,i] \not\inBGpq \BS{p}{pq}$ and, hence, can safely set $Y_i = \emptyset$.
Otherwise, we use (the rightmost column of) \cref{lem: TableCond} and, if necessary, \cref{lem: FormulaAltW}:

% To be able to do this, we need to construct a compact marking $Y_{i-1}$ such that $\eps(Y_{i-1})=y_{i-1}$ (we use then as $y$ in the second column in \cref{lem: TableCond}).

We first search for the maximal $j \in \interval{1}{i} $ such that $\beta_{j+1}=\beta_{j}$.
 If no such $j$ exists, we take $j=-1$ and set $uv[-1,-1]= 1$. So we are in one of the first two cases of the table in \cref{lem: TableCond}.
 In the first case ($\beta_{j+1}=\beta_{j} = b$), we can compute $y_j = \cA(w_j t^{-\hght{w_jw'_j}} w_j')$ and in the second case (\ie $\beta_{j+1}=\beta_{j} = b^{-1}$), we have $y_j = \hght{w_jw'_j}$. Now, $\hght{w_jw'_j}$ can be computed as $\eps(T^{(u)}_{j})+\eps(T^{(v)}_{j})$ and then $\cA(w_j t^{-\hght{w_jw'_j}} w_j')$ via  \cref{lem:computeAh} using constantly many layers of \addition and \powertwo. 
   Thus, under the assumption $uv[j,j] \inBGpq \BS{p}{pq}$, using  constantly many layers of \addition and \powertwo, we can compute a marking $Y_j$ such that $uv[j,j]\eqBGpq\alpha^{\eps(Y_j)}$.

If $j=i$ we are finished with Step 1. 
Otherwise, first assume that $\beta_{i+1} = b$.
 %
 %For the \redpc of $\tilde{w}$ over $\Delta$ we have to construct a marking $T$ with
%
   Now we apply \cref{lem: FormulaAltW} to $\tilde w = bw_{i}\cdots \beta_{j+2}w_{j+1}\alpha^{\eps(Y_j)}w'_{j+1}\beta'_{j+2}\cdots w'_{i}b^{-1}$. % with $k=i$ if $\beta_i=b$ and $k=i-1$ if $\beta_i=b^{-1}$. 
   The lemma tells us that from $Y_j$ and the markings $K^{(\mu,\ell)}_{r}$,  $T^{(\mu)}_{\ell}$, (for $\mu \in \{u,v\}$, $\ell \in \interval{j+1}{i}$ and suitable $r$s), we can compute a marking $Y_i$ with $\tilde w \eqBGpq t^{\eps(Y_i)}$ by a layer of  \addition operations followed by a layer of \powertwo operations and then again a layer of \addition operations.
 
 Finally, if $\beta_{i+1} = b^{-1}$, we first compute $Y_{i-1}$ and then we obtain $Y_{i}$ by again using the last column of  \cref{lem: TableCond} (note that for $Y_{i-1}$ we are in the case $\beta_{i}=b$).

\paragraph{Complexity of Step 1.}
Looking at the formulas we see that we only need constantly many layers of \addition and \powertwo (including multiplications by powers of $-q$, see \cref{rem:MultByMinQ}).
To give some more details, let ${\cU\cV[i,i]}$ denote the \redpc of $uv[i-1,i-1]$ given by composing the appropriate suffix of $\mathcal{U}$ with a prefix of $\mathcal{V}$.
 
    The crucial observation is that each marking $K^{(\mu, \ell)}_{r}$,  $T^{(\mu)}_{\ell}$ appears at most once outside of an exponent (\ie as ``mantissa'') multiplied by a suitable power of $q$ in the sum in \cref{lem: FormulaAltW} or the formulas \cref{lem: TableCond}. 
     Note that, since we have the markings $T_\ell^{(\mu)}$, we do not use the $X^{(\mu, \ell)}_{r}$ markings outside of an exponent. 
     Moreover, by \cref{lem: powtwo}, multiplying by a power of $q$ does not increase the size of the support of a marking. This together with the previous observation, by \cref{lem: add} and \cref{lem:AddMult}, implies that $\abs{\sigma(Y_{i})}\leq \msumkt{\cU\cV[i,i]}$.

   %  Let $q^{x_i}\cdot \eps(L^{(j)}_i) $, $1\leq i \leq \ell_j$ (with $\ell_j \leq \len{\Theta}{uv}$) be the multiplications in the $j$-th layer of \powertwo. Then, $\sum_{1\leq i \leq \ell_j} \abssmall{\sigma(L^{(j)}_i)} \leq   \msumkt{\mathcal{Y}}$.

  From this, in turn, we conclude by \cref{lem: powtwo} that there is a constant $c$ such that we insert at most $c \cdot\msumkt{\cU\cV[i,i]}$ new chains during all layers of \powertwo used in Step 1. 
    Because the length of words in $\Thepr$ does not increase during \BSpq-Britton-reductions (\cref{lem:BRinBSinTC}), each sum has length at most $\len{\Theta}{\cU\cV[i,i]}$. Thus, by the \OpLem we can take $c$ such that we can construct $Y_{i}$ in a reduced power circuit $(\Gamma_i, \delta_i)$ with
\begin{align} \label{SizeGammaOne} 
   \abs{\sigma(Y_{i})}&\leq \msumkt{\cU\cV[i,i]}\nonumber\\[0.7mm] 
     % \sum_{\ell=0}^{i-1}\left(\msumk{\mathcal{W}^{(u)}_\ell}+\msumk{\mathcal{W}^{(v)}_\ell}+\abssmall{\sigma(T^{(u)}_\ell)}+\abssmall{\sigma(T^{(v)}_\ell)} \right) \\ \nonumber
   \abs{\mathcal{C}_{\Gamma_i}}
   %&\leq  \abs{\mathcal{C}_{\Gamma}}+ \sum_{\ell=0}^{i-1}\left(\msumk{\mathcal{W}^{(u)}_\ell}+\msumk{\mathcal{W}^{(v)}_\ell}+\abs{\sigma(T^{(u)}_\ell)}+\abs{\sigma(T^{(v)}_\ell)} \right)\\ \nonumber
   & \leq \abs{\mathcal{C}_{\Gamma}}+c\cdot (\msumkt{{\cU\cV[i,i]}}) \\[0.7mm] \nonumber
     \abs{\Gamma_i} &\leq \abs{\Gamma}+c \cdot\left(\msumkt{\cU\cV[i,i]}+ \abs{\mathcal{C}_{\Gamma}} \right)\cdot \left(\ceil{\log(\len{\Theta}{\cU\cV[i,i]})}+1\right).
\end{align}

\paragraph{Step 2.}

Now, we assume that  $uv[i-1,i-1] \eqBGpq\alpha_i^{y_{i-1}}$ for $\alpha_i \in \{a,t\}$ and $y_{i-1} \in \mathbb{Z}$ and aim for computing the bit $\bit_i$ indicating whether $uv[i,i] \inBGpq \BS{p}{pq}$ (if the assumption is not met, the bit can be arbitrary).
In order to do so, we want to apply the middle column in the table in \cref{lem: TableCond}.
 For this, we need to decide whether either $w_i\alpha_i^{y_{i-1}}w'_i \inBGpq  \gen{t}$ or $w_i\alpha_i^{y_{i-1}}w'_i \inBGpq  \gen{a}$ depending on $\beta_{i+1}$.

In order to decide whether $w_i\alpha_i^{y_{i-1}}w'_i \inBGpq \gen{a}$, we need to compute the \BSpq-Britton-reduced \redpc for $w_i\alpha_i^{y_{i-1}}w'_i$, which can be done using \cref{lem:BRinBSinTC}. Be aware that \cref{lem:BRinBSinTC} only allows us to compute the \BSpq-Britton-reduction of the concatenation of two already \BSpq-Britton-reduced \redpc{}s, while  $w_i\alpha_i^{y_{i-1}}w'_i$ is the concatenation of \emph{three} \BSpq-Britton-reduced \redpc{}s. This problem can be easily circumvented by  applying the lemma first to $w_i\alpha_i^{y_{i-1}}$ and then to the concatenation of the outcome and $w_i'$. Finally, one needs to check whether the outcome contains some non-trivial $t$-power.

In order to decide if $w_i\alpha_i^{y_{i-1}}w'_i \inBSpr \gen{t}$, we check if $w_i\alpha_i^{y_{i-1}}w'_i t^{-\hght{w_i\alpha_i^{y_{i-1}}w'_i}} =_{\BSpq} 1$. As above, this can be done by using \cref{lem:BRinBSinTC} three times (as $w_i\alpha_i^{y_{i-1}}w'_i t^h$ consists of four \BSpq-Britton-reduced words) and then checking whether the outcome is the empty word.

During this whole process we can use separate reduced power circuits for each $i$ (as we do for calculating the $Y_i$). We discard them after the checking is done, only remembering the $\lambda_i$. So, this checking step is possible in  $\uTC{0}$ by the \OpLem and \cref{lem:BRinBSinTC}.  
In summary, this gives us Boolean values indicating whether $uv[i-1,i-1] \inBGpq \BS{p}{pq}$ implies $uv[i,i] \inBGpq \BS{p}{pq}$.

\paragraph{Step 3.}
 Now, we only have to find the maximal $i_0$ such that for all $i \leq i_0$ this implication is true. Since $uv[\mOne,\mOne] = 1 \in \BS{p}{pq}$, it follows inductively that  $uv[i,i] \inBGpq \BS{p}{pq}$ for all $i\leq i_0$. Moreover, as the implication $uv[i_0 ,i_0 ]  \inBGpq \BS{p}{pq} \implies uv[i_0 +1, i_0 +1]  \inBGpq \BS{p}{pq}$ fails, we have $uv[j,j] \not\inBGpq \BS{p}{pq}$ for $j \geq i_{0}+1$.
\medskip

Finally, using the same procedure as we did in Step 2 above, we construct a \redpc $\mathcal{W}^{(z)}_{i_0+1}$ of a $\BSpq$-Britton-reduced word $z_{i_0 +1} \in \Thepr$ such that $z_{i_0+1}\eqBSpr w_{i_0+1}\alpha^{\eps(Y_{i_0})}w'_{i_0+1}$.
 We use \cref{lem:BRinBSinTC} to compute \BSpq-Britton-reductions, which guarantees that $\len{\Theta}{z_{i_0+1}} \leq \len{\Theta}{w_{i_0+1}\alpha^{\eps(Y_{i_0})}w'_{i_0+1}}$.
We compute the new marking $T^{(z)}_{i_{0}+1}$  by  $\eps(T^{(z)}_{i_{0}+1})=\eps(T^{(u)}_{i_0+1})+\eps(T^{(v)}_{i_{0}+1})+ \eps(Y_{i_0})$ if $\alpha = t$ and $\eps(T^{(z)}_{i_{0}+1})=\eps(T^{(u)}_{i_{0}+1})+\eps(T^{(v)}_{i_{0}+1})$ if $\alpha = a$.
Then the word
\begin{align*}
    z=\beta_n w_{n}\cdots \beta_{i_0+2} z_{i_0+1}\beta'_{i_0+2} \cdots w'_{n} \beta'_n
\end{align*}
is Britton-reduced with $z=_{\BGpq} uv$.
We only keep the power circuit $(\Gamma_{i_0}, \delta_{i_0})$ from Step 1, which represents $y_{i_0}$ and work on it to compute the \redpc for $z_{i_0+1}$.

To obtain a \redpc $\mathcal{Z}=(\mathcal{W}^{(z)}_\ell,B_\ell^{(z)}, T_\ell^{(z)})_{\ell \in \interval{1}{\len{\Delta}{z}}}$  of $z$, in the \redpc for $uv$ we remove all markings used to represent $w_{i_0+1}uv[i_0,i_0]w'_{i_0+1}$ and use the \redpc of $z_{i_0+1}$ we just constructed instead (and change the indices appropriately).

\paragraph{Complexity of Step 3.}
 We have shown (see Step 1 and 2)  that the \redpc for $ z_{i_0+1}$ can be constructed using constantly many layers of \addition and \powertwo (including multiplications by powers of $-q$, see \cref{rem:MultByMinQ}).

We have the size bound \eqref{SizeGammaOne} for $(\Gamma_{i_0}, \delta_{i_0})$ from Step 1. Let $(\Gamma',\delta')$ denote the power circuit resulting from Step 3. By \cref{lem:BRinBSinTC}  and \eqref{SizeGammaOne} we obtain the following bound (for some suitable constant $c$ different from the one in \eqref{SizeGammaOne}):
  \begin{align*}
 \abs{\Gamma'} &\leq \abs{\Gamma}+c\cdot\left( \mathcal{S}+\abs{\mathcal{C}_{\Gamma}}\right)\cdot \ceil{\log(\mathcal{M})}, \\
\abs{\mathcal{C}_{\Gamma'}}&\leq \abs{\mathcal{C}_{\Gamma}}+c \cdot \mathcal{S}. 
 \end{align*}

To see the bound  $\msumkt{\mathcal{Z}} \leq \msumkt{\mathcal{U}}+\msumkt{\mathcal{V}}$, first consider the case that $\alpha = t$.
Then we have
 \begin{align*}
	\abssmall{\sigma(T^{(z)}_{i_{0}+1})} &\leq  \abssmall{\sigma(T^{(u)}_{i_0+1})} + \abssmall{\sigma(T^{(v)}_{i_0+1})} + \abssmall{\sigma(Y_{i_0})}\\
	&\leq \abssmall{\sigma(T^{(u)}_{i_0+1})} + \abssmall{\sigma(T^{(v)}_{i_0+1})} + \msumkt{\cU\cV[i_0,i_0]}\tag{by \eqref{SizeGammaOne}}
	\shortintertext{and}
	\msumk{\mathcal{W}^{(z)}_{i_0+1}} &\leq \msumk{\mathcal{W}^{(u)}_{i_0+1}} + \msumk{\mathcal{W}^{(v)}_{i_0+1}} \tag{by \cref{lem:BRinBSinTC}}
\end{align*}
yielding that
 \begin{align*}
		\msumk{\mathcal{W}^{(z)}_{i_0+1}} + 	\abssmall{\sigma(T^{(z)}_{i_{0}+1})}&\leq \msumkt{\cU\cV[i_0 + 1,i_0 + 1]}
\end{align*}
As we obtain $z$ by replacing $w_{i_0+1}uv[i_0,i_0]w_{i_0+1}'$ by $z_{i_0+1}$, the bound  $\msumkt{\mathcal{Z}} \leq \msumkt{\mathcal{U}} + \msumkt{\mathcal{V}}$ follows.

The case $\alpha = a$ can be seen similarly using the fact that $Y_{i_0}$ is then only used in the computation of $\msumk{\mathcal{W}^{(z)}_{i_0+1}}$ but not $T^{(z)}_{i_{0}+1}$.

 Finally, we have $ \len{\Theta}{\mathcal{Z}^{(i)}} \leq \len{\Theta}{\mathcal{U}^{(i)}}+\len{\Theta}{\mathcal{V}^{(i)}}$ because $w_{i_0+1}uv[i_0,i_0]w_{i_0+1}'$ is replaced by the \BSpq-Britton-reduction of $w_{i_0+1}\alpha^{\eps(Y_{i_0})}w'_{i_0+1}$, which by \cref{lem:BRinBSinTC} is not longer.
\qed\end{proof}
\begin{theorem}\label{thm:TC1PC} The following is in $\uTC{1}$:
	\compproblem{A \redpc  $\mathcal{Z}=(\mathcal{W}_i,  B_i, T_i)_{i \in \interval{1}{\len{\Delta}{z}}}$ for a word $z \in \Delta^{*}$  over $(\Gamma, \delta)$. }{A \redpc $\cZ^{(r)}$ for a Britton-reduced word $z_{red} \in \Delta^{*}$ over  $(\Gamma^{(r)}, \delta^{(r)})$ such that $z_{red}=_{\BGpq}z$ and there is a constant $c$ such that
$\abssmall{\Gamma^{(r)}} \leq \abs{\Gamma}+c\cdot \len{\Theta}{\mathcal{Z}} \cdot \log(\len{\Theta}{\mathcal{Z}})^3 \cdot \abs{\Gamma}.$
} 
\end{theorem}
\begin{proof}
	Let $z=z_1 \cdots z_n$ with $z_i \in \Delta \cup \{1\}$ (so we can use $1$ as padding symbol). We proceed in a tree-shape manner: we Britton-reduce all factors of length two, then all factors of length four of $z$ (as word over $\Delta$) and so on using \cref{lem:Bredstep}. Observe that this approach is similar to \cite[Theorem 5.10]{MattesWCC23}. 

W.l.o.g. we assume that $n$ is a power of two and write $z_i=z^{(0)}_i$. Observe that by definition, $\mathcal{Z}^{(0)}_i$ is a \redpc of a $\BSpq$-Britton-reduced word  representing  $z^{(0)}_i$ if $z_i \in \BSpq$. \\

We assume that we already have constructed \redpcs  $\cZ^{(k)}_{j}$ of Britton-reduced words $z^{(k)}_{j} \in \Delta^*$ for $j \in \interval{1}{n/2^k}$ all over $(\Gamma^{(k)}, \delta^{(k)})$  such that \\ $\sum_{j=1}^{n/2^k}\msumkt{\cZ_j^{(k)}} \leq \msumkt{\cZ}$. %After proving the size constraints below, 
Using this, by \cref{lem:Bredstep} we can construct in $\uTC{0}$ \redpcs  $\cZ^{(k+1)}_{j}$ of Britton-reduced words $z_j^{(k+1)}=_{\BGpq}z_{2j-1}^{(k)}z_{2j}^{(k)}$ all over the same power circuit  $(\Gamma^{(k+1)}, \delta^{(k+1)})$,  with 
\begin{align}
\label{AbschMskt}	\sum_{j=1}^{n/2^{k+1}}\msumkt{\cZ_j^{(k+1)}}& \leq \sum_{j=1}^{n/2^{k+1}}(\msumkt{\cZ_{2j-1}^{(k)}}+\msumkt{\cZ_{2j}^{(k)}}) \leq\msumkt{\cZ}
	\shortintertext{and} 
\label{AbschLeng}	\sum_{j=1}^{n/2^{k+1}}\len{\Theta}{\cZ^{(k+1)}_j}&\leq \sum_{j=1}^{n/2^{k+1}}(\len{\Theta}{\cZ_{2j-1}^{(k)}}+\len{\Theta}{\cZ_{2j}^{(k)}}) \leq \len{\Theta}{\cZ}.
\end{align}
Moreover, by (\ref{AbschMskt}), 
\cref{lem:Bredstep} implies that in step $k$ (constructing the \redpc over the power circuit $(\Gamma^{(k)}, \delta^{(k)})$) we introduce at most $c \cdot  \msumkt{\cZ}$ new chains. Because of (\ref{AbschLeng}) each sum has length at most $\len{\Theta}{\cZ}$. Thus, in step $k$ we introduce at most $k \cdot c \cdot 
\msumkt{\cZ}\cdot \left(\ceil{\log(\len{\Theta}{\cZ})}+1\right)$ new nodes. In both cases, $c$ is a suitable constant. This implies that there is a constant $c$ such that
\begin{align*}
	\abs{C_{\Gamma^{(k)}}}
&	 \leq  \abs{C_{\Gamma}}+k \cdot c \cdot 
	\msumkt{\cZ}, \\[1mm]
	\abssmall{\Gamma^{(k)}}
	& \leq \abs{\Gamma}+k\cdot c\cdot\left( \abs{C_{\Gamma}}+k \cdot c \cdot 
	\msumkt{\cZ}\right)\cdot \left(\ceil{\log(\len{\Theta}{\cZ})}+1\right). 
\end{align*}
Observe that $\msumkt{\cZ}\leq \len{\Theta}{\cZ}\cdot \abs{\Gamma}$ (if $z_i \in \BSpq$ and $z_i \in \gen{a}$, then the markings $X_i$ and $T_i$ are the empty marking). 
Because we only construct power circuits $(\Gamma^{(k)}, \delta^{(k)})$ for $k \leq \log(\len{\Delta}{z})$ and $\len{\Delta}{z} \leq \len{\Theta}{\cZ}$,  for some large enough $c'$ we have that
\begin{align*}
\abssmall{\Gamma^{(k)}} 
&\leq  \abssmall{\Gamma}+c'\cdot \len{\Theta}{\cZ}\cdot \log(\len{\Theta}{\cZ})^3 \cdot \abs{\Gamma}
\end{align*} 
 for all occurring power circuits $(\Gamma^{(k)}, \delta^{(k)})$, including $(\Gamma^{(r)}, \delta^{(r)})$. 
This, firstly, shows the bound on $\abssmall{\Gamma^{(r)}}$ stated in the theorem and, secondly, establishes that the inputs of all subsequent stages of the circuit  from \cref{lem:Bredstep} are of size polynomial in the input. So, each step is in $\uTC{0}$. Therefore, as the whole algorithm consists of $\ceil{\log(\len{\Delta}{z})}$ many $\uTC{0}$ steps, it is in $\uTC{1}$.
\qed\end{proof}

Recall that $\Sigma=\{a, a^{-1},b,b^{-1},t,t^{-1}\}$.
\begin{corollary} \label{cor:TC1} 
The following is in $\uTC{1} $: 
\compproblem{A word $z \in \Sigma^*$. }{A \redpc $\cZ^{(r)}$ for a Britton-reduced word $z_{red} \in \Delta^*$ over  $(\Gamma, \delta)$ such that $z_{red}=_{\BGpq}z$ and there is a constant $c$ such that
\[\abssmall{\Gamma} \leq c\cdot\len{\Sigma}{z} \cdot \log(\len{\Sigma}{z} )^3 .\]  }
\end{corollary}
\begin{proof}
	Let $\len{\Sigma}{z}=n$ and $z=z_1\cdots z_n$ with $z_i \in \Sigma$. Let $(\Gamma, \delta)$ be the reduced power circuit consisting of exactly one node.  
	%For $j \in \interval{1}{n}$ we set $w^{(0)}_{j} = w_{j} $.
	We construct a \redpc $\cZ = (\cW_j, B_j,T_j)_{
		j\in \interval{1}{n}}$ of $z$ in a straightforward way. Here, $\cW_j$ is just the tuple $(X^{(j)},K^{(j)})$ for two markings $X^{(j)},K^{(j)}$.
	 
	If $z_{j} = \beta$ with $\beta \in \oneset{b, b^{-1}}$,  we set  $ B_j = \beta$ and $T_j=X^{(j)}= K^{(j)}=\emptyset$.  
	If $z_j \in \oneset{a,a^{-1}}$ we define $K^{(j)}$ on $\Gamma$ such that $\eps(K^{(j)})=\pm 1$ and set $X^{(j)}=T_j=\emptyset$ and $B_j=\$$. 
	If $z_j \in \oneset{t,t^{-1}}$ we define $X^{(j)}$ and $T_j$ on $\Gamma$ such that $\eps(X^{(j)})=\eps(T_j)=\pm 1$ and we set $K^{(j)}=\emptyset$ and $B_j=\$$.
	The corollary follows using \cref{thm:TC1PC}.
\qed\end{proof}

Now we are ready to prove Theorem~\ref{thm:WPinBG}. 

\begin{corollary}[Theorem~\ref{thm:WPinBG}]
	For every $p,q \in \Z$ with  $\abs{p}, \abs{q} \geq 1$ the word problem of the Baumslag group $\BGpq$ is in $\uTC{1}$.
\end{corollary}

\begin{proof} 
On input of a word $z \in \Sigma^*$ we  construct the \redpc of a Britton-reduced word $z_{red}\in \Delta^*$ such that $z_{red} \eqBGpq z$. By \cref{cor:TC1} this is possible in $\uTC{1}$. 
By 	\cref{lem: BRinGpr}, for Britton-reduced words $z_{red} \in \Delta^*$ it is straightforward to check if $z_{red}\eqBGpq 1$ (simply check if $z_{red}$ is the empty word). 

Note that if $\abs{q}=1$, we have $ta^pt^{-1}=a^{\sgn{q}\cdot p }$. Thus, we do not obtain huge exponents as in \cref{ex:blowup} when applying Britton-reductions. Moreover, we do not need multiplications by powers of $r$ for any $r \geq 2$,  but we only need to calculate modulo for finding the sign and addition.  Hence, we do not need to use power circuits or we can just use power circuits to base $2$.    
\qed\end{proof}

\section{Conjugacy in $\BGpq$}

Let $u\in \Delpr$ as above. A word $v\in \Delpr$ is called a \emph{cylic permutation} of $u$ if we can write $u=xy$ and $v=yx$ for some $x,y \in \Delpr$.
A word $u\in \Delpr$ is called \emph{cyclically Britton-reduced} if all its cyclic permutations are Britton reduced (equivalently, if $u\in \BS{p}{r}$ or $uu$ is Britton-reduced). For $g,h \in \Delpr$ and $A \sse \BGpr$ we write $g\sim_{A}h$ if  they are conjugate by some element of $A$, i.e. there exists some $z \in A$ with $g \eqBGpr z^{-1} h z$.

\begin{lemma}[\normalfont Collins’ Lemma for \BGpr\bfseries \ {\cite[Theorem IV.2.5]{LS01}}] \label{lem:collins} 
	Let 
	\begin{align}\label{UVConj}
		u =  \beta_{1} \gamma_1 \cdots \beta_{h}\gamma_h,  \qquad v = \beta'_{1} \gamma'_1 \cdots  \beta'_{\ell}\gamma'_\ell 
	\end{align}
	be cyclically Britton-reduced words in $\BGpr$ with $h,\ell \geq 1$, $\beta_i,  \beta'_i\in \{b,b^{-1}\}$ and $ \gamma_i , \gamma'_i \in \BS{p}{r}$. Then $u\simBG v$ if and only if $h= \ell $ and there is a cyclic permutation $v'$ of $v$ and $ u \sim_{\gen{t} \cup \gen{a}} v'$. In particular, $ u \sim_{\gen{a}} v'$ if $\beta_1=b^{-1}$ and $ u \sim_{\gen{t}} v'$ if $\beta_1=b$.
\end{lemma}
Note that, by \hyperref[lem:Britton]{Britton's Lemma}, $ u \sim_{\gen{t} \cup \gen{a}} v$ implies that  $ \beta_{i} = \beta'_{i}$ for all $i \in \interval{1}{\ell}$. Thus, if $u\simBG v$, then $\beta_{1} \cdots \beta_{h}$ is a cyclic permutation of $\beta'_{1} \cdots  \beta'_{\ell}$. This holds for arbitrary $p$ and $r$ as Collins Lemma is stated for HNN-extensions in \cite{LS01}.

\begin{lemma}[\cite{LS01}, Lemma IV.2.3]\label{lem:CondXEqY}
Let $u=\beta_1\gamma_1\cdots \beta_k \gamma_k$ and $v=\beta'_1\gamma'_1\cdots \beta'_k \gamma'_{k'}$ be Britton-reduced words in $\Delpr$ as in (\ref{UVConj}) such that $k\geq 1$ and $u=_{\BGpr}v$. Then:
\begin{align*} %\label{CondXEqY}
&k=k',~ (\beta_1, \ldots, \beta_k)=(\beta'_1, \ldots, \beta'_k) \text{ and  } \\ &\nonumber \gamma'_1 \inBGpr  \begin{cases}
\gamma_1 \gen{t} & \text{ if } \beta_2=b\\
\gamma_1 \gen{a} &\text{ if }  \beta_2=b^{-1}.
\end{cases}
\end{align*} 
\end{lemma}

From now on, we again consider the group $\BGpq$ (i.e. $p$ divides $r$). The following proposition is our crucial tool to decide, together with \cref{lem:collins}, if two cyclically Britton-reduced words are conjugate.   The proof  follows the idea of the proof of \cite[Theorem 3]{DiekertMW2016SI}. 

\begin{proposition}\label{prop: CondForConj}
Let $u =  b^{-1} \gamma_1 \beta_{2} \cdots\beta_{n}\gamma_n$, $v =b^{-1}  \gamma'_1   \beta_{2} \cdots  \beta_{n} \gamma'_n $ be cyclically Britton-reduced words in $\Delpr$ as in (\ref{UVConj}).  Assume that there exists $k \in \Z$ with $a^kua^{-k}\eqBGpq v$.  Then we are in one of the following cases: 
\begin{enumerate}[(i)]
\item $n=1$ and $k=\hght{\gamma'_1}-\hght{\gamma_1}$,
\item $n \geq 2$, $\beta_2=b$ and $k=\log_{q}\left(\frac{\numba{\gamma_1'}}{\numba{\gamma_1}}\right)$ (in this case, $\frac{\numba{\gamma_1'}}{\numba{\gamma_1}}$ is a power of $q$),
\item $n \geq 2$,  $\beta_2=b^{-1}$  and $k=\hght{\gamma'_1}-\hght{\gamma_1}$.
\end{enumerate}
\end{proposition}
\begin{proof}
Let $k \in \Z$ be such that $a^kua^{-k}\eqBGpq v$. We consider three different cases:\\

\textbf{Case $n=1$:} In this case $u=b^{-1}\gamma_1$ and $v=b^{-1}\gamma'_1$ with $\gamma_1, \gamma'_1 \in \BS{p}{pq}$. So: 
\begin{align*}
&&a^kb^{-1}\gamma_1a^{-k}&\eqBGpq b^{-1}\gamma'_1&&\\
 &\iff& ba^kb^{-1}\gamma_1&\eqBGpq \gamma'_1a^k &&\\
&\iff& t^k\gamma_1&\eqBGpq\gamma'_1a^k&&
\end{align*} 
Since $\hght{}$ does not change under ($\BSpq$)-Britton-reductions we obtain that $\hght{t^k\gamma_1 }=\hght{\gamma_1'a^{k}}$. Thus, $k+\hght{\gamma_1 }=\hght{\gamma_1'}$. 
So, $k=\hght{\gamma'_1}-\hght{\gamma_1}$. 
\medskip

\textbf{Case $n \geq 2$ and $\beta_2=b$:} So, $u=b^{-1}\gamma_1b u'$ and $v=b^{-1}\gamma'_1b v'$ with  $u',v' \in \Delta^*$ and $\gamma_1, \gamma'_1 \in \BS{p}{pq}$. Because $u$ is Britton-reduced by assumption, $\gamma_1 \not\inBSpr \gen{t}$.
By \cref{lem:CondXEqY}, whenever $w \in \Delta^*$ with $w \eqBGpq a^kua^{-k}$ is Britton-reduced and does not start with a power of $a$, then there exists $m \in \mathbb{Z}$ and $w' \in \Delta^*$ such that $w= b^{-1}\ [t^k\gamma_1 t^m] bw'$. 
Thus, it follows that $\gamma'_1 \eqBSpr t^k \gamma_1 t^m$ for some $m \in \mathbb{Z}$.
 By \cref{lem: ConcOfwv},
\begin{align} \label{Gamma1DivQ}
\numba{\gamma'_1}=\numba{t^k \gamma_1 t^m}=\numba{t^k}+q^{k}\cdot \numba{\gamma_1 t^m} = q^k \cdot \numba{\gamma_1}.
\end{align}
So $q^{k}=\frac{\numba{\gamma'_1}}{\numba{\gamma_1}} $. Thus, $k=\log_{q}\left(\frac{\numba{\gamma'_1}}{\numba{\gamma_1}} \right)$. 

\medskip

\textbf{Case $n \geq 2$ and $\beta_2=b^{-1}$: }
 So, $u=b^{-1}\gamma_1b^{-1} u'$ and $v=b^{-1}\gamma'_1b^{-1} v'$ with $u',v' \in \Delta^*$.
 By \cref{lem:CondXEqY}, if $w \in  \Delta^*$ with $w \eqBGpq a^kua^{-k}$ is Britton-reduced and does not start with a power of $a$, then there exists $m \in \mathbb{Z}$ and $w' \in \Delta^*$ such that $w= b^{-1}[t^k\gamma_1 a^m] b^{-1}w'$.

Therefore, there exists $m \in \mathbb{Z}$ such that $\gamma'_1\eqBSpr t^k\gamma_1a^m$. Thus, $\hght{\gamma'_1}=k+\hght{\gamma_1}$ and $k=\hght{\gamma'_1}-\hght{\gamma_1}$.\qed
 
\end{proof}

\begin{lemma}\label{lem:cyclicallyReduce}
	There exists a constant $c$ such that the following is in  $\uTC{0}$: 
	\compproblem{A \redpc $\cZ$   of a Britton-reduced word $z \in \Delpr$  over $(\Gamma, \delta)$. }{
		A \redpc  $\cZ'$ of a cyclically Britton-reduced word $z_{cyc}$ over  $(\Gamma', \delta')$ with   $z_{cyc}\sim_{\BGpq} z$ such that $(\Gamma, \delta) \leq (\Gamma', \delta')$ and $\abs{\Gamma'} \leq \abs{\Gamma}+c\cdot \len{\Theta}{\mathcal{Z}} \cdot \log(\len{\Theta}{\mathcal{Z}}) \cdot \abs{\Gamma}$.
 	} 
\end{lemma} 
Using \cref{thm:TC1PC} and \cref{lem:cyclicallyReduce}, on  input of a \redpc $\cZ$ of an arbitrary  $z\in \Delpr$ we can compute a cyclically Britton-reduced word $z_{cyc}$ such that $z_{cyc} \sim_{\BGpq} z$ in $\uTC{1}$.
\begin{proof}
	It is a standard fact from group theory (see \eg \cite[Lemma 25]{Weiss16}), that we can compute $z_{cyc}$ by performing one cyclic permutation to $z$ (cutting right through the middle) and applying Britton reductions. Thus, we write $z= uv$ with $0\leq \absbeta{u} - \absbeta{v}\leq 1$ and apply \cref{lem:Bredstep} to the corresponding \redpc Thus, the lemma follows.
\qed\end{proof}

\begin{theorem}\label{thm:CPgeneric}
	The following is in $\uTC{1} $: 
\vspace*{-.1mm}\compproblem{ \Redpc's for words $u,v \in \Delpr$. }{Is $u$ conjugate to some element in $\BS{p}{pq}$? If no, is $u\sim_{\BGpq} v$?
} 
\noindent In particular, the conjugacy problem for $\BGpq$ is strongly generically in $\uTC{1} $.	
\end{theorem}

\begin{proof}
	By  \cref{cor:TC1} and \cref{lem:cyclicallyReduce}  we can compute \redpc's  of cyclically Britton-reduced words  $\tilde u$ and $\tilde v$ in a power circuit $(\Gamma', \delta')$ with 
    for a suitable constant $c$ as follows: 
	We first construct \redpcs
	for Britton-reduced $\hat u, \hat v$ with $\hat u \eqBGpq u$ and $\hat v \eqBGpq v$ in $\uTC{1}$ using \cref{cor:TC1}. Then, by \cref{lem:cyclicallyReduce} we can Britton-reduce them cyclically yielding \redpcs of cyclically Britton-reduced words  $\tilde u$ and $\tilde v$ with $\tilde u \eqBGpq u$ and $\tilde v \eqBGpq v$. Now, by Collins' Lemma, $u$ is conjugate to some element of $\BS{p}{pq}$ if and only if $\tilde u $ is a single letter from $\BS{p}{pq}$~-- which can be easily tested in $\uTC{0}$.

	Thus, in the following,  we can assume that $u$ cannot be conjugated into \BS{p}{pq}. Moreover we assume that $u$ and $v$ are cyclically Britton-reduced words.	After possibly replacing them by $u^{-1}$ and $v^{-1}$  or a cyclic permutation of either one, we may assume that the first letters of $u$ and $v$ are both $b^{-1}$ (otherwise, they are certainly not conjugate). Concerning the \redpcs we just need to change the indices appropriately and, if replacing by the inverse, change the sign of the evaluations of the respective markings.  Thus, we can assume that we have $u$ and $v$ as in \cref{prop: CondForConj}. 
	
	We calculate a marking evaluating to $k$  in the respective case of this proposition. According to the size constraints in  \cref{cor:TC1} and \cref{lem:cyclicallyReduce}  the \redpcs of $u,v$ are  in a power circuit of size polynomial in the input, so this is possible in $\uTC{0}$ using \addition and \cref{lem: log}.  Then we check if $a^kua^{-k}\eqBGpq v$ for all cyclic permutations of $v$ in parallel. For this we just need one instance of the word problem of $\BGpq$ which we showed to be in $\uTC{1} $ (\cref{thm:TC1PC}). By  \cref{lem:collins} (Collins Lemma) this suffices to check if $u \sim_{\BGpq} v$. 
\medskip

By the proof of \cite[Theorem 5]{DiekertMW2016SI} (see also \cite[Example 2]{DiekertMW17}) the set of words $u\in \Delpr$ representing elements of $\BGpq$ that cannot be conjugated into $\BS{p}{pq}$ is strongly generic.  Hence, the second part of the theorem follows.
\qed\end{proof}

%%%%%%%%%%%%%%%%%%%%%%%%%%%%%%%%%%%%%%%%%%%%%%%%

\subsection{Conjugating by a fixed element in $\BGq$}\label{sec: ConjFixedEl}

In this section, we assume that $p=1$. We  use the notation of $\BS{1}{q}$ as a semi-direct product we introduced in \cref{sec:BSgroups}. We start with the following Lemma:

\begin{lemma}[\!\,{\cite[Equation (5), Proposition 5 and 6]{DiekertMW2016SI}}]\label{prop:nixH}
	\begin{enumerate}[(a)]
		\item Let $m \geq 1$. Then $(r,m)\sim_{\BSq} (s,n)$ if and only if $m=n$ and there is some $k \in \interval{0}{m-1}$ such that $r\cdot \pcBase^k \equiv s \mod (\pcBase^m - 1)$.
		\item 	Let $r,m \in \Z$, $m \neq 0$. If $(r,m)\not\sim_{\BSq} (0,m)$, then $(r,m)\sim_{\BGq} (s,n) $ if and only if $(r,m)\sim_{\BSq} (s,n).$
		\item Let $m,n \in \Z$. Then 
		$(0,m)\sim_{\BGq} (0,n)$ if and only if $(m,0)\sim_{\BSq} (n,0)$ if and only if  $\exists k \in \Z: m = \pcBase^k n.$
	\end{enumerate}
\end{lemma}

\begin{proposition}\label{prop: ResultConj}
		For every fixed $g \in \BGq$ the following is in $\uTC{1}$: 
		\ynproblem{A word $w \in \Delta^*$.}{Is $g \sim_{\BGq} w$?} 
\end{proposition}

\begin{proof}
	By the semi-direct product representation introduced in  \cref{sec:BSgroups}, we only need to consider the case that $g = t^ka^rt^m \in \BSq$. By conjugating with $t^{-k}$, we may assume that we can write $g=(r,m)$ with $r \in \Z$ and $q$ does not divide $r$. 
	Moreover, using \iffull \cref{cor:TC1} \else \cref{thm:TC1PC} \fi and \cref{lem:cyclicallyReduce} we can (cyclically) Britton-reduce the input word $w$ leading to a \redpc for $w =_{\BGq} t^\ell a^st^n \in \BSq$ (if the cyclic reduction of $w$ is not in $\BSq$, we know that $w$ is not conjugate to $g$). By \cref{thm:TC1PC} this is possible in $\uTC{1}$. Again by conjugating with  $ t^{-\ell}$, we may assume that we can write $w=_{\BGq} (s,n)$ with $s \in \Z$ and $q$ does not divide $s$. In the following, we use the representation of $\BSq$  as semi-direct product.

	Now, first assume that $m \neq 0 $ and $(r,m)\not\sim_{\BSq} (0,m)$. By replacing $g$ and $(s,n)$ by their inverses if necessary, we may assume $m\geq 1$. Hence, by \cref{prop:nixH}, we have $(r,m)\sim_{\BGq} (s,n) $ if and only if $(r,m)\sim_{\BSq} (s,n)$.  By \cref{prop:nixH}, we know that the latter holds if and only if $m=n$ and there is some $k \in \interval{0}{m-1}$ such that $r\cdot \pcBase^k \equiv s \mod (\pcBase^m - 1)$. Now, as $(\pcBase^m - 1)$ is a constant, we can compute $r \cdot \pcBase^k \bmod (\pcBase^m - 1)$ and $s \bmod (\pcBase^m - 1)$ by \cref{lem:moduloPowerOfTwoConstant}. As there are only a constant number of possibilities for $k$ we can check for all of them whether $r\cdot \pcBase^k \equiv s \mod (\pcBase^m - 1)$.

	If	$(r,m)\sim_{\BGq} (0,m)$, we also know that $(r,m) \sim_{\BGq} (m,0)$ (as $b^{-1}(0,m)b =_{\BGq} (m,0)$). Thus, it only remains to consider the case $g= (r,0)$ with $r \in \Z$ (\ie $m=0$). %Moreover, without loss of generality $r > 0$.
	In this case \cref{prop:nixH} tells us that $(s,0)\sim_{\BGq}  (r,0) \iff \exists k \in \Z: s = \pcBase^k r.$ We can check this condition using \cref{lem: floatpOp}. 
	
	Thus, it remains to test whether either $n=0$ or  $(s,n)\sim_{\BGq} (0,n)$ (as otherwise, it is straightforward to see that $(s,n)\not\sim_{\BGq} (r,0)$ for any $r$).
	Checking whether $n=0$ can be done by \cref{lem:compareCompactMarkings}. For the latter test we proceed as follows: 
	
	After possibly replacing $(s,n)$ by $(s,n)^{-1}$ and $(r,m)$ by $(r,m)^{-1}$, we can assume that $n >0$.	
	Now, because of \cite[Proposition 6]{DiekertMW2016SI} we know that, if $n\neq 0 $ and $(s,n)\sim_{\BGq} (r,0)$, then $n = \pcBase^k r$ for some $k \in \Z$ (be aware that compared to the above notation we flipped the coordinates $s$ and $n$); moreover, we can compute this $k$ if it exists (\cref{lem: powtwo}). 
	
	By our assumption $n = \pcBase^k r > 0$, we can use \cite[Equation (5)]{DiekertMW2016SI} stating that $(s,\pcBase^k r)\sim_{\BGq} (0,\pcBase^k r)$ if and only if there is some $x \in \interval{0}{\pcBase^k r-1}$ with $s \cdot \pcBase^x \equiv 0 \mod \pcBase^{\pcBase^k r} - 1$. 
	Recall that  $s \in \Z$ and $\pcBase$ does not divide $s$. As also $\pcBase$ does not divide $\pcBase^{\pcBase^k r} - 1\in \Z$, the only way for satisfying $s \cdot \pcBase^x \equiv 0 \mod \pcBase^{\pcBase^k r} - 1$ is if it is already satisfied for $x=0$. 
	Thus, we need to test whether $s \equiv 0 \mod \pcBase^{\pcBase^k r} - 1$. 
	As $s$ is given as a compact marking $U$ with $\eps(U) = s$, it is sufficient to compute $\eps(P) \bmod  \pcBase^{\pcBase^k r} - 1$ for all $P \in \supp(U)$.
	
	Now, observe that $\pcBase^{\ell }\bmod  \pcBase^{\pcBase^k r} - 1 = \pcBase^{\ell \bmod  \pcBase^k r} $; thus, it suffices to compute $\eps(L) \bmod  \pcBase^k r $ for some arbitrary marking $L$ on our power circuit. This can be done by \cref{lem:moduloPowerOfTwoConstant}.
	\qed
\end{proof}

Now, Theorem~\ref{thm:main_conjugacy} follows from \cref{thm:CPgeneric} and \cref{prop: ResultConj}. As for Theorem~\ref{thm:TC1PC} we do not need power circuits if $\abs{q}=1$ or just can use power circuits to base $2$.

\vspace*{-1mm}\enlargethispage{\baselineskip}
\section*{Conclusion and open questions}

We have shown that the word problem of the  Baumslag groups $\BGpq$ is in $\uTC{1}$ for $\abs{p},\abs{q} \geq 1$. The same complexity applies to the conjugacy problem for elements outside $\BS{p}{pq}$. For a fixed element $g \in \BGq$ we could show that the conjugacy problem is in $\uTC{1}$. For these problems, $\uTC{1}$ seems to be the best possible using the current approach of tree-like Britton reductions. 
Moreover, we show that the power word problem is in $\uAcf{0}$ for the groups $\BS{p}{pq}$ with $\abs{p},\abs{q} \geq 1$ in the special case that the input is a word over powers of the generators. Thus we give a partial answer to the question raised in \cite{LohreyZetsche23} about the complexity of the power word problem in the groups $\BS{p}{r}$.

\subsection*{Open Questions:}
\begin{itemize}
	\item  What is the ``actual'' complexity of the word problem of \BGpq? Are there any better lower bounds other than that it contains a non-abelian free subgroup?
	\item
	What is the complexity of the word problem of $\BGpr$ if $p \nmid r$? For these groups, our approach does not work (see \cref{rem:BSTwoThree}) and it is completely unclear whether the problem can be solved even in polynomial time. 
	\item Can our methods be generalized to the Higman group $H_4$? This is closely related to the growth of its Dehn function, which, to the best of our knowledge, is not known to be in $\tow(\Oh(\log n))$ like for the Baumslag group.
\end{itemize}
%
%

%\bibliography{bibexport}

\newcommand{\Ju}{Ju}\newcommand{\Ph}{Ph}\newcommand{\Th}{Th}\newcommand{\Ch}{Ch}\newcommand{\Yu}{Yu}\newcommand{\Zh}{Zh}\newcommand{\St}{St}\newcommand{\curlybraces}[1]{\{#1\}}

\end{document}